\documentclass{article}

\usepackage{amsthm}
\usepackage{amsmath}
\usepackage{amssymb}
\usepackage{bm}
\usepackage{mathtools}
\usepackage[dvipsnames]{xcolor}

\definecolor{col1}{HTML}{BBBBBB}
\definecolor{col2}{HTML}{009988}
\definecolor{col3}{HTML}{CC3311}
\definecolor{col4}{HTML}{EE3377}
\definecolor{col5}{HTML}{33BBEE}
\definecolor{col6}{HTML}{0077BB}
\definecolor{col7}{HTML}{EE7733}

\usepackage{tikz}
\usetikzlibrary{patterns,positioning}
\usepackage{pgfplots}
\pgfplotsset{compat=newest}
\usepgfplotslibrary{fillbetween,groupplots}

\pgfmathsetmacro{\cmapi}{ln(0+1)}
\pgfmathsetmacro{\cmapii}{ln(1+1)}
\pgfmathsetmacro{\cmapiii}{ln(3+1)}
\pgfmathsetmacro{\cmapiv}{ln(5+1)}
\pgfmathsetmacro{\cmapv}{ln(7+1)}
\pgfmathsetmacro{\cmapvi}{ln(8+1)}

\pgfplotsset{ % A colormap transforming jet via log(1+x)
/pgfplots/colormap={invlog_jet}{
rgb255(\cmapi cm)=(0,0,128);
rgb255(\cmapii cm)=(0,0,255);
rgb255(\cmapiii cm)=(0,255,255);
rgb255(\cmapiv cm)=(255,255,0);
rgb255(\cmapv cm)=(255,0,0);
rgb255(\cmapvi cm)=(128,0,0)}
}

\pgfmathsetmacro{\cmapi}{exp(0)-1}
\pgfmathsetmacro{\cmapii}{exp(0.1)-1}
\pgfmathsetmacro{\cmapiii}{exp(0.3)-1}
\pgfmathsetmacro{\cmapiv}{exp(0.5)-1}
\pgfmathsetmacro{\cmapv}{exp(0.7)-1}
\pgfmathsetmacro{\cmapv}{exp(0.8)-1}

\pgfplotsset{ 
/pgfplots/colormap={log_jet}{
rgb255(\cmapi cm)=(0,0,128);
rgb255(\cmapii cm)=(0,0,255);
rgb255(\cmapiii cm)=(0,255,255);
rgb255(\cmapiv cm)=(255,255,0);
rgb255(\cmapv cm)=(255,0,0);
rgb255(\cmapvi cm)=(128,0,0)}
}

\usepackage{booktabs}
\usepackage{setspace} 
\usepackage{enumerate}
\usepackage{geometry}
\usepackage{datatool,fp}
\usepackage{graphicx}

\usepackage{dsfont}
\usepackage{ifthen}
\usepackage{csquotes} % For better "" control with \enquote
\usepackage{soul} % For better underlines with \ul
\usepackage[hidelinks]{hyperref}
\usepackage[
	backend=biber,
	isbn=false,
	url=false,
	eprint=false,
	sorting=none,
	giveninits=true,
	sortcites=true]{biblatex}
\usepackage{xparse}
\AtEveryBibitem{%
  \clearlist{language} % Avoids printing language for each citation
}
\DeclareSourcemap{
  \maps[datatype=bibtex]{
    \map{
      \step[fieldset=urldate, null] % Apparently necessary to remove "visited on"?
    }
  }
}
\allowdisplaybreaks

\newcommand{\DTLfetchsave}[5]{%
  \edtlgetrowforvalue{#2}{\dtlcolumnindex{#2}{#3}}{#4}%
  \dtlgetentryfromcurrentrow{\dtlcurrentvalue}{\dtlcolumnindex{#2}{#5}}%
  \let#1\dtlcurrentvalue
}

\graphicspath{{graphics/}{graphics/eps/}}

\usepackage{float}
\floatstyle{ruled}
\newfloat{algorithm}{h}{loa}
\floatname{algorithm}{Algorithm}
\usepackage{algpseudocode}
\algnewcommand\AND{\textbf{ and }}
\usepackage{tabularx}
\usepackage{booktabs}
\usepackage{array}
\usepackage{enumitem}
\usepackage{accents}
\usepackage{wrapfig}
\usepackage{xargs}
\usepackage[font=small,skip=3pt]{caption}

\newlength{\figsize}
\usepackage{algorithm}
\usepackage{algpseudocode}
\usepackage{multicol}

\newtheorem{prop}{Proposition}
\newtheorem{rem}[prop]{Remark}
\newtheorem{thm}[prop]{Theorem}
\newtheorem{defn}[prop]{Definition}

\newtheorem{cor}[prop]{Corollary}
\newtheorem{ass}[prop]{Assumption}

\newtheorem{exa}[prop]{Example}

\usepackage{mdframed}
\usepackage{diagbox}
\newmdtheoremenv{framedthm}{Theorem}

\newcommand{\cref}[1]{Corollary \ref{#1}}
\newcommand{\dref}[1]{Definition \ref{#1}}

\newcommand{\pref}[1]{Proposition \ref{#1}}
\newcommand{\tref}[1]{Theorem \ref{#1}}
\newcommand{\asref}[1]{Assumption \ref{#1}}
\newcommand{\rref}[1]{Remark \ref{#1}}
\newcommand{\appref}[1]{Appendix \ref{#1}}

\newcommand{\figref}[1]{Figure \ref{#1}}
\newcommand{\tabref}[1]{Table \ref{#1}}
\renewcommand{\algref}[1]{Algorithm \ref{#1}}
\newcommand{\iref}[1]{\ref{#1}.}
\newcommand{\sref}[1]{Section \ref{#1}}
\newcommand{\ssref}[1]{Subsection \ref{#1}}

\usepackage{tabularx}
\usepackage{booktabs}
\usepackage{array}

\newcommand{\wav}{k}

\newcommand{\Lt}{L^2(\Omega)}

\newcommand{\HtM}{H^2(\M,\C)}

\newcommand{\HtmM}{H^{-2}_0(\M,\C)}

\newcommand{\e}{\epsilon}
\newcommand{\eps}{\epsilon}
\newcommand{\dl}{\delta}
\newcommand{\lm}{\lambda}

\newcommand{\Gm}{\Gamma}

\newcommand{\Om}{\Omega}

\newcommand{\sg}{\sigma}

\newcommand{\oc}{o}%\mathcal{o}}
\newcommand{\Cc}{\mathcal{C}}

\newcommand{\Fc}{\mathcal{F}}

\newcommand{\Jc}{\mathcal{J}}

\newcommand{\Lc}{\mathcal{L}}

\newcommand{\Nc}{\mathcal{N}}
\newcommand{\Oc}{\mathcal{O}}

\newcommand{\Rc}{\mathcal{R}}
\newcommand{\Sc}{\mathcal{S}}

\newcommand{\Jf}{\mathfrak{J}}
\newcommand{\Kf}{\mathfrak{K}}

\newcommand{\C}{\mathbb{C}}

\newcommand{\M}{\mathbb{M}}
\newcommand{\N}{\mathbb{N}}

\newcommand{\R}{\mathbb{R}}

\newcommand{\maththis}[1]{\mathop{\mathrm{#1}}}
\newcommand{\mathsthis}[1]{{\maththis{#1}\nolimits}}

\newcommand{\mpst}{{m_\mathsthis{post}}}

\newcommand{\Cp}{{\Cc_0}}
\newcommand{\Cpi}{\Cc_0^{-1}}

\newcommand{\Cph}{{\Cp^{\mkern-8mu 1/2}}}

\newcommand{\Cpst}{{\Cc_\mathsthis{post}}}

\newcommand{\Gmn}{\Gm_\mathsthis{noise}}
\newcommand{\Gmni}{\Gmn^{-1}}
\newcommand{\Gmnih}{\Gmn^{-1/2}}

\newcommandx{\consK}[2][1=p,2=m_0]{K_{#2}^{#1}}
\newcommandx{\RcK}[2][1=m_0,2=p]{\Rc_{#1}^{#2}}
\newcommandx{\prob}[2][1=p,2=m_0]{\href{OEDp}{(\mathrm{OED}_{#2}^{#1})}}

\newcommand{\umo}{{\underline{m_0}}}
\newcommand{\lmo}{{\overline{m_0}}}

\newcommand{\Fb}{\mathbf{F}}

\newcommand{\Mb}{\mathbf{M}}
\newcommand{\Mbh}{\Mb^{1/2}}
\newcommand{\Mbhi}{\Mb^{-1/2}}
\newcommand{\Mbco}{\mathbf{M}_\mathsthis{co}}

\newcommand{\what}[1]{\widehat{#1}}
\newcommand{\Cbh}{C}

\newcommand{\Cbt}{\what{\Cb}}

\newcommand{\Cbth}{\Cbt^{1/2}}

\newcommand{\argmin}{\mathsthis{argmin}}

\newcommand{\tr}{\mathsthis{tr}}
\newcommand{\rank}{\mathsthis{rank}}

\newcommand{\diag}{\mathsthis{diag}}
\newcommand{\Diag}{\mathsthis{Diag}}
\renewcommand{\d}{\,\mathrm{d}}
\renewcommand{\Re}{\mathsthis{Re}\,}
\renewcommand{\Im}{\mathsthis{Im}\,}

\newcommand{\dom}{\mathsthis{dom}}
\newcommand{\spann}{\mathsthis{span}}

\newcommand{\hada}{\odot}

\newcommand{\dual}[3][X]{\left({#2},{#3}\right)_{{#1}^*,{#1}}}
\newcommand{\duall}[4]{\left({#3},{#4}\right)_{{#1},{#2}}}

\newcommand{\inner}[3][X]{\left\langle {#2},{#3}\right\rangle_{#1}}
\newcommand{\Mass}[3][M]{\left({#2},{#3}\right)_{{#1}}}

\newcommand{\nco}{n_{\maththis{co}}}

\newcommand{\bmid}{\bigm\vert}
\newcommand{\midsp}{\,\,\bmid\,\,}

\newcommand{\msens}{m}
\newcommand{\mobs}{{m_\mathsthis{obs}}}
\newcommand{\tPDE}{t_\mathsthis{PDE}}

\newcommand{\Mobs}[1][\w]{\mathsthis{\mathbf{Diag}}(#1)}
\newcommand{\kobs}{k}
\newcommand{\lobs}{l}

\renewcommand{\Fb}{F}
\renewcommand{\Lc}{L}
\renewcommand{\Cbt}{C}

\renewcommand{\a}{{\bm{a}}}
\newcommand{\w}{{\bm{w}}}
\renewcommand{\v}{{\bm{v}}}
\newcommand{\g}{{\bm{g}}}
\renewcommand{\e}{\bm{e}}
\renewcommand{\eps}{{\bm{\epsilon}}}
\renewcommand{\sg}{\bm{\sigma}}
\newcommand{\z}{\bm{z}}
\newcommand{\x}{\bm{x}}
\newcommand{\s}{\bm{s}}

\newcommand{\oh}{\what{o}}

\newcommand{\Mw}[1][\w]{\,\Diag(#1)}
\newcommand{\gw}[1][\w]{\g}
\newcommand{\wpcont}[1][m_0]{\overline{\w}^{#1}}

\newcommand{\mobssum}{\sum\limits_{\substack{\mathsthis{collapse} \\ \mobs}}}

\newcommand{\rowsum}{\sum_{\mathsthis{rows}}}

\definecolor{GoodGreen}{HTML}{228B22}

\newcommand{\revia}[1]{#1}
\newcommand{\revib}[1]{#1}
\newcommand{\revix}[1]{#1}

\newcommand{\reviasafe}[1]{#1}

\providecommand{\keywords}[1]
{
  \small	
  \textbf{\textit{Keywords---}} #1
}

\title{Global optimality conditions for sensor placement, \\ with extensions to binary \revix{low-rank A-optimal designs}}

\author{
Christian Aarset\footnote{\href{mailto:c.aarset@math.uni-goettingen.de}{c.aarset@math.uni-goettingen.de}}\\[1ex]
\normalsize{University of Göttingen}
}
\date{}

\begin{document}

% Load posterior databases
\DTLloaddb
  {posteriors}% <db name>
  {graphics/posteriors.csv}% <filename>
  
 \DTLloaddb
  {posterior24}% <db name>
  {graphics/posterior_24.csv}% <filename>
  
% Create posterior max values for colorbar scaling
\DTLfetchsave\posterioronemax{posteriors}{design}{ones}{posteriormax}
\DTLfetchsave\posteriorzeromax{posteriors}{design}{zero}{posteriormax}
\DTLfetchsave\posteriortwentyfourmax{posterior24}{design}{24}{posteriormax}
\DTLfetchsave\posteriortwentyfourdiffmax{posterior24}{design}{24diff}{posteriormax}

% These dbs won't be needed further
\DTLdeletedb{posteriors}

% Load pseq database
\DTLloaddb
  {pseq}% <db name>
  {graphics/pseq_24.csv}% <filename>

\maketitle

\vspace{-5mm}

\begin{abstract}
The \emph{sensor placement problem} for stochastic linear inverse problems consists of determining the optimal manner in which sensors can be employed to collect data. Specifically, one wishes to place a limited number of sensors over a large number of candidate locations, quantifying and optimising over the effect this data collection strategy has on the solution of the inverse problem. In this article, we provide a global optimality condition for the sensor placement problem via a subgradient argument, obtaining sufficient and necessary conditions for optimality\revix{, and marking certain sensors as \emph{dominant} or \emph{redundant}, i.e.~always on or always off}. We demonstrate how to take advantage of this optimality criterion to find approximately optimal binary designs, i.e.~designs where no fractions of sensors are placed. Leveraging our optimality criteria, we derive a powerful low-rank formulation of the A-optimal design objective for finite element-discretised function space settings, demonstrating its high computational efficiency, particularly in terms of derivatives, and study globally optimal designs for a Helmholtz-type source problem and extensions towards optimal binary designs.
\end{abstract}

\vspace{-3mm}
\tableofcontents

\pagebreak

\pagenumbering{arabic}
\setcounter{page}{2}

\section{Introduction}

To solve an \emph{inverse problem}, one reconstructs an unknown from indirect measurements of data. It is ubiquitous in real-life experimental settings that the experimenter must control \emph{how} such data is obtained or collected; such choice can significantly affect the quality of any subsequent reconstruction in the inversion process. This effect has driven interest in the field of optimal experimental design (OED), that is, the field of prescribing the best possible parameter-to-observable map in terms of the resulting quality of reconstructions. 

More explicitly, assume some possibly infinite-dimensional quantity of interest $f\in X$ is sought, and that the experimenter has the ability to design a parameter-to-observable map $\Fc_\w: X\to\R^m$, $m\in\N$ fixed, where $\Fc_{\w}$ can be chosen to depend on some \emph{design parameter}, or just \emph{design}, ${\w}$, which can be freely chosen and controlled by the experimenter prior to experimentation.

For any given design $\w$, experimental data may be collected via
\begin{equation}\label{eq:forward}
	\g = \Fc_\w f + \eps \in \R^m
\end{equation}
subject to measurement noise $\eps$. By solving the associated inverse problem, $f$ may be reconstructed from the observable $\g$; we refer to \cite{EngHanNeu1996} for a broad overview. A key question in the field of optimal experimental design is therefore the quantification of the effect of the design $\w$ on the reconstruction of $f$. Explicitly, this leads to a minimisation problem of the form
\begin{equation}\label{eq:OED}
\w^* \in \mathop{\mathrm{argmin}}\limits_{\w} \Jc(\w) + \Rc(\w),
\end{equation}
where the objective functional $\Jc$ specifies the design-dependent quality of the reconstruction $f$, while the penalty functional $\Rc$ enforces feasibility constraints and desirable behaviour of the design $\w$.

\paragraph{Notation}\label{par:notation}

In what follows, $X$ will always denote a separable Hilbert space, and $X^*$ its topological dual space, consisting of all bounded linear functionals $x^*:X\to\R$. For any bounded linear operator $A:X\to X$, its \emph{operator trace}, or simply trace, is for any choice of orthonormal basis $\{e_i\}_{i=1}^\infty\subset X$ given as $\tr(A):=\sum_{i=1}^\infty\inner{Ae_i}{e_i}$ if this sum converges, in which case it is independent of the precise basis choice above \cite[Thm.~3.1]{Sim1979}. \revia{Moreover, $\Mw[\cdot]$ is used to denote the vector-to-matrix operation mapping a vector to the diagonal of an otherwise zero matrix, while $\diag(\cdot)$ is its adjoint matrix-to-vector operation, extracting the diagonal of a matrix as a vector. Throughout, finite-dimensional vectors, particularly those in the measurement space or in the spatial domain, are expressed via boldcase, e.g.~$\w\in\R^m$, while scalars, infinite-dimensional variables and their discretisations are not bolded. }

\paragraph{The sensor placement problem and the A-optimal objective functional}

The choice of objective $\Jc$ depends on the experimental setting and on the experimenter's goals. As an example, we consider the \emph{A-optimal sensor placement problem} for linear Bayesian inverse problems. \revia{This problem shares similarities with the field of compressed sensing for inverse problems \cite{EldKut12, Alb24}, in the sense that one seeks to select the most informative sensors out of a larger set of candidates, albeit with additional emphasis on the spatial location of the candidate sensor locations. In this setting, $\w\in\left\{0,1\right\}^m$ acts as a mask on the data, with $\w_k=1$ corresponding to placing a sensor in the $k$-th out of $m$ candidate locations and observing the $k$-th component of the full data, and $\w_k=0$ corresponding to not placing the $k$-th sensor and so not making this observation. Thus, 
\[
	\Fc_\w = \Mw\Fc:X\to\R^m,
\]
with $\Fc:X\to\R^m$ independent of $\w$, recalling that $\Mw\in\R^{m\times m}$ is the diagonal matrix with the design $\w\in\{0,1\}^m$ on the diagonal. Additionally, let us assume that

\begin{itemize}
\item The parameter-to-observable map $\Fc:X\to\R^m$ is linear and bounded.
\item The noise is Gaussian white noise with distribution $\eps\sim\Nc(0,\Gmn)$ with $\Gmn:=\Mw[\sg]$, $\sg\in\R^m$, $\sg>0$ pointwise, i.e.~the noise is spatially uncorrelated, albeit not necessarily spatially uniform.
\end{itemize}

In this setting, the Bayesian inversion formula \cite[Ex.~6.23]{Stu10}, coupled with \cite[Sec.~2.4, Sec.~3.1]{AttCon2022} justifies that in the diagonal noise covariance case, given data $\g$ and a prior distribution $\Nc(m_0,\Cp)$ of the unknown parameter $f$, the posterior distribution of $f$ can be expressed as $\Nc(\mpst(\w),\Cpst(\w))$, with
\begin{equation}\label{eq:bayesian_inversion}
\begin{aligned}
\mpst(\w) & = m_0 + \Cpst\Fc_\w^*\Gmni\left(\g - \Fc_\w m_0\right) \in X,\\ %\nonumber
\Cpst(\w) & = \left(
\Fc^*\Gmnih\Mw\Gmnih\Fc + \Cpi
\right)^{-1}
\in L(X^*,X).
\end{aligned}
\end{equation}
}

%\footnote{We remark that as we will later want to consider the situation where $\Cp$ inherently has finite rank, it is clear that $\Cpi$ will not generally exist. This motivates us to employ the formulation \eqref{eq:bayesian_inversion}, rather than the more widely employed formulation \cite[p.~536,~(6.13a)]{Stu10}, which is equivalent whenever $\Cpi$ exists. Note that by construction, \emph{every} covariance operator permits the decomposition $\Cp=\hpri\hpria$ for some $\hpri:L^2(\P)\to X$, $\P$ a probability space, as there is some random variable $f:\P\to X$ such that $\Cp = \E[\dual{\cdot}{f}f]$, i.e.~one may choose $\hpri = \E[\cdot f]$.}

A-optimal designs are designs $\w^*$ that minimise the operator trace of the posterior covariance $\Cpst(\w)$, that is, setting $\Jc(\w):=\tr(\Cpst(\w))$. This can be seen as minimising the average uncertainty in the reconstruction, a consequence of Mercer's theorem \cite{Mer1909}. More broadly, A-optimal experimental designs \revib{can be viewed as an infinite-dimensional analogue to the special case $p=-1$ of Kiefer's $\Phi_p$ criteria, which for finite-dimensional positive definite information matrices $I$ can be given as
\begin{equation*}\label{eq:kiefer}
	\Phi_p(I) := \begin{cases}
		\lm_{\max}(I), & p = \infty, \\
		\left(
			\tr\left(
				I^{p}
			\right)		
		\right)^{1/p}, & p\in(-\infty,0)\cup(0,\infty), \\
		\det(I), & p = 0, \\
		\lm_{\min}(I), & p = -\infty,
	\end{cases}
\end{equation*}} 
see \cite{Kie1974, Ahi21}, \revib{and utilising the fact that in our Gaussian linear case, the covariance matrix is the inverse of the information matrix}; $\lm_{\max}$ and $\lm_{\min}$ denote the largest resp.~the smallest eigenvalue. In the above, $p=0$ corresponds to the D-optimal objective, while $p=-\infty$ corresponds to the E-optimal objective; this highlights a large class of interesting, interconnected optimality criteria.

Various other optimality criteria are also considered in the existing OED literature, such as $c$-optimality \cite{Elf1952} and expected information gain; a broader overview can be found in e.g.~\cite{Puk06}. While we will not explicitly address these criteria by ways of example, we will retain sufficient generality in our choice of objective $\Jc$ to allow for future extensions. \revib{In particular, we note that in the case where $\Fc$ is non-linear, the posterior distribution of $f$ may not permit straightforward analysis; we refer to \cite{WuOleCheGha2023, AleNicPet2024, ChoAttAle2024} for some approaches in this direction. Although the present article will not address the non-linear case, the aforementioned generality we will allow in the objective $\Jc$ means extensions also to this case remains an open and interesting question.}\label{sentence:nonlinear}

\paragraph{Penalty functional in the sensor placement problem}

We now discuss the appropriate choice of penalty functional $\Rc$ in the sensor placement problem discussed above discussed above. 

%\dontknow
Generally speaking, it is necessary to add further penalty terms to control the desired properties of the optimal design $\w^*$, e.g.~to prevent a design with $\w^*_k=1$ for all $k\in\N$, $k\leq m$ from being optimal. Frequently, this is achieved by introducing an additive penalty term, e.g.~of the form $\alpha\|\w\|_1$ \revia{or $\alpha\|\w\|_0$} to $\Rc$, where $\alpha>0$; \revia{$\|\cdot\|_0$, despite not being a norm, is often referred to as the \enquote{$0$-norm}, defined as the number of non-zero indices of the input,
\[
	\|\w\|_0 := \#\{k\in\N \mid k\leq m, \, \w_k\neq 0\}.
\]}
Moreover, the assumption $\w\in\left\{0,1\right\}^m$ is relaxed to $\w\in[0,1]^m$, \revia{enabling the use of continuous optimisation techniques when approaching \eqref{eq:OED}.}

Such \enquote{soft constraints}, i.e.~additive penalty terms, demonstrably promote binary or nearly binary designs, but do not offer an intuitive link between the parameter $\alpha$ and the resulting number of active sensors in the optimal design $\w^*$. With this motivation, we instead turn our attention towards a penalisation strategy that naturally and precisely enforces any user-defined number of target sensors.

%Additional terms must generally be added to the penalty functional to prevent a design with $\w^*_k=1$ for all indices $k$ from being optimal. One typically relaxes  To counteract this relaxation, it is natural to also attempt to encourage the binarity $\w^*\in\{0,1\}^m$ of the optimal design via the penalty functional, improving interpretability of the design in terms of physically meaningful sensor placement. 

%One major strategy for such penalisation, which we here will not pursue, is by including sparsity-promoting terms, such as $\alpha\|\w\|_1$, $\alpha>0$, or nearly convex approximations of $\alpha\|\w\|_0$, $\alpha>0$ \cite{AlePetStaGha14, HerAleSai20}. 

\paragraph{Penalty functional in the best $m_0$ sensors placement problem}

As an alternative to sparsity-promoting forms of the penalty functional $\Rc$ as in the above, one may consider the case where one is given a fixed budget of exactly $m_0\in\N$ sensors, $m_0<m$. The \emph{best ($m_0$) sensors placement problem} considered in e.g.~\cite{YuAni21} corresponds to the additional hard constraint that $\|\w\|_0\leq m_0$. Typically, $m$ is rather large; increasing $m$ while keeping $m_0$ fixed can be interpreted as increasing the number of potential sensor locations, while not increasing the budget of available sensors to actually place, allowing for a more fine-tuned sensor grid at the cost of increased complexity in \eqref{eq:OED}. 

\revia{The naive solution of testing every design $\w\in\left\{0,1\right\}^m$ and choosing the best one is of computational complexity at least $O(\binom{m}{m_0})$ and so quickly becomes unfeasible already for moderately large $m$. Therefore,} inspired by the techniques for additive penalty terms, we will also allow $0\leq \w\leq 1$ pointwise and approach the best sensors placement problem as a continuous optimisation techniques. \revib{However, the $0$-norm is discontinuous, as well as non-convex; as such, attempting to apply standard computational tools is highly daunting}. \label{sentence:nonconvex}

For the soft-constrained sensor placement problem above, \cite{AlePetStaGha14} counteracted this difficulty by studying a regularised $0$-norm sparsification approach: By solving with increasingly non-convex approximations of the $0$-norm as additive penalty and utilising continuation to combat local minima, the authors of the cited work were able to obtain binary approximate optimal designs. In this article, we will adapt this approach to the hard-constrained best sensors placement problem and study a class of \emph{$p$-relaxed best sensors placement problems} for the power $p\in[0,1]$, \revia{employing the convention $0^0=0$}. Keeping for now $m$ fixed, this corresponds to a family of constraint sets and penalty functionals
\begin{equation}\label{eq:penalty_target_sensors}
\begin{aligned}
	\consK & := 
	\left\{
		\w \in \R^m \midsp
		0\leq \w\leq 1 \, \wedge \sum_{k=1}^m\w_k^p\leq m_0	
	\right\}, \\%\qquad 
	\RcK(\w) & := 
	\begin{cases}
		0, & \w\in \consK, \\
		\infty, & \text{else,}
	\end{cases}
\end{aligned}
\end{equation}
under which \eqref{eq:OED} is equivalent to the $p$-relaxed best sensor placement problem
\begin{equation}\label{eq:OED_p}\hyperlink{OEDp}{\tag{$\mathrm{OED}_{m_0}^{p}$}}
\w^{*p} \in \mathop{\argmin}\limits_{\w\in\consK}\Jc(\w).
\end{equation}

\begin{wrapfigure}{R}{0.39\textwidth}
\centering
\vspace{5pt}
\includegraphics[width=0.33\textwidth,keepaspectratio]{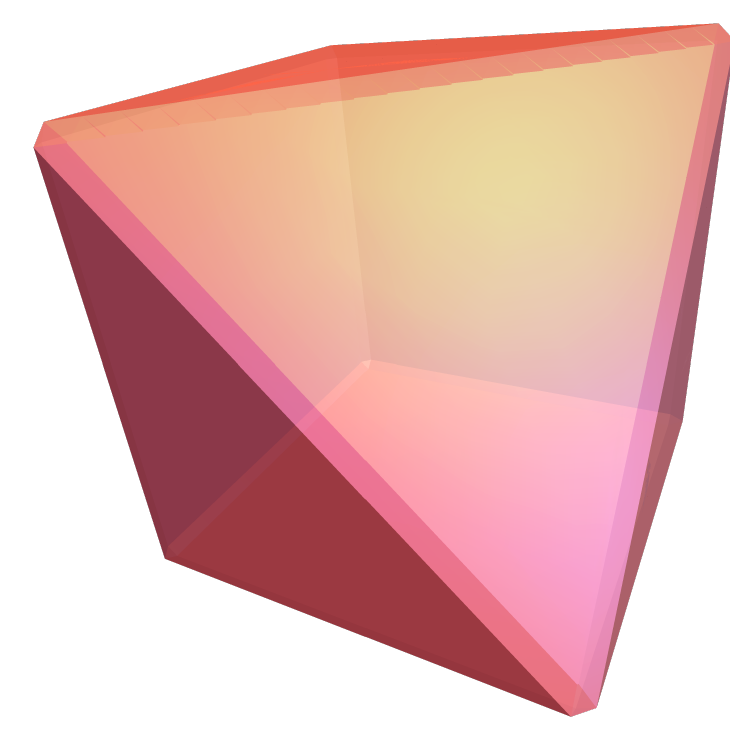}
\caption{Feasible set $\consK[1][1]$ in $\R^3$}
\label{fig:feasible}
\vspace{0cm}
\end{wrapfigure}

As $\consK\subset\consK[p']$ for all $0\leq p\leq p'\leq 1$, one in particular has $\|\w^{*p}\|_1\leq m_0$ for all $p$, allowing us to guarantee that at most $m_0$ (possibly fractional) sensors are employed by each $p$-relaxed solution; \revia{this can be viewed as a significant advantage over the soft-constrained sparsification approach, which cannot a priori guarantee the number of sensors employed by the found binary design without careful tuning of the weighting of the sparsity-enforcing term}. When $p=0$, \revia{$\prob[0]$ matches the original best sensors placement problem whenever the objective $\Jc$ is  continuous and monotonous in the sense 
\begin{equation*}\label{eq:cond_monotonicity}
	\Jc(\w)\leq\Jc(\v) \quad \text{whenever $\w\geq \v$ pointwise,}
\end{equation*}
as this guarantees that at least one global optimum is achieved when placing exactly $m_0$ sensors with weight $1$; we will later see that e.g.~the A-optimal objective is indeed monotonous in this sense.}

\reviasafe{While the $1$-relaxed problem $p=1$ allows for severely non-binary designs, $\consK[1][m_0]$ is a compact, convex set, and so $\prob[1]$ becomes a convex optimisation problem for convex objective functionals $\Jc$.} With this as our starting point, we aim to give an \underline{explicit characterisation of global optima of $\prob[1]$}, and then study how this information can be used to construct a continuation-based approach in the spirit of \cite{AlePetStaGha14} to estimate optima of $\prob[0]$. In doing so, we will demonstrate how this method yields completely binary, high-performing experimental designs utilising no more than the prescribed number $m_0$ of sensors.

\paragraph{State of the art}

While the history of optimal experimental design can be traced back more than a century, and its formalisation as a general mathematical discipline might be attributed to Elfving in the 1950s \cite{Elf1952, Elf1959}, there has in recent years been a surge of interest in OED for infinite-dimensional inverse problems, especially those governed by partial differential equations (PDEs); for a detailed review, see \cite{Ale21}. While, as indicated above, various settings and optimality criteria are studied, we will here again limit our discussion to the example of A-optimal designs for Bayesian linear inverse problems, as this setting well illustrates the computational challenges involved, and serves to motivate our developments.

In this setting, there are, broadly speaking, four major obstacles to overcome when approaching the optimal experimental design problem \eqref{eq:OED} with penalty $\Jc(\w):=\tr(\Cpst(\w))$ as in \eqref{eq:bayesian_inversion} via iterative algorithms:

\begin{enumerate}[label=(\roman*)]

\item \label{problem:PDEs} The cost of repeated forward and adjoint evaluations of the parameter-to-observable map $\Fc:X\to\R^m$, typically involving a PDE solution map.

\item \label{problem:inverse} The cost of inverting the \emph{misfit Hessian} with shift $\Cpi$, i.e.~inverting $\Fc^*\Gmnih\Mw\Gmnih\Fc + \Cpi:X\to X^*$.

\item \label{problem:trace} The cost of evaluating the trace of the infinite-dimensional operator $(\Fc^*\Gmnih\Mw\Gmnih\Fc + \Cpi)^{-1}:X^*\to X$.

\item \label{problem:sparse} The difficulty in imposing the binary nature $\w^*\in\{0,1\}^m$ of the optimal experimental design $\w^*$.%, in the sense that most or all indices $k$ should satisfy $\w^*_k=0$ or $\w^*_k=1$.

\end{enumerate}

Significant advances have been made in terms of developing computational techniques to handle these issues. \cite{BThGhaMarSta13} lays out an efficient computational framework for the finite element discretisation and low-rank approximation of the posterior covariance, invoking the truncated singular value decomposition and the Sherman-Morrison-Woodbury formula \cite{SheMor50} to handle \ref{problem:PDEs} and \ref{problem:inverse}. Several variations on the low-rank approximation approach exist. \cite{SaiAleIps17} proposed a matrix-free method to obtain the low-rank approximation of the prior-preconditioned misfit Hessian $\Cph\Fc^*\Gmnih\Mw\Gmnih\Fc\Cph$ for D-optimal designs, via the randomised subspace iteration algorithm \cite{LibWooMarRokTyg07}. By an equivalent formulation of \eqref{eq:bayesian_inversion}, \cite{KovAleSta20} demonstrated that when the number $m$ of candidate sensor locations is low, it is instead beneficial to obtain a low-rank approximation in the measurement dimension $m$, as this can be done independently of the design $\w$, and so re-used in iterative optimisation schemes. In \cite{AlePetStaGha14}, one instead took a \enquote{frozen low-rank approximation} of $\Gmnih\Fc\Cph$, which retains independence of the design $\w$, but is of significantly smaller dimension than the measurement-space formulation when $m$ is larger than the discretisation dimension of the parameter space; we will later employ a similar approach when studying our numerical examples. For all variations, the low-rank decomposition and subsequent approximate inversion of the prior-preconditioned misfit Hessian enables the evaluation of the trace in \ref{problem:trace}, frequently via Monte Carlo trace estimators as proposed in \cite{HabMagLucTen12}, although we will in this article instead obtain trace evaluation directly as the trace of very low-dimensional matrices.

%, or directly as the trace of low-dimensional matrices as in \cite{}

The above computational methods greatly alleviate \ref{problem:PDEs}--\ref{problem:trace}; however, the total computational cost of using iterative methods to solve the optimal design problem \eqref{eq:OED} may still be significant when using fine discretisations of the parameter space $X$ and when the number of candidate sensor locations $m$ is large. Moreover, as previously described, current approaches to resolving \ref{problem:sparse} may in themselves be computationally costly, e.g.~due to requiring continuation algorithms, employing non-smooth or non-convex penalty functionals, or requiring the use of specialised strategies as in \cite{YuAni21}. While the convexity of the A-optimal objective (see \cite[p.~267 (B.96)]{Uci04}) ensures the existence of a globally optimal design $\w^*$ for convex penalty functionals, thus also guaranteeing convergence of various iterative methods towards it, one generally does not obtain completely binary designs, and can in general not verify whether binary approximations of the iteration results are global optima.

\hypertarget{Contribution}{\paragraph{Contribution}}\label{par:contribution}

The key \revix{theoretical} contribution of this article is the novel perspective of studying non-smooth convex first-order optimality criteria for the $1$-relaxed sensor placement problem $\prob[1]$, and in doing so obtaining the ability to efficiently and precisely characterise global optima \revia{for a broad class of design objectives, including but not limited to the A-optimal objective}. In particular, \tref{thm:optimality} provides explicitly verifiable necessary and sufficient criteria for a given design $\w$ to be a global optimum of $\prob[1]$, enabling an optimality analysis that to the author's knowledge thus far was not available. \revix{Simultaneously, this result provides a lower bound on the achievable objective value of any binary design, and enables a surprisingly powerful classification of sensors as \emph{dominant} or \emph{redundant}, that is, always on or always off.}

\revix{Our second main contribution is the formulation of a highly computationally efficient framework for identifying binary A-optimal designs, driven by the following three key innovations:

\begin{itemize}

\item A $p$-continuation algorithm, leveraging \tref{thm:optimality} to identify the global non-binary optimum and fix sensors classified as dominant or redundant, then iteratively solving $\prob$ for smaller and smaller $p$ until an approximate solution of $\prob[0]$ is obtained. Due to the exact hard constraint $\|w\|^p_p\leq m_0$, this algorithm is guaranteed to find a design using no more than a prescribed number $m_0$ of sensors, without the need for parameter tuning (\algref{alg:p_cont}).

\item Frozen low-rank forms of the A-optimal objective and its derivatives, with emphasis on obtaining computationally efficient, partially trace-free formulations that do not require componentwise evaluation (Theorem \ref{thm:low_rank_OED} and, for multiple observations, \tref{thm:multiple_OED}).%, the latter showcasing the situation where multiple observations are made per sensor).

\item A full derivation of the computational complexities in the above formulations, highlighting their benefits for high-performance numerical application (\ssref{ssec:complexity}). 

\end{itemize}

%In contrast to previous continuation-type algorithms for optimal experimental design, \algref{alg:p_cont} is ,

%In particular, Theorems \ref{thm:low_rank_OED} and \ref{thm:multiple_OED} directly support this effort by advancing the theory of low-rank, efficient evaluations of the A-optimal objective, developing expressions for trace-free evaluation of its two first derivatives without the need for componentwise calculations. The latter Theorem showcases efficient approaches to the situation where multiple observations are made per sensor, i.e.~when $\Fc:X\to\R^{m\mobs}$, $\mobs>1$, with $m$ being the number of sensors. \ssref{ssec:complexity} demonstrates the efficiency of these formulations by outlining the associated computational complexities.

Finally, \sref{sec:numerics} provides a detailed numerical study of our techniques -- both theoretical and computational -- when applied to an inverse source problem governed by the Helmholtz equation.

}

\section{Optimality in the sensor placement problem}\label{sec:optimal}

\revia{In this section, we will provide an optimality criterion for the optimal experimental design problem. While we will keep the best sensors placement penalty functional $\Rc$ of \eqref{eq:OED_p} fixed for the remainder of the work, we allow for a very general class of objective functionals $\Jc$. In particular, the results of this section are not restricted to the A-optimal framework that was discussed in the Introduction.}

Given $m_0\in\N$ and $p\in[0,1]$, a $p$-relaxed \emph{optimal design} $\w^{*p}\in\R^m$ is one that satisfies \eqref{eq:OED_p}. As the constraint set $\consK$ is a compact set, existence of $\w^{*p}$ in $\consK$ is ensured for all continuous $\Jc$, although it may in general be non-unique. As previously remarked, the constraint set $\consK$ is only convex for $p=1$; as such, we devote the rest of this section to finding global optima $\w^*:=\w^{*1}$ of the $1$-relaxed problem $\prob[1]$, connecting it to the remaining $p$-relaxed problems in \sref{sec:numerics}.

\revib{The upcoming optimality analysis is strongly motivated by a surprising observation from empirical testing: Numerical solutions of $\prob[1]$, while in principle allowed to be fully non-binary, consistently produce outputs with a large number of binary entries, equal to $0$ or $1$. As the case $p=1$ of $\prob$ is the most permissive, in the sense that the cost of placing a fractional sensor is lower than for $p<0$, it is not unreasonable to conjecture that $\w^*_k=0$ may imply $\w^{*p}_k=0$ for $p<1$ as well. This observation suggests the following naming conventions:}

 %-- as will be demonstrated in the following section -- applying \cref{cor:a_priori} may often solve the sensor placement problem in one step. Even if it does not, it allows to reformulate the problem as a potentially significantly lower-dimensional one. The key idea behind this reduction is utilising \cref{cor:a_priori} to label sensors as \emph{dominant} -- a sensor that must necessarily be part of any globally optimal design -- and \emph{redundant} -- a sensor that can never be part of any globally optimal design. By determining dominant and redundant indices, we can then permanently remove these indices from the optimisation problem \eqref{eq:OED_constrained}, significantly reducing computational complexity.

\begin{defn}\label{def:reddom}

Given the $1$-relaxed OED problem $\prob[1]$, an index $k\in\N$, $k\leq m$ and a global optimum $\w^*$, we call $k$:

\begin{itemize}

\item \emph{Dominant} if $\w^*_k=1$.

\item \emph{Redundant} if $\w^*_k=0$.

\item \emph{Free} otherwise.

\end{itemize}

\end{defn}

\revib{To be able to prove such an effect, we require the following standing assumptions:}

\begin{ass}\label{ass:standing}

Throughout, we assume that on a neighbourhood of $[0,1]^m$, the objective functional $\Jc:\R^m\to\R\cup\{+\infty\}$ is finite, convex and continuously differentiable. Whenever the target number $m_0$ of sensors has been fixed, we will furthermore always assume that the monotonicity condition
\begin{equation*}
\nabla \Jc(\w )_k < 0 \quad \text{for all $\w \in\consK[1]$}
\end{equation*}
is satisfied for at least $m_0$ indices $k$. 
\end{ass}

The assumptions on convexity and continuous differentiability are rather natural, as they are key requirements for identifying global minima via gradient-based methods. \revib{Extensions to non-convex $\Jc$ are clearly interesting, as this would allow for treatment of very general optimal experimental design problems; however, as such results by nature must either be local or must take advantage of some specific structure of the objective functional in question, we here restrict ourselves to the convex case.} \label{sentence:convexity}

It is natural to ask whether the global optimum $\w^*$ of $\prob[1]$ will be found in the interior of $\consK[1]$, or whether it will be found on its boundary $\partial\consK[1]$. A design $\w \in\consK[1]$ satisfies $\w \in\partial\consK[1]$ if and only if either $\w_k=0$ or $\w _k=1$ for \emph{any} index $k\in\N$, $k\leq m$, or if $\sum_{k=1}^m\w_k=m_0$.

From this, we see that it is desirable that at least $\w^*\in\partial\consK[1]$, as there is otherwise no hope for $\w^*$ to be binary in the sense $\w^*\in\{0,1\}^m$. The key to ensuring this is precisely the monotonicity condition of \asref{ass:standing}. Indeed, it is clear that such a condition prevents the optimum $\w^*$ from occuring in the interior of $\consK[1]$, as the smooth first-order optimality criterion $\nabla \Jc(\w^*)=\mathbf{0}$ cannot be satisfied. This is equivalent to demanding that
\[
	t\in[0,1] \mapsto
	\Jc\left(
	\w_1,\ldots,\w_{k-1},t,\w_{k+1},\ldots,\w_m
	\right)\in\R
\]
be a strictly decreasing map for all $\w \in\consK[1]$, for at least $m_0$ indices $k$. Intuitively, this can be interpreted as saying that at least $m_0$ of the sensors are always \enquote{informative}, in the sense that placing them will always improve, rather than worsen, the objective $\Jc$. While this condition may seem somewhat abstract, it turns out to be naturally satisfied for e.g.~the A-optimal objective $\Jc(\w ):=\tr(\Cpst(\w ))$ -- indeed, satisfied for all indices $k$ -- as we will demonstrate in \sref{sec:tracefree}.

The most desirable outcome would be that $\w^*\in\{0,1\}^m$ and $\sum_{k=1}^m\w_k=m_0$, that is, $\w^*$ is fully binary and utilises exactly $m_0$ sensors; however, this does not seem to occur frequently in practice for the $1$-relaxed problem $\prob[1]$. In spite of this, as we will demonstrate, the presence of \emph{any} binary indices in $\w^*$ already provides significant information on how to approximately solve $\prob$ for general $p\in[0,1)$.

With this as our motivation, we give sufficient conditions for binary indices in $\w^*$ to occur, which turn out to be surprisingly mild. By treating the constrained optimisation problem exactly, as opposed to employing a soft penalty term and/or a barrier method-based approach, this feature becomes a natural consequence of the first order optimality conditions in the convex, non-smooth optimisation problem, \revib{and allows us to fully characterise global optima of $\prob[1]$ in the following key Theorem. It is well known, see e.g.~\cite[Prop.~B.10]{Ber1999}, that a local and global minima coincide for convex optimisation over convex, compact sets, and that at least one global minimum exists; thus, in fact, all minima can be characterised in this manner.}

\begin{thm}[\revix{Redundant-dominant classification of the optimum}]\label{thm:optimality}

Given \asref{ass:standing} and $m_0\in\N$, $m_0\leq m$, fix any $\w^*\in\consK[1]$. By reordering, if necessary, assume moreover that the gradient $\nabla\Jc(\w^*)\in\R^m$ satisfies the ordering
\begin{equation}\label{eq:ordering}
	\nabla \Jc(\w^*)_1 \leq \nabla \Jc(\w^*)_2 \leq \ldots 
	\leq \nabla \Jc(\w^*)_m.
\end{equation}
\revib{We distinguish between two possible cases:}

\begin{enumerate}

\item\label{thm:optimality:good} \revib{\ul{$\nabla\Jc(\w^*)_{m_0} < \nabla\Jc(\w^*)_{m_0+1}$}.} Then $\w^*$ is a global minimum of $\prob[1]$ if and only if the following statements hold:

\begin{enumerate}

\item\label{thm:optimality:good:dominant} $\w^*_k=1$ for all $k\leq m_0$\revix{, that is, all $k\leq m_0$ are \emph{dominant}.}

\item\label{thm:optimality:good:redundant} $\w^*_k=0$ for all $k>m_0$\revix{, that is, all $k> m_0$ are \emph{redundant}.}

\end{enumerate}

\revib{

In other words, we have the ordering
\begin{gather*}
\w^* = (1,\ldots,1,0,\ldots,0), \\
	\nabla\Jc(\w^*)_1 \leq \ldots \leq 
	\nabla\Jc(\w^*)_{m_0} <
	\nabla\Jc(\w^*)_{m_0+1} \leq
	\nabla\Jc(\w^*)_m.
\end{gather*}

}

\item\label{thm:optimality:bad} \revib{\ul{$\nabla\Jc(\w^*)_{m_0} = \nabla\Jc(\w^*)_{m_0+1}$}}. Define
\begin{align*}
	\umo & := 
	\max\left\{
		k\in\N, k\leq m \midsp \nabla\Jc(\w^*)_k < \nabla\Jc(\w^*)_{m_0}
	\right\}, \\
	\lmo & := \min\hspace{2pt}\left\{
		k\in\N, k\leq m \midsp \nabla\Jc(\w^*)_{m_0} < \nabla\Jc(\w^*)_k
	\right\},	
\end{align*}
under the convention that $\umo=0$ resp.~$\lmo=m+1$ if the respective above sets are empty. Then $\w^*$ is a global minimum of $\prob[1]$ if and only if the following statements hold:
%, that is, $\nabla\Jc$ is additionally ordered as
%\[
%	\nabla\Jc(\w^*)_\umo < 
%	\nabla\Jc(\w^*)_{\umo + 1} =
%	\ldots =
%	\nabla\Jc(\w^*)_{m_0} = 
%	\ldots =
%	\nabla\Jc(\w^*)_{\lmo - 1} <
%	\nabla\Jc(\w^*)_\lmo,
%\]
%assuming the endpoints exist. 

\begin{enumerate}

\item\label{thm:optimality:bad:dominant} $\w^*_k=1$ for all $k\leq\umo$\revix{, that is, all $k\leq \umo$ are \emph{dominant}.}

\item\label{thm:optimality:bad:redundant} $\w^*_k=0$ for all $k\geq\lmo$\revix{, that is, all $k\geq \lmo$ are \emph{redundant}.}

\item\label{thm:optimality:bad:boundary}
$\sum_{k=1}^m\w^*_k = m_0$. In particular, $\w^*\in\partial\consK[1]$.

\end{enumerate}

\revib{

In other words, we have the ordering
\begin{gather*}
\hspace{-10mm}\w^* = (1,\ldots,1,
	\w^*_{\umo+1},\ldots,\w^*_{\lmo-1},
	0,\ldots,0), \\
	\resizebox{.9\hsize}{!}{$\nabla\Jc(\w^*)_1 \leq \ldots \leq
	\nabla\Jc(\w^*)_\umo < 
	\nabla\Jc(\w^*)_{\umo + 1} = 
	\ldots = 
	\nabla\Jc(\w^*)_{m_0} =
	\ldots =
	\nabla\Jc(\w^*)_{\lmo - 1} <
	\nabla\Jc(\w^*)_\lmo \leq \ldots \leq
	\nabla\Jc(\w^*)_m$},
\end{gather*}
assuming the endpoints exist, where we have no further information regarding $(\w^*_k)_{k={\umo+1}}^{\lmo-1}\subset[0,1]$ except $\sum_{k=\umo+1}^{\lmo-1}\w^*_k = m_0 - \umo$.

}

\end{enumerate}

\end{thm}

\begin{proof}

%To see \revib{that \iref{thm:optimality:boundary} is necessary for $\w^*$ to be a global minimum, assume instead that $m_{\w^*}:=\sum_{k=1}^m\w^*_k<m_0$. In particular, one must have $\w^*_k<1$ for at least one $k\leq m_0$. Now
%\[
%	t \in [0,\min\{
%	m_0-m_{\w^*},
%	1 - \w^*_k	
%	\}] \mapsto
%	\Jc\left(\w^* + te_k
%	\right) \in \R
%\]
%is a strictly decreasing function on a non-trivial interval, where $e_k\in\R^m$ is the $k$-th unit vector. By assumption, $\w^*+te_k\in\consK[1]$ for all sufficiently small $t$, meaning in particular that $\w^*$ is not even a local minimum of $\Jc$ in $\consK[1]$, which is a contradiction.}

%If $R$ has at least one and at most $m-m_0$ columns that are identically zero, let $I\subset\N$ denote the set of $k$ such that the $k$-th column $R_{:k}$ of $R$ is non-zero. For $k\notin I$, one has $Re_k=R_{:k}=\mathbf{0}$; thus, $\nabla \Jc(\w )_k=0$ for all $\w \in\consK[1]$, that is, $\Jc$ is constant in $\w_k$ for all $k\notin I$. As $\nabla \Jc(\w )_k<0$ for all $k\in I$, it is immediate that one must have $\w^*_k=0$ for all $k\notin I$, and the previous paragraph applies, as there are at least $m_0$ indices $k\in I$. 

%To see \iref{thm:optimality:dominant} and \iref{thm:optimality:redundant}, as well as \iref{thm:optimality:dominant:suff} and \iref{thm:optimality:redundant:suff}

%\cite[Lemma~2.21]{Tro10}
%\revib{In order to obtain the remaining claims,} w

We start by noting that as $\Jc$ is convex by \asref{ass:standing}, \cite[Thm.~27.4]{Roc1972} implies that $\w^*$ is a global minimum of $\Jc$ in $\consK[1]$ if and only if
\begin{equation}\label{eq:optimality}
	\inner[\R^m]{\Jf}{\w }\leq\inner[\R^m]{\Jf}{\w^*}, \qquad \mathfrak{J}:=-\nabla \Jc(\w^*)\in\R^m 
	%One sees that \eqref{eq:optimality} is equivalent to demanding $\inner[\R^m]{\Jf}{\w }\leq\inner[\R^m]{\Jf}{\w^*}$ for all $\w \in\consK[1]$.\inner[\R^m]{\Jf}{\w^*  - \w} \geq 0
\end{equation}
for all $\w \in\consK[1]$. Since the $(\Jf_k)_{k=1}^m$ are already in increasing order by construction, with $\Jf_{m_0}>0$ by the monotonicity requirement of \asref{ass:standing}, it is immediate \revib{in both case \iref{thm:optimality:good} and \iref{thm:optimality:bad}} that
\[
\sup_{\w \in\consK[1]}\inner[\R^m]{\mathfrak{J}}{\w }=\sup_{\stackrel{\w \in[0,1]^m}{ \sum_{k=1}^m\w_k\leq m_0}}\sum_{k=1}^m\Jf_k\w_k=\sum_{k=1}^{m_0}\Jf_k,
\]
and that the continuous function $\w \in\consK[1]\mapsto\inner[\R^m]{\mathfrak{J}}{\w }\in\R$ attains its maximum over the compact set $\consK[1]$. Thus, $\w^*$ is a global minimum if and only if $\inner[\R^m]{\mathfrak{J}}{\w^*} = \sum_{k=1}^{m_0}\Jf_k$. In case \iref{thm:optimality:good}, this is immediately equivalent to $\w^*_k=1$ for $k\leq m_0$ and $\w^*_k=0$ otherwise, due to the strict ordering of the first $m_0+1$ indices of $\Jf$ and due to the constraint $\w^*\in\consK[1]$, completing the claim in this case.

\revib{For the case \iref{thm:optimality:bad}, assume first that \iref{thm:optimality:bad:dominant}--\iref{thm:optimality:bad:boundary}~are satisfied. Then
\begin{align*}
	\inner[\R^m]{\mathfrak{J}}{\w^*} & = 
	\sum_{k=1}^{m}\Jf_k\w^*_k =
	\sum_{k=1}^{\umo}\Jf_k + 
	\sum_{k=\umo+1}^{\lmo-1}\Jf_k\w^*_k
	 = 
	\sum_{k=1}^{\umo}\Jf_k + 
	\Jf_{m_0}\sum_{k=\umo+1}^{\lmo-1}\w^*_k \\
	& = 
	\sum_{k=1}^{\umo}\Jf_k + 
	\Jf_{m_0}\left(m_0 -
	\sum_{k=1}^{\umo}\w^*_k -
	\sum_{k=\lmo}^{m_0}\w^*_k
	\right) \\
	& = 
	\sum_{k=1}^{\umo}\Jf_k + 
	\Jf_{m_0}\left(m_0 - \umo\right)
	= 
	\sum_{k=1}^{\umo}\Jf_k + 
	\sum_{k=\umo+1}^{m_0}\Jf_{m_0}
	= \sum_{k=1}^{m_0}\Jf_k.
\end{align*}
It follows that \iref{thm:optimality:bad:dominant}--\iref{thm:optimality:bad:boundary}~are sufficient for $\w^*$ to be a global minimum. Conversely, assume $\w^*$ fails to satisfy at least one of \iref{thm:optimality:bad:dominant}--\iref{thm:optimality:bad:boundary} In this case, it is straightforward (e.g.~by switching two indices or by addition) to construct some $\w \in\consK[1]$ so that $\inner[\R^m]{\mathfrak{J}}{\w^*}<\inner[\R^m]{\mathfrak{J}}{\w }$, thereby showing that \iref{thm:optimality:bad:dominant}--\iref{thm:optimality:bad:boundary}~are also necessary for $\w^*$ to be a global minimum; this completes the proof.}

%As a supremum of a continuous function over a compact set, this is attained -- necessarily at $\w^*$ -- meaning
%\[
%	\sum_{k=1}^m\Jf_k\w^*_k = \inner[\R^m]{\Jf}{\w^*} = \sum_{k=1}^{m_0}\Jf_k.	
%\]
%As $\w^*\in\consK[1]$, this equality can only hold if the $\w^*$ maximise the effect of the ordering of the $(\Jf_k)_{k=1}^m$. In particular, it holds if and only if $\sum_{k=1}^m\w^*_k=m_0$, $\w^*_k=1$ for $k\leq \umo$ and $\w^*_k=0$ for $k\geq \lmo$. Due to the equality $\Jf_k=\Jf_{m_0}>0$ for all $k\in\N$, $\umo< k<\lmo$, no further information can be deduced regarding these \revib{remaining} indices. The further claims are all variations on the above proof.

% \iref{thm:optimality:dominant} and \iref{thm:optimality:redundant} simply reformulate this statement, while \iref{thm:optimality:dominant:suff} and \iref{thm:optimality:redundant:suff} address the situation where $j_k=j_{\sg(m_0)}$ for more than one index $k$.

\end{proof}

To summarise, the above Theorem demonstrates that under a mild monotonicity condition, the optimal design will always utilise all $m_0$ of its permitted sensor placements, and further ties optimality to the ordering of the gradient components in the optimum. \figref{fig:w_star_1} on page~\pageref{fig:w_star_1} provides a visual representation of this ordering for \revib{case \iref{thm:optimality:bad}, and demonstrates that the equality $\nabla\Jc(\w^*)_k=\nabla\Jc(\w^*)_{m_0}$ for the indices} $k\in\N$, $\umo<k<\lmo$ can be numerically verified and may in fact hold for a quite large number of $k$. \revix{Moreover, the gradient ordering can be used to determine whether a given sensor is truly redundant, or whether it is a free sensor whose weight is merely close to zero.}

\revib{To illustrate how \tref{thm:optimality} enables us to fully characterise all global minima of $\prob[1]$, we consider a low-dimensional example, where the two cases \iref{thm:optimality:good} and \iref{thm:optimality:bad} lead to exactly three possible patterns.

\begin{exa}\label{exa:optimality}

Assume the conditions of \tref{thm:optimality}, and let us place $2$ sensors out of $3$ candidate locations, i.e.~$m=3$, $m_0=2$. Given $\w^*\in\consK[1]$ and the ordering \eqref{eq:ordering} of $\nabla\Jc(\w^*) = (\nabla\Jc(\w^*)_1,\nabla\Jc(\w^*)_2,\nabla\Jc(\w^*)_3)$ of the objective functional's gradient, there are exactly three possibilities for a design $\w^*\in\consK[1]$ to be a global minimum of $\prob[1]$:

\begin{enumerate}

\item \ul{$\nabla\Jc(\w^*)_1 \leq \nabla\Jc(\w^*)_2 < \nabla\Jc(\w^*)_3$}. In this case, \tref{thm:optimality}, \iref{thm:optimality:good}~applies, and $\w^*=(1,1,0)$. Sensors $1$ and $2$ are dominant, while sensor $3$ is redundant.

\item \ul{$\nabla\Jc(\w^*)_1 < \nabla\Jc(\w^*)_2 = \nabla\Jc(\w^*)_3$}. In this case, \tref{thm:optimality}, \iref{thm:optimality:bad}~applies with $\umo=1$ and $\lmo=4$. In this case, \tref{thm:optimality}, \iref{thm:optimality:bad:redundant}~offers no information, although $\w^*=(1,\w_2^*,\w_3^*)$ and $\w_2^*+\w_3^*=1$ follows from \tref{thm:optimality},  \iref{thm:optimality:bad:dominant}~and \iref{thm:optimality:bad:boundary} Sensor $1$ is dominant, while sensors $2$ and $3$ are free.

\item \ul{$\nabla\Jc(\w^*)_1 = \nabla\Jc(\w^*)_2 = \nabla\Jc(\w^*)_3$}. In this case, \tref{thm:optimality}, \iref{thm:optimality:bad}~applies with $\umo=0$ and $\lmo=4$. In this case, \tref{thm:optimality}, \iref{thm:optimality:bad:dominant}-\iref{thm:optimality:bad:redundant}~offer no information regarding $\w^*=(\w_1^*,\w_2^*,\w_3^*)$, although $\w_1^*+\w_2^*+\w_3^*=2$ follows from \tref{thm:optimality}, \iref{thm:optimality:bad:boundary} All sensors are free.

\end{enumerate}

\end{exa}

Thus, we see that dominant indices accompany large negative gradient entries, redundant indices accompany non-negative or small negative gradient entries, and free indices can only occur when the gradient takes identical values for all those free indices.

}

Based on \tref{thm:optimality}, it is possible to derive a rather surprising corollary, allowing one to infer \emph{individual} indices of any global optimum $\w^*$ given only computations of the gradient in auxiliary points \revia{whenever $\nabla \Jc$ is Lipschitz continuous, which follows naturally on compact domains e.g.~whenever $\Jc$ is twice continuously differentiable, as a result of the mean value theorem}. Estimating a Lipschitz-type constant may then prove sufficient to obtain some information on the ordering \eqref{eq:ordering} already from observing the ordering of gradients computed in arbitrarily chosen other points $\w $. Explicitly, this result takes the following form:%, allowing us to potentially fix some or all of the indices of the optimal design $w^*$ from a small number of handpicked gradient computations.

\begin{cor}\label{cor:a_priori}

%Assume the setting of \tref{thm:optimality}, including the ordering \eqref{eq:ordering}. 

Given \asref{ass:standing}, let $\w^*$ be a global optimum of $\prob[1]$, and \revia{assume there is a Lipschitz} constant $L>0$ be such that $\|\nabla \Jc(\w^*)-\nabla \Jc(\w )\|\leq L\|\w^*-\w \|$ for all $\w \in [0,1]^m$. Define the constants
\[
	L_{0} := \sqrt{m_0}L,
	\qquad
	L_{1} := \sqrt{m-m_0}L,
	\qquad
	L_{2} := \sqrt{2m_0}L.
\]
For any fixed index $k\in\N$, $k\leq m$, the following then both hold:

\begin{enumerate}

\item\label{cor:a_priori:dominant} Assume that at least one of the following hold:
\begin{itemize}
	\item $\nabla \Jc(\mathbf{0})_k + 2L_{0} <
	\nabla \Jc(\mathbf{0})_{k'}$ for at least $m-m_0$ indices $k'\in\N$, $k'\leq m$.
	\item $\nabla \Jc(\mathbf{1})_k + 2L_{1} <
	\nabla \Jc(\mathbf{1})_{k'}$ for at least $m-m_0$ indices $k'\in\N$, $k'\leq m$.
	\item There is some $\w\in\consK[1]$ so that $\nabla \Jc(\w)_k + 2L_{2} <
	\nabla \Jc(\w)_{k'}$ for at least $m-m_0$ indices $k'\in\N$, $k'\leq m$.	
\end{itemize}
Then $\w^*_k=1$, that is, $k$ is dominant.

\item\label{cor:a_priori:redundant} Assume that at least one of the following hold:
\begin{itemize}
	\item $\nabla \Jc(\mathbf{0})_k - 2L_{0} >
	\nabla \Jc(\mathbf{0})_{k'}$ for at least $m_0$ indices $k'\in\N$, $k'\leq m$.
	\item $\nabla \Jc(\mathbf{1})_k - 2L_{1} >
	\nabla \Jc(\mathbf{1})_{k'}$ for at least $m_0$ indices $k'\in\N$, $k'\leq m$.
	\item There is some $\w\in\consK[1]$  so that $\nabla \Jc(\w)_k - 2L_{2} >
	\nabla \Jc(\w)_{k'}$ for at least $m_0$ indices $k'\in\N$, $k'\leq m$.
\end{itemize}
Then $\w^*_k=0$, that is, $k$ is redundant.

\end{enumerate}

\end{cor}

\begin{proof}
%The existence of $L$ is clear from continuous differentiability of $\Jc$; indeed, as $\nabla \Jc$ is continuous and $[0,1]^m$ is compact, $\nabla \Jc$ is Lipschitz continuous, and the Lipschitz constant serves as an upper bound on feasible $L$.

Fix a global optimum $\w^*\in\consK[1]$. We proceed by first proving \iref{cor:a_priori:dominant} Assume that indices $k$, $k'$ satisfying
\[
\nabla \Jc(\mathbf{0})_k + 2L_0 <
	\nabla \Jc(\mathbf{0})_{k'}
\]
are given. We claim that $\nabla \Jc(\w^*)_k < \nabla \Jc(\w^*)_{k'}$. Indeed, 
\[
|\nabla \Jc(\w^*)_k - \nabla \Jc(\mathbf{0})_k| \leq
\|\nabla \Jc(\w^*) - \nabla \Jc(\mathbf{0})\| \leq 
L\|\w^*\| \leq 
\sqrt{m_0}L = L_0,
\]
as $\w^*\in\consK[1]\subseteq[0,1]^m$ implies $\|\w^*\|=\left(\sum_{l=1}^m(\w^*_l)^2\right)^{1/2} \leq \left(\sum_{l=1}^m\w^*_l\right)^{1/2} \leq \sqrt{m_0}^{1/2}$.

As an identical bound holds for $|\nabla \Jc(\w^*)_{k'} - \nabla \Jc(\mathbf{0})_{k'}|$, one has
\[
	\nabla \Jc(\w^*)_k \leq 
	\nabla \Jc(\mathbf{0})_k + L_0 < 
	\nabla \Jc(\mathbf{0})_{k'} - L_0 <
	\nabla \Jc(\w^*)_{k'}, 
\]
as required. By assumption, there are thus at least $m-m_0$ indices $k'$ such that $\nabla \Jc(\w^*)_k < \nabla \Jc(\w^*)_{k'}$; hence, it follows from \tref{thm:optimality}, \iref{thm:optimality:good:dominant}~or \iref{thm:optimality:bad:dominant} that $\w^*_k=1$, as claimed.

The remaining cases are all analogous, utilising 
\[
	\|\mathbf{1} - \w^*\| = \left(\sum_{l=1}^m(1-\w^*_l)^2\right)^{1/2} \leq \left(\sum_{l=1}^m1-\w^*_l\right)^{1/2} = \sqrt{m-m_0},
\]
where we have used \tref{thm:optimality} (either via \iref{thm:optimality:good} or via \iref{thm:optimality:bad:boundary}), as well as 
\[
	\|\w^* - \w\| = \left(\sum_{l=1}^m(\w^*_l-\w_l)^2\right)^{1/2} \leq \left(\sum_{l=1}^m|\w^*_l - \w_l|\right)^{1/2} \leq \left(\sum_{l=1}^m\w^*_l+ \w_l\right)^{1/2} \leq \sqrt{2m_0}
\]
due to $\w^*$, $\w\in\consK[1]\subseteq[0,1]^m$.

\end{proof}

\cref{cor:a_priori} is remarkable in that one may be able to provide a partial solution of the best sensors placement problem solely by computing the Lipschitz-type constant $L$ and two or three gradients. Moreover, any information yielded by \cref{cor:a_priori} is fully binary, that is, sensors will be set to either exactly $0$ (no sensor placed at the $k$-th location) or $1$ (sensor placed at the $k$-th location).

In practice, the Lipschitz-type constant $L$ may be too large to make practical use of \cref{cor:a_priori}, and in \sref{sec:numerics}, we will instead determine $\w^*$ via standard constrained solvers and employ \tref{thm:optimality} as verification. Nevertheless, we do wish to comment on an interesting consequence of \cref{cor:a_priori}: If there is indeed a sufficient gap in the gradient, due to some sensors in some sense \enquote{always} providing better information than others, or conversely some sensors providing inferior information, then one may fix these sensors a priori, eliminating them as variables, and obtaining a permanent dimensionality reduction in the minimisation problem; this idea will inspire our strategies for the later numerics.
\revib{\section{Low-rank redundant-dominant techniques for binary A-optimal designs}\label{sec:tracefree}}

In this section, we demonstrate the applicability of \tref{thm:optimality} \revib{as a starting point for a continuation-type algorithm -- \algref{alg:p_cont} -- to enable computation of near-optimal binary designs}. \revib{Two key benefits are gained from \tref{thm:optimality}: The ability to verify that any numerically approximate global minimum $\w^*$ of $\prob[1]$ satisfies the optimality criteria, and the ability to categorise sensors as \emph{redundant} or \emph{dominant} as in \dref{def:reddom}. This information will then be leveraged to simplify $\prob$ also for $p<1$ via the following continuation-based approach.}

\revib{

\subsection{The $p$-continuation algorithm}\label{ssec:p_cont}

}

Strongly inspired by \cite{AlePetStaGha14}, \algref{alg:p_cont} constitutes a continuation algorithm, where one first solves $\prob[1]$, \revib{noting which sensors are dominant ($\w^*_k=1$) or redundant ($\w^*_k=0$).} \label{sentence:theoretical_benefits} One then solves $\prob$ for increasingly smaller values of $p\in(0,1)$, until a binary design is found, \revib{keeping redundant and dominant sensors fixed throughout, reflecting their relatively high resp.~low value compared to the remaining, free sensors}. Compared to the cited work, where additive penalties smoothly approximating the $p$-norm were considered, we instead treated the $p$-norms exactly.%, as summarised in \algref{alg:p_cont}.

\begin{algorithm}[h!]
\caption{Binary OED by $p$-continuation \revix{via redundant-dominant classification}}\label{alg:p_cont}
\begin{algorithmic}[1]
\Require Initialisation as globally optimal non-binary design $\w^{0}:=\w^*$, power $p=1$, iteration $i=0$, continuation parameter $\dl\in(0,1)$
\While {$\w^i$ has entries significantly different from $0$ and $1$}
	\State $p \gets (1-\dl)p$ and $i\gets i+1$
	\State Solve the non-convex constrained optimization problem (e.g.~via the SLSQP algorithm \cite{SLSQP})
	\begin{align} \label{eq:p_cont:objective}
	\Jc^p(\z) & := \Jc(\z^{1/p}), \\
	\nabla \Jc^p(\z) & = \frac{1}{p}\nabla\Jc(\z^{1/p})\z^{1/p-1}, \\
	\mathclap{\hspace{2.5cm}
	0 \leq \z_k \leq 1 \qquad \text{for all $k\in\N$, $k\leq m$}, \qquad
	\sum_{k=1}^m \z_k \leq m_0,} \label{eq:p_cont:constr}
	\end{align}
	initialised at $\z := (\w^{i-1})^p$ and keeping \revix{dominant and redundant indices ($\w^*_k=1$, resp.~$\w^*_k=0$)} fixed, returning $\z^i$
	\State $\w^i \gets (\z^i)^{1/p}$
\EndWhile
\Ensure Binary design $\wpcont:=\w^i$ as approximate solution of $\prob[0]$.
\end{algorithmic}
\end{algorithm}

For $p<1$, \algref{alg:p_cont} rewrites the $p$-relaxed sensor placement problem $\prob$ via $\z :=\w^p$ to avoid the non-linear, non-smooth constraint $\|\w \|_p^p\leq m_0$. Indeed, the function $\w \in\consK[p]\mapsto\|\w \|_p^p\in\R$ is not differentiable in any $\w \in\consK[p]$ satisfying $\w_k=0$ for at least one index $k\in\N$, $k\leq m$; in contrast, the rewritten objective function $\Jc^p$ in \eqref{eq:p_cont:objective} is continuously differentiable in all $\z \in\consK[1]$, as $p<1$ implies $1/p-1>0$. Moreover, the rewriting allows the problem to be formulated as a linear constrained optimisation via \eqref{eq:p_cont:constr}, which empirically was found more stable than the original nonlinear constrained optimisation.

\revix{
\begin{rem}

This rewriting is consistent with the choice of keeping redundant indices satisfying $\w^*_k=0$ fixed. Since the gradient vanishes in any direction where $\z_k=0$, it is not possible for a gradient step to \enquote{recover} such indices once they have been zeroed out, that is, $\w^i_k=0$ implies $\w^{i+1}_k=0$ for all $i\in\N_0$, $k\in\N$, $k\leq \msens$. As such, the verification of the global optimality criterion in \tref{thm:optimality} becomes not only a way to validate the proposed global optimum $\w^*$, but also allows one to verify whether $\w^*_k$ is merely numerically close, but not identically equal to, $0$, or whether the index $k$ is truly redundant. \label{sentence:truly_redundant}

\end{rem}
}

Clearly, the elimination also leads to computational speed-up, as one only needs to compute the components of $\nabla\Jc(\z^{1/p})$ corresponding to non-redundant sensors. At the same time, this choice also serves to counteract oscillation and potentially remove many local optima in the $p$-relaxed problem $\prob$.

However, precisely due to this, it cannot be ascertained whether the global optima of $\prob$ become unreachable\revib{, or whether only local minima are reached; nor is it a priori clear at what value of $p$ one can expect the resulting design to be binary or numerically binary. Studying the consequences of the choice of parameter $\delta$ on the convergence of \algref{alg:p_cont} is also outside the scope of this article.}\label{sentence:p_cont_limitations} As such, a careful comparison of \algref{alg:p_cont} to the original problem $\prob$ would be of great interest. For the scope of this article, we will instead offer a comparison of the outputs of \algref{alg:p_cont} to randomly drawn designs in \sref{sec:numerics}, where we draw $10^3$ randomly chosen designs for each $m_0$ and note their A-optimality, as well as with the A-optimality of the non-binary global optima $\w^*$.

\revix{\subsection{The A-optimal objective in finite element discretisation}\label{ssec:FEM}}

To enable more concrete discussion, we fix for the remainder of the paper $X:=\Lt$ for some compact domain $\Om\subseteq\R^d$, $d\in\N$, and assume that the linear, bounded infinite-dimensional source-to-observable map is of the form 
\[
	\Fc:\Lt\to\R^m, \qquad \Fc=\Oc\circ \Sc,
\]
where $\Sc$ is a PDE solution operator and $\Oc$ a finite observation map. Specifically, we consider the problem of obtaining A-optimal designs when the experimenter has fixed an $n$-dimensional discretisation of $\Lt$ a priori, $n\in\N$. As in the Introduction, the noise is assumed to be Gaussian white noise with diagonal covariance matrix $\Gmn=\Mw[\sg]\in\R^{m\times m}$, $\sg\in\R^m$, $\sg>0$ pointwise.

\revib{As evaluation of the objective $\Jc$ and its gradient $\nabla\Jc$ would formally require solving infinite-dimensional PDEs, we will take the rather standard approach (see \cite{Ale21}) of instead working with finite-dimensional approximations of all infinite-dimensional components via a finite element (FEM) discretisation.} Following the conventions of \cite{BThGhaMarSta13}, \revia{we approach this by discretising the solution \eqref{eq:bayesian_inversion} of the Bayesian inverse problem \eqref{eq:forward}, that is, discretising} $\Lt$, the forward map $\Fc:\Lt\to\R^m$ and the prior covariance. The optimal experimental design is then found with respect to the corresponding discretised inverse problem.

\revia{We begin by considering a truly infinite-dimensional prior covariance operator on $\Lt$ in \eqref{eq:bayesian_inversion}}, which we assume consists of two applications of a self-adjoint PDE solution operator, that is, being of the form $\Kf^{-2}$ for some densely defined differential operator $\Kf$ on $\Lt$. Discretisation is carried out in the following manner: One constructs a finite-dimensional subspace of linearly independent Lagrange basis functions $(\phi_i)_{i=1}^n\subset\Lt$, and considers the discretised problems of recovering the coefficients $(a_i)_{i=1}^n\subset\R^n$ of the projection $f_h=\sum_{i=1}^na_i\phi_i\in V_h$ of the true unknown source $f\in\Lt$ onto $V_h$. One can then construct the \emph{mass} and \emph{stiffness} matrices $\mathbf{M}$ and $\mathbf{K}$ in $\R^{n\times n}$ via
\begin{gather*}
	\mathbf{M}_{ij} := \int_\Om \phi_i\phi_j\d x, \qquad
	\mathbf{K}_{ij} := \int_\Om \phi_i\Kf\phi_j\d x \quad
	\text{for all $i$, $j\in\N$, $i$, $j\leq n$,} \\
		\spann\{\phi_i\}_{i=1}^n =: V_h \subset \Lt \cap \dom(\Kf)
\end{gather*}
\revia{Abusing notation to avoid needing to distinguish between the undiscretised and the discretised prior, the latter is now given as $\Cp:\R^n_{\Mb}\to\R^n_{\Mb}$, $\Cp=\Cph\Cph:=\mathbf{K}^{-1}\mathbf{M}\mathbf{K}^{-1}\mathbf{M}$, where $\R^n_{\Mb}$ is the weighted finite-dimensional inner product space equipped with the inner product $\Mass{a}{a'}:=\inner[\R^n]{a}{\Mb a'}$, i.e.~the finite element space. A similar strategy discretises the components $\Sc$ and $\Oc$ of the parameter-to-observable map $\Fc$; again by abuse of notation, we redefine $\Fc:\R^n_{\Mb}\to\R^m$ as the discretised source-to-observable map, satisfying $\Fc^* = \mathbf{M}^{-1}\Fc^T$ by \cite[p.~9 (3.4)]{BThGhaMarSta13}, with $\Fc^T$ denoting the transpose of the associated system matrix.}

%\subsection{The A-optimal objective}\label{ssec:A_optimal}

Employing \cite[p.~10 (3.10) \& p.~13 (5.2)]{BThGhaMarSta13}, the posterior covariance of the discretised inverse problem \eqref{eq:forward} given the discretised prior $\Cp$ has the two equivalent formulations
\begin{equation}\label{eq:FEM_posterior}
	\Cpst = \left(
		\Fc^*\Gmnih \Mw \Gmnih\Fc + \Cpi		
	\right)^{-1} =
	\Cph\left(
		\Cph\Fc^*\Gmnih \Mw \Gmnih\Fc\Cph + I		
	\right)^{-1}\Cph.
\end{equation}

As adjoints in $\R^n_{\Mb}$ require some additional attention, which we here seek to bypass, we further manipulate \eqref{eq:FEM_posterior} so that the key operations take place in $\R^n$, without the weighted inner product. Indeed, let the $n\times n$ mass matrix $\Mb$ be decomposed as $\mathbf{M}=(\mathbf{M}^{1/2})^T\mathbf{M}^{1/2}$. Treating the half-power $\mathbf{M}^{1/2}$ as a map from $\R^n_{\Mb}$ to $\R^n$, resp.~treating the half-power $\mathbf{M}^{-1/2}$ as a map from $\R^n$ to $\R^n_{\Mb}$, and readily verifying that $\Mbh $ and $\Mbhi $ are each other's adjoints, as well as repeatedly employing the Shermann-Morrison-Woodbury formula $(A+(BC)^{-1})^{-1}=C(BAC+I)^{-1}B$ for generic matrices $A$, $B$, $C$, the latter two being invertible, \eqref{eq:FEM_posterior} is equivalent to
\begin{equation}\label{eq:FEM_posterior_equiv}
\begin{aligned}
	\Cpst(\w) & = \Cph\Mbhi \left(
		\Mbh \Cph\Fc^*\Gmnih \Mw \Gmnih\Fc\Cph\Mbhi  + I		
	\right)^{-1}\Mbh \Cph \\
	%& = \Cph\Mbhi \left(
	%	\Mbh \Kb^{-1}\Mb\Mb^{-1}\Fc^T\Gmnih \Mw \Gmnih\Fc\Kb^{-1}\Mb\Mbhi  + I		
	%\right)^{-1}\Mbh \Cph \\
	%& = \Cph\Mbhi \left(
	%	\Mbh \Kb^{-1}\Fc^T\Gmnih \Mw \Gmnih\Fc\Kb^{-1}(\Mbh )^T + I		
	%\right)^{-1}\Mbh \Cph \\
	& = \Cph\Mbhi \left(
		\Fb^T \Mw \Fb + I		
	\right)^{-1}\Mbh \Cph
\end{aligned}
\end{equation}
%\begin{equation}\label{eq:FEM_posterior_equiv}
%\begin{aligned}
%	\Cpst(w) & = \mathbf{K}^{-1}\mathbf{M}\left(
%		\mathbf{K}^{-1}\mathbf{M}\Fc^*\Gmnih \Mw \Gmnih
%		\Fc\mathbf{K}^{-1}\mathbf{M} + I		
%	\right)^{-1}\mathbf{K}^{-1}\mathbf{M} \\
%	& = \mathbf{K}^{-1}\mathbf{M}\mathbf{M}^{-1/2}
%		\left(
%		\mathbf{M}^{1/2}\mathbf{K}^{-1}\mathbf{M}\mathbf{M}^{-1}\Fc^T
%		\Gmnih \Mw \Gmnih
%		\Fc\mathbf{K}^{-1}\mathbf{M}\mathbf{M}^{-1/2} + I		
%	\right)^{-1}
%	\mathbf{M}^{1/2}\mathbf{K}^{-1}\mathbf{M} \\
%	& = \mathbf{K}^{-1}(\mathbf{M}^{1/2})^T
%		\left(
%		\mathbf{M}^{1/2}\mathbf{K}^{-1}\Fc^T
%		\Gmnih \Mw \Gmnih
%		\Fc\mathbf{K}^{-1}(\mathbf{M}^{1/2})^T + I		
%	\right)^{-1}
%	\mathbf{M}^{1/2}\mathbf{K}^{-1}\mathbf{M}
%\end{aligned}
%\end{equation}
where 
\begin{equation}\label{eq:F}
	\Fb:=\Gmnih\Fc\Cph\Mbhi =\Gmnih\Fc\mathbf{K}^{-1}(\mathbf{M}^{1/2})^T\in\R^{m\times n}
\end{equation}
is treated as a real matrix mapping between $\R^n$ and $\R^m$, without the weighted FEM inner product. Utilising the cyclic nature of the trace, we have
\begin{equation}\label{eq:A_optimal_FEM}
\Jc(\w):=\tr(\Cpst(\w)) = 
\tr\left(
		\left(
		\Fb^T
		 \Mw 
		\Fb + I		
	\right)^{-1}
	\Mbh \Cp\Mbhi 
	\right)
\end{equation}
with the real matrix $\Mbh \Cp\Mbhi :\R^n\to\R^n$.

\subsection{Low-rank calculation of the A-optimal objective}\label{ssec:low-rank}

While \eqref{eq:A_optimal_FEM} provides a rather compact expression for the A-optimal objective, it is nevertheless a trace involving the inverse of a four-times PDE solution operator; bypassing the computational effort required to evaluate this expression is a major focus in current optimal experimental design literature, see \cite{Ale21}.

To overcome this challenge, and as discussed in the Introduction, \emph{low-rank decompositions} have seen a high degree of popularity in works on A-optimality. Here, parts or all of the so-called prior-preconditioned misfit Hessian $\Fb^T\Mw \Fb:\R^n\to\R^n$ appearing in \eqref{eq:A_optimal_FEM} are substituted by a product of low-rank matrices. In so doing, one is able to significantly facilitate the inversion and subsequent trace evaluation, which may otherwise prove computationally unfeasible.

In what follows, we take an approach that is rather similar to the frozen low-rank approach employed in \cite{AlePetStaGha14}, wherein the prior-preconditioned design-free forward operator $\Fb$ (equivalently, its transpose $\Fb^T$) is approximated via low-rank matrices. In the cited work, the authors take a singular value decomposition (SVD)-based approach to this low-rank decomposition. We will instead formulate this as a QR decomposition, \revix{as the full information of the SVD is not required by \tref{thm:low_rank_OED}, and as the QR decomposition form, once obtained, provides the most computationally efficient formulation of the Theorem.} 

%One noteworthy aspect of the former is that most standard implementations of the QR algorithm, particularly those that provide low-rank approximations via e.g.~pivoting \cite{Cha87}, are not matrix-free. 

%One can briefly compare advantages and disadvantages to QR-based decompositions over SVD-based decompositions. One noteworthy aspect of the former is that most standard implementations of the QR algorithm, particularly those that provide low-rank approximations via e.g.~pivoting \cite{Cha87}, are not matrix-free. Thus, in order to apply them, the system matrix $\Fb^T\in\R^{n\times m}$ (equivalently, $\Fb\in\R^{m\times n}$) must be set up. This clearly requires storage of an $n\times m$ matrix; moreover, in the setting that will be discussed after \tref{thm:low_rank_OED}, computing $\Fb\Cph\Mbhi $ requires $2n$ discretized PDE solves, followed by $mn$ integrals $\dual{o_k}{\cdot}$. %Finally, the QR algorithm has complexity of order ???. Thus, the total complexity

\begin{rem}\label{rem:frozen}

\revib{A frozen approach can yield significant computational savings by avoiding the need for further PDE solves. As the frozen decomposition is independent of the design $\w$, and so needs only be computed once, this allows for vastly accelerated computation.}

\end{rem}

In the below, we will not address the precise origin of the decomposition; indeed, as the discretised adjoint operator $\Fb^T:\R^m\to\R^n$ is a matrix, the existence of an exact (that is, error-free) QR decomposition is theoretically guaranteed via Gram-Schmidt orthogonalisation\revib{, which we discuss further in \appref{app:QR};}\label{sentence:QR_app_mention} the error analysis when using an inexact QR decomposition mirrors that of \cite[p.~14]{BThGhaMarSta13}, wherein it is asserted that the error commited by inverting the low-rank approximation is proportional to the rational sum $\sum_{i>\ell}\frac{\lm_i}{\lm_i+1}$ of neglected eigenvalues of the misfit Hessian. We refer to \appref{app:QR}, \algref{alg:rSVD} for one suggested, matrix-free implementation of an approximate QR algorithm.

In what the author believes to be a novel result, we show that such a frozen decomposition leads to a trace-free, low-dimensional expression \eqref{eq:low_rank_OED:jacobian} for evaluating the entire gradient of the A-optimal objective \eqref{eq:A_optimal_FEM} in a single computation, and similarly for the Hessian. This is particularly noteworthy, as it entirely bypasses the need for separate evaluation of each gradient component, which can be computationally taxing for large values $m\in\N$ of candidate sensor locations, and does not necessitate methods such as trace estimation, which have proven to be a major computational cost in other works, e.g.~\cite{AlePetStaGha14}. \revix{In particular, this aligns well with \tref{thm:optimality}'s focus on obtaining information from the gradient.}

\begin{thm}[A-optimal objective via QR]\label{thm:low_rank_OED}

Consider the A-optimal objective 
\[
\Jc(\w):=\tr(\Cpst(\w)) = 
\tr\left(
		\left(
		\Fb^T
		 \Mw 
		\Fb + I		
	\right)^{-1}
	\Mbh \Cp\Mbhi 
	\right)
\]
as in \eqref{eq:A_optimal_FEM}, recalling $F:=\Gmnih\Fc\Cph\Mbhi :\R^n\to\R^m$ from \eqref{eq:F}.

Let $\ell\in\N$ be such that there exists a thin QR decomposition $\Fb^T=QR\in\R^{n\times m}$, where $Q\in\R^{n\times \ell}$ is semi-orthogonal, that is, $Q^TQ=I_{\ell}\in\R^{\ell\times\ell}$, the $\ell$-dimensional identity matrix, and $R\in\R^{\ell\times m}$ is upper triangular. Additionally, write 
\[
	\Lc_\w :=R\Mw R^T + I_\ell\in\R^{\ell\times\ell}, \qquad
	\Cbt:=Q^T\Mbh \Cp\Mbhi Q\in\R^{\ell\times\ell},
\]
the latter being independent of $\w$ and being decomposed as $\Cbt=(\Cbt^{1/2})^T\Cbt^{1/2}$. 

Then, given data $\gw\in\R^{m}$ and a prior mean $m_0\in\R^n_\Mb$, the posterior distribution in \eqref{eq:bayesian_inversion} satisfies
\begin{equation}\label{eq:low_rank_OED:bayesian_inversion_QR}
\begin{aligned}
\mpst(\w) & = \Cph\Mbhi Q\Lc_\w^{-1}R\Gmnih \gw + \Cpst\Cpi m_0
, \\
\Cpst(\w) & = \Cph\Mbhi \left(
I - QQ^T + 
Q\Lc_\w^{-1}Q^T
\right)\Mbh \Cph
\end{aligned}
\end{equation}
and in a neighbourhood of $[0,1]^m$, the objective functional $\Jc$, its gradient and its Hessian matrix satisfy
\begin{align}\label{eq:low_rank_OED:objective}
	\Jc(\w) & = \tr(\Cp) - \tr(\Cbt) + \tr\left(
		\Lc_\w^{-1}\Cbt
	\right) \in \R, \\
	\label{eq:low_rank_OED:jacobian}
	\nabla \Jc(\w) & = 
	\left(- \left|
\Cbt^{1/2}
\Lc_\w^{-1}R \e_k
\right|^2
	\right)_{k=1}^m \in \R^m, \\
	\label{eq:low_rank_OED:hessian}
	H_\Jc(\w) & = 
2\left[
R^T\Lc_\w^{-1}\Cbt\Lc_\w^{-1}R\right]\odot
\left[
R^T\Lc_\w^{-1}R
\right] \in \R^{m\times m}.
\end{align}

Here, $\odot: \R^{m_1\times m_2}\times \R^{m_1\times m_2}\to \R^{m_1\times m_2}$ is the Schur product (also known as Hadamard product or elementwise matrix product) satisfying
\[
	(A\odot B)_{lk} := A_{lk}B_{lk} \quad \text{for all $k$, $l\in\N$, $k\leq m_1$, $l\leq m_2$}
\]
for any positive integers $m_1$, $m_2\in\N$ and generic $m_1\times m_2$ matrices $A$, $B$. Moreover, $\e_k\in\R^m$ is the $k$-th unit vector.

\end{thm}

\begin{proof}

By construction,
\begin{equation*}
\tr(\Cpst(\w)) = 
\tr\left(
		\left(
		\Fb^T
		 \Mw 
		\Fb + I		
	\right)^{-1}
	\Mbh \Cp\Mbhi 
	\right)
= 
\tr\left(
\left(
QR\Mw R^TQ^T + I
\right)^{-1}\Mbh \Cp\Mbhi 
\right),
\end{equation*}
where $I\in\R^{n\times n}$ is the $n$-dimensional identity matrix. As $R\Mw R^T\in\R^{\ell\times\ell}$ is possibly singular (particularly if $\w$ contains multiple zero entries), the above inversion requires a variation of the famous Shermann-Morrison-Woodbury formula, due to \cite[p.~7 (23)]{HenSea1981}; applying it once forwards and once backwards yields
\[
\left(
QR\Mw R^TQ^T + I
\right)^{-1} = I - Q\left(R\Mw R^T + I_\ell\right)^{-1}R\Mw R^TQ^T = I - Q\left(I_\ell - \Lc_\w^{-1}\right)Q^T
\]
via the identity $Q^TQ=I_\ell$. \eqref{eq:low_rank_OED:bayesian_inversion_QR} follows when inserting the above into $\mpst(\w)$, while \eqref{eq:low_rank_OED:objective} follows by again taking advantage of the cyclic nature of the trace and collapsing $C=Q^T\Mbh \Cp\Mbhi Q\in\R^{\ell\times\ell}$.

To obtain \eqref{eq:low_rank_OED:jacobian}, we fix $k\in\N$, $k\leq m$ and differentiate (c.f.~\cite[Thms.~B.17 \& B.19]{Uci04}):
\begin{align*}
\frac{\partial}{\partial \w_k}\Jc(\w) & = 
\frac{\partial}{\partial \w_k}\tr\left(\Lc_\w^{-1}\Cbh\right)
= -\tr\left(
		\Lc_\w^{-1}\left[\left[
		D_\w \Lc_\w 
		\right]\e_k\right]\Lc_\w^{-1}\Cbt
	\right) \\
& = -\tr\left(
\Lc_\w^{-1}R\Mw[\e_k] R^T\Lc_\w^{-1}\Cbt
\right) 
= -\tr\left(
R^T\Lc_\w^{-1}\Cbt
\Lc_\w^{-1}R\Mw[\e_k]
\right) \\
& = -\sum_{l=1}^m
\inner[\R^m]{
R^T\Lc_\w^{-1}\Cbt
\Lc_\w^{-1}R\Mw[\e_k]\e_l
}{\e_l}
= -
\inner[\R^m]{
R^T\Lc_\w^{-1}\Cbt
\Lc_\w^{-1}R\e_k
}{\e_k} \\
& = -
\inner[\R^\ell]{\Cbt^{1/2}
\Lc_\w^{-1}R\e_k
}{\Cbt^{1/2}
\Lc_\w^{-1}R\e_k}
= - \left|
\Cbt^{1/2}
\Lc_\w^{-1}R\e_k
\right|^2,
\end{align*}
using the quotient and chain rules as well as the cyclic nature and definition of the trace, along with the fact that $\Mw[\e_l]\e_k=\e_k$ if $k=l$ and equals $0$ otherwise.

Finally, the Hessian \eqref{eq:low_rank_OED:hessian} follows from a straightforward differentiation of \eqref{eq:low_rank_OED:jacobian}, coupled with the facts that $\inner[\R^m]{x}{\Mw[\e_l]y}=x_ly_l$ and $(A\e_k)_l=A_{lk}$ for all $x$, $y\in\R^m$, all $A\in\R^{m\times m}$ and all standard unit vectors $\e_k$, $\e_l\in\R^m$, $k$, $l\in\N$, $k$, $l\leq m$. Explicitly,
\begin{align*}
\frac{\partial}{\partial \w_k\partial \w_l}\Jc(\w) & = 
2\inner[\R^\ell]{
\Cbth\Lc_\w^{-1}R\e_k
}{
\Cbth\Lc_\w^{-1}R\Mw[\e_l]R^T\Lc_\w^{-1}R\e_k
} \\
& = 
2\inner[\R^m]{
R^T\Lc_\w^{-1}\Cbt\Lc_\w^{-1}R\e_k
}{
\Mw[\e_l]R^T\Lc_\w^{-1}R\e_k
} \\
& = 
2\left(R^T\Lc_\w^{-1}\Cbt\Lc_\w^{-1}R\e_k\right)_l
\left(R^T\Lc_\w^{-1}R\e_k\right)_l \\
& = 
2\left(R^T\Lc_\w^{-1}\Cbt\Lc_\w^{-1}R\right)_{lk}
\left(R^T\Lc_\w^{-1}R\right)_{lk},
\end{align*}
which is equivalent to the claimed form of the Hessian. Positive semidefiniteness is a consequence of the Schur product theorem \cite[Thm.~VII]{Sch1911} and the positive semidefiniteness of the matrices $\Cbt$ and $\Lc_\w$ coupled with evident self-adjointness of the two components.

\end{proof}

As a consequence of the above, we consistently employ the notation $\ell\in\N$ to denote our \enquote{low-rank dimension}, being the smallest integer allowing for a frozen $\ell$-rank QR decomposition of the prior-preconditioned adjoint and satisfying $\ell\leq\min\{m,n\}$. 

\begin{rem}

\revix{This exposes a powerful property of our low-rank formulation: Even if the discretisation dimension $n\in\N$ is extremely large, it has no impact on the computational complexity of the calculations of \eqref{eq:low_rank_OED:objective}--\eqref{eq:low_rank_OED:hessian} beyond its indirect, limited effect on the low-rank dimension $\ell$ and the one-time cost of obtaining the QR decomposition.} 

\end{rem}

Practically, $\ell$ may be chosen significantly smaller than $\min\{m,n\}$ via truncation; as an example, one may indeed compute the low-rank approximation of $\Fb$ via a truncated SVD, then collapse the last two components to form the matrix $R$\revix{; we refer to \appref{app:QR}, \algref{alg:rSVD}}. The following additional remarks are furthermore immediate:

\begin{rem}

The components of $\nabla \Jc(\w)$ equal $-1$ times the column norms of the thin $\ell\times m$ matrix $\Cbt^{1/2}\Lc_\w^{-1}R$, allowing for efficient computation by assembling the matrix, then evaluating the Euclidean norms of each of its columns. Strategies to evaluate this quantity as efficiently as possible are highlighted in \ssref{ssec:complexity}. \revix{In particular, compared to the natural approach of differentiating $\Jc$ with respect to each component of the design $\w$ individually, which would require the computation of $m$ separate matrix traces, this expression allows for evaluation of $\nabla\Jc$ via a single matrix manipulation.}

\end{rem}

\begin{rem} 

If $R$ has no columns identically equal to $0$, one has $\nabla \Jc(\w)_k<0$ for all $\w \in[0,1]^m$ and all $k\in\N$, $k\leq m$, verifying the monotonicity condition of \tref{thm:optimality}. Evidently, $\nabla \Jc(\w)_k\leq 0$, and equality cannot hold as $R\e_k\neq 0$, with the two other matrices being invertible by construction. Conversely, $R\e_k=0$ would imply that the $k$-th sensor is redundant, and so can be removed from the OED.

\end{rem}

\begin{rem}

 $H_\Jc(\w)$ is positive semidefinite for all $\w$, that is, $\Jc$ is a convex function, and $H_\Jc(\w)$ satisfies
\begin{equation}\label{eq:low_rank_OED:hessian_norm}
	\|H_\Jc(\w)\| \leq 2\|\Lc_\w^{-1}\|^3\|\Cbt\|\|RR^T\|^2 \leq 2\|\Cbt\|\|RR^T\|^2
\end{equation}
for all $\w \in K$ by first employing the general inequality $\|A\odot B\|\leq\|A\|\|B\|$ for positive integers $m_1$, $m_2\in\N$ and generic $m_1\times m_2$ matrices $A$, $B$, see \cite{Dav1962}\footnote{In fact, as pointed out in the reference, the much stronger estimate $\|A\odot B\|\leq\left(\max_{ij}|A_{ij}|\right)\|B\|$ holds; taking advantage of this characterisation would significantly sharpen the bound \eqref{eq:low_rank_OED:hessian_norm}.}, then applying the Shermann-Morrison-Woodbury formula, which yields that the eigenvalues of $\Lc_\w^{-1}=(R\Mw  R^T + I_\ell)^{-1}$ are precisely
\[
	\left\{
		\dfrac{1}{1 + \lm_\w } \midsp \lm_\w \in\sigma(R\Mw R^T)
	\right\};
\]
by positive semi-definiteness of $R\Mw R^T$ for all $\w \in[0,1]^m$, this family is bounded above by $1$, yielding $\|(R\Mw R^T + I_\ell)^{-1}\|\leq 1$.

\end{rem}

%For the second inequality, we will prove that $\|(L_\w +I)^{-1}\|\leq 1$ for all $\w \in K$. In fact, denote by $\sg_{\min}(A)$ the smallest eigenvalue of $A$. We will rather show that
%\[
%	t\in\begin{cases}
%	(-1/\sg_{\min}(A),\infty), & \sg_{\min}(A)>0, \\
%	\R, & \text{otherwise}
%	\end{cases} 
%	\mapsto 
%	\left\|
%	\left(
%		tA + I
%	\right)^{-1}
%	\right\| \in [0,\infty) 
%\]
%is monotonously decreasing for any positive semi-definite symmetric matrix $A$. For take the eigenvalue decomposition $USU^{-1}=A$, where $S$ is the diagonal matrix of eigenvalues of $A$, all of which are non-negative. In particular, the smallest diagonal entry of $S$ is $\sg_{\min}(A)$. As $U(tS+1)U^{-1}$ is an eigendecomposition of $tA+I$, we have 
%\begin{equation*}
%\left\|
%	\left(
%		tA + I
%	\right)^{-1}
%\right\| 
%= \max_{i}\left((tS+1)^{-1}\right)_{ii} 
%%= \min_i\frac{1}{tS_{ii} + 1} 
%= \frac{1}{t\sg_{\min}(A) + 1}.
%\end{equation*}

%Thus, $\|(R\Mw R^T+I)^{-1}\|\leq\|(0\cdot R\Mw R^T+I)^{-1}\|=1$ for all $\w \in K$This yields an alternative to \cite[Thm.~B.27]{Uci04} for the proof of the convexity of the objective functional.

\begin{rem}\label{rem:hessian_matvec}

The Hessian $H_\Jc(\w)$ can efficiently be applied to arbitrary $\v \in\R^m$ without setting up any large (i.e.~$m\times m$) matrices by taking advantage of the elementary identities $(A\hada B)\v = \diag(A\Mw[\v ]B^T)$ and $\diag(AB^T) = \left(\sum_{l=1}^\ell A_{kl}B_{kl}\right)_{k=1}^m = (\sum_{l=1}^\ell (A\hada B)_{kl})_{k=1}^m=:\rowsum (A\hada B)\in\R^m$ for arbitrary $\ell\times m$ matrices $A$, $B$; combining these properties yields the formula
\begin{equation}\label{eq:low_rank_OED:hessian_matvec}
\begin{aligned}
	H_\Jc(\w)\v & = 2\left[
R^T\Lc_\w^{-1}\Cbt\Lc_\w^{-1}R\right]\odot
\left[
R^T\Lc_\w^{-1}R
\right]\v \\
	& = 2\,\diag\left(
		R^T\Lc_\w^{-1}\Cbt\Lc_\w^{-1}R\Mw[\v ]R^T\Lc_\w^{-1}R
	\right) \\
	& = 2\rowsum
		\left[
			R^T\Lc_\w^{-1}\Cbt\Lc_\w^{-1}R\Mw[\v ]R^T
		\right]\hada
		\left[
			R^T\Lc_\w^{-1}
		\right]
\end{aligned}
\end{equation}
where it is sufficient to obtain the small $\ell\times m$ matrix $\Lc_\w^{-1}R=(R\Mw R^T+I_\ell)^{-1}R$ only once, and which requires setting up only the Schur product of two small $m\times\ell$ matrices, as opposed to taking the Schur product of two large $m\times m$ matrices when setting up the full Hessian.

\end{rem}

The above bound on the Hessian immediately allows us to apply the results of \sref{sec:optimal}:

\begin{cor}\label{cor:a_priori_A_optimal}

\cref{cor:a_priori} applies for $\Jc$ being the A-optimal objective, with
\[
	L_0 := \sqrt{m_0}\ell^2\|\Cbt\|\|RR^T\|^2,
	\qquad
	L_1 := \sqrt{m-m_0}\ell^2\|\Cbt\|\|RR^T\|^2,
	\qquad
	L_2 := \sqrt{2m_0}\ell^2\|\Cbt\|\|RR^T\|^2.
\]

\end{cor}

\begin{proof}

This is an immediate consequence of the multivariate mean value inequality 
\[
\|\nabla \Jc(\w^*) - \nabla \Jc(\w)\| \leq 
\sup_{t\in(0,1)}\|H_J(t\w^*+(1-t)\w)\|_F\|\w^*-\w \| \leq 
2\ell^2\|\Cbt\|\|RR^T\|^2\|\w^*-\w \|
\]
for all $\w \in K$, where $\|\cdot\|_F$ denotes the Frobenius norm, and employing the norm bound in \eqref{eq:low_rank_OED:hessian_norm} along with the rank inequality $\rank (A\hada B)\leq\rank\,A \cdot \rank\,B$ and the norm inequality $\|A\|_F\leq\rank(A)\|A\|$ for generic matrices $A$, $B$.

\end{proof}

\subsection{Computational complexity}\label{ssec:complexity}

The development of the low-rank objective and its derivatives presented in Theorems \ref{thm:low_rank_OED} and \ref{thm:multiple_OED} puts particular emphasis on computational efficiency. \revia{This allows the experimenter to perform most or all computations in the low-rank dimension $\ell\in\N$, eschewing the need for elementwise computation of the gradient or explicit assembly of the Hessian via \rref{rem:hessian_matvec}, and moreover suppressing the effect of large discretisation dimensions $n\in\N$ or number of sensor placements $m\in\N$ if the relevant information can be condensed down to a potentially much smaller low-rank dimension $\ell$}. \label{sentence:low_rank_benefits} We recall that the relevant dimensions are $n\in\N$, representing the discretisation size of the parameter space $X$ wherein the experimenter wishes to recover the unknown quantity $f$, see \eqref{eq:forward}, as well as $m\in\N$, representing the number of candidate locations for sensor placement in the design, and the low-rank dimension $\ell\in\N$, $\ell\leq\min\{m,n\}$, representing the rank of the QR decomposition appearing in \tref{thm:low_rank_OED}. We also employ $\tPDE>0$ to denote time complexity of applying the prior-preconditioned parameter-to-observable map, which is assumed to represent two PDE solves and application of an observation operator (alternatively, the adjoint operation), and which will depend on both $n$ and $m$, while $q\in\N$ will be the number of subspace iterations required by \algref{alg:rSVD}. %We will only provide numbers for the single observation case, \tref{thm:low_rank_OED}, as the additional scaling in the multiple observations case is apparent.

In the following paragraphs, we discuss the computational cost of any necessary precomputations, which are independent of the design $\w$ and are thus done in a completely offline manner, as well as evaluations of the objective $\Jc(\w)\in\R$, its gradient $\nabla\Jc(\w)\in\R^m$ and the Hessian matrix-vector product $H_\Jc(\w)\v \in\R^m$, where $\w \in[0,1]^m$, $\v \in\R^m$ are arbitrary but fixed. \tabref{tab:complexity} provides a summary. 

Throughout, it is assumed that for arbitrary positive integers $n_1$, $n_2$, $n_3\in\N$ and matrices $A\in\R^{n_1\times n_2}$, $B\in\R^{n_2\times n_3}$, the computational complexity of calculating $AB$ is $O(n_1n_2n_3)$, while for invertible matrices $C\in\R^{n_1\times n_1}$, one can assemble e.g.~the LU decomposition in $O(n_1^3)$ time, after which one can solve $C^{-1}x\in\R^{n_1}$ with computational complexity of order $O(n_1^2)$ for any $x\in\R^{n_1}$. In particular, computing $C^{-1}A\in\R^{n_1\times n_2}$ can thus be done in a total of $O(n_1^3+n_1^2n_2)$ time.

\paragraph{Precomputation} To obtain the low-rank dimension $\ell$ in the above Theorems, various quantities need to be precomputed, most saliently the QR decomposition $\Fb^T=QR$, $Q\in\R^{n\times\ell}$, $R\in\R^{\ell\times m}$ which enables the low-rank objective calculations. \revib{We here consider computational costs for the randomised QR algorithm appearing in the Appendix, \algref{alg:rSVD}, based on the randomised subspace iteration algorithm \cite[Algs.~4.4 \& 5.1]{HalMarTro11}}, remarking that this version permits non-symmetric input matrices, in contrast to the one highlighted in e.g.~\cite[Algorithm 1]{Ale21}. In the final singular value decomposition, one may deliberately truncate additionally by removing sufficiently small eigenvalues in step \ref{alg:rSVD:SVD}, enabling further flexibility than simply adhering to the previously fixed choice of $\ell$. As noted in \cite[p.~240, p.~243 \& p.~247]{HalMarTro11}, the total cost of \algref{alg:rSVD} is dominated by $\ell(\tPDE(2q+1) + m n)$, ignoring various lower order terms and assuming that factors depending on the time needed to draw a real Gaussian are irrelevant relative to factors involving the number $m$ of candidate sensor locations. %, or $(2q+1)\ell_0(\tPDE + m\tRNG + n) + \ell_0^2(n+m) + \ell m n$ when there is a significant difference between the pre-truncated low-rank dimension $\ell_0$ and its value $\ell$ after eigenvalue truncation, where $\tRNG>0$ is the time needed to draw a real Gaussian.

After assembling the QR decomposition, and recalling the matrices introduced in \ssref{ssec:FEM} and \tref{thm:low_rank_OED}, one must additionally compute the trace of $\Cp:\R^n_\M\to\R^n_\M$ for \eqref{eq:low_rank_OED:objective}, and must assemble and compute the trace of the repeatedly employed small basis change matrix $C=Q^T\Mbh \Cp\Mbhi Q\in\R^{\ell\times\ell}$. Assuming the bandwidth of the FEM mass matrix $\Mb$ is much smaller than the discretisation dimension $n$, the cost of these operations is at worst of order $O((n+\ell)(n + \tPDE) + \ell^2n)$. 

\paragraph{Inverse $\Lc_\w^{-1}$}

The common term appearing in all three quantities of interest -- the objective $\Jc(\w)$, the gradient $\nabla\Jc(\w)$ and the Hessian matrix-vector product -- is the inverse $\Lc_\w^{-1}=(R\Mw R^T+I_\ell)^{-1}$ of the low-rank misfit Hessian. In fact, this is, rather strikingly, the \emph{only} term that depends on the design $\w$. As such, while $\Lc_\w$ cannot be computed or inverted in an offline manner, it is possible to assemble it for any given $\w$, then e.g.~compute, store and share its LU decomposition. 

Since $\Mw $ is a multiplication operation, $R\Mw \in\R^{\ell\times m}$ can be computed in $O(\ell m)$ time, rather than the $O(\ell m^2)$ time needed for a full matrix product, while assembling the matrix $\Lc_\w  = (R\Mw )R^T + I_\ell\in\R^{\ell\times\ell}$ can be done in $O(\ell^2 m)$ time\revib{, although a small speed-up is available by employing the upper triangular structure of $R$}, with the addition being negligible. Meanwhile, the LU decomposition of $\Lc_\w$ can be completed in $O(\ell^3)$ time. While this cost applies to all the below calculations, it need only be done once for each design $\w$.

\paragraph{Objective $\Jc(\w)\in\R$} 

With the LU decomposition already available, the inversion $\Lc_\w^{-1}\Cbt\in\R^{\ell\times\ell}$ can be carried out in $O(\ell^3)$ time. As this assembles the full (small) matrix $\Lc_\w^{-1}\Cbt\in\R^{\ell\times\ell}$, its trace can be computed in negligible $O(\ell)$ time.

\paragraph{Gradient $\nabla\Jc(\w)\in\R^m$}

We have previously claimed that the gradient can be computed in an efficient manner, by being trace-free, low-rank and not requiring each gradient component to be computed individually, but rather providing the entire gradient vector as the output of a single collection of low-rank matrix operations; we now demonstrate these properties. %Moreover, it is particularly noteworthy that the gradient and the objective evaluation can share almost the entire calculation.

From equation \eqref{eq:low_rank_OED:jacobian}, it is clear that it suffices to compute the column norms of the $\ell\times m$ matrix $\Cbth\Lc_\w^{-1}R$. Again employing the above LU decomposition of $\Lc_\w^{-1}$, one might naively compute $\Lc_\w^{-1}R\in\R^{\ell\times m}$ in $O(\ell^2m)$ time, then compute $\Cbth(\Lc_\w^{-1}R)\in\R^{\ell\times m}$ in additional $O(\ell^2m)$ time.

A speedup is, however, available. The rewriting $\Cbth\Lc_\w^{-1}R = (\Lc_\w^{-1}(\Cbth)^T)^TR$ suggests one may first obtain $\Lc_\w^{-1}(\Cbth)^T)$ in $O(\ell^3)$ time. The final matrix product $(\Lc_\w^{-1}(\Cbth)^T)^TR$ leads to a total of $O(\ell^2m+\ell^3)$ time, as the column norm can be computed in approximately $O((2\ell-1)m)$ time and is thus generally negligible. 

Once $\nabla\Jc(\w)$ has been computed, the evident equality $\Lc_\w^{-1}\Cbt=(\Lc_\w^{-1}(\Cbth)^T)\Cbth\in\R^{\ell\times\ell}$ means $\Jc(\w)$ can be computed without additional inversion; while both this matrix multiplication and the solve $\Lc_\w^{-1}\Cbt\in\R^{\ell\times\ell}$ require $O(\ell^3)$ time, the former can generally be expected to be several times faster in practice.

\paragraph{Hessian matrix-vector product $H_\Jc(\w)\v $}

From \eqref{eq:low_rank_OED:hessian_matvec}, it suffices to compute the row sum of the $m\times\ell$ matrix $2
		\left[
			R^T\Lc_\w^{-1}\Cbt\Lc_\w^{-1}R\Mw[\v ]R^T
		\right]\hada
		\left[
			R^T\Lc_\w^{-1}
		\right]$. In this calculation, it is not possible to avoid computing $\Lc_\w^{-1}R\in\R^{\ell\times m}$, requiring an inevitable $O(\ell^2m)$ time. This can then be employed two additional times in the Hessian matrix-vector product calculation, so that the remaining matrix products
\[
	\left(\Lc_\w^{-1}R\right)\left(\Mw[\v ]R^T\right)\in\R^{\ell\times\ell}, \quad 
	\Cbt\left(\Lc_\w^{-1}R\Mw[\v ]R^T\right)\in\R^{\ell\times\ell}, \quad
	\left(\Lc_\w^{-1}R\right)^T\left(\Cbt\Lc_\w^{-1}R\Mw[\v ]R^T\right)\in\R^{m\times\ell}
\]
can be computed in a total of $O(2\ell^2m + \ell^3)$ time. Finally, the elementwise product $\hada$ and the row sum lead to a negligible contribution of $O((2\ell-1) m)$.
		
When $H_\Jc(\w)\v $ is computed in the above manner, the assembled matrix $\Lc_\w^{-1}R$ can be used to obtain a dramatic speed-up in the gradient computation, as the only remaining cost is the matrix product $\Cbth(\Lc_\w^{-1}R)\in\R^{\ell\times m}$, contributing only $O(\ell^2m)$ time, as well as the cheap column norm extraction. The drawback of this strategy is that if one additionally needs to compute the objective $\Jc(\w)$, then the product $\Lc_\w^{-1}\Cbt\in\R^{\ell\times\ell}$ must still be obtained; nevertheless, as argued previously, this can still be done in $O(\ell^3)$ time.

\begin{table}[h!]
\centering
\begin{tabular}{|c|c|c|}%c|c|c|}
\hline
& Worst case & Best case \\
\hline
Precomputation & $\tPDE((2q+2)\ell + n) + \ell m n$ & - \\
\hline
$\Jc(\w)$ & $\ell^2m + 2\ell^3$ &  $\ell^3$ if $\nabla\Jc(\w)$ or $H_\Jc(\w)\v $ already computed \\
%$\ell^4$ or $\ell$ if  already computed} \\
\hline
$\nabla\Jc(\w)$ & $2\ell^2m + 2\ell^3$ & $\ell^2m$ if $H_\Jc(\w)\v $ already computed \\
\hline
$H_\Jc(\w)\v $ & $4\ell^2m + 2\ell^3$ & - \\
\hline
\end{tabular}
\caption{Computational complexities (as $O(\cdot)$), particularly highlighting the surprising computational efficiency of this formulation for first-order methods. $m$: Number of candidate sensor locations. $n$: Discretisation dimension of unknown $f$. $\ell$: Low-rank dimension satisfying $\ell\leq\min\{m,n\}$.}
\label{tab:complexity}
\end{table}

From the above, it is clear that not only does the proposed approach obtain highly competitive complexity in terms of evaluating the A-optimal objective $\Jc(\w)$, but it rather surprisingly allows the gradient $\nabla\Jc(\w)$ to be computed at very little additional cost, despite returning a vector rather than a scalar, as it does not require any part of the computation to be calculated separately for each gradient component. This, coupled with the fact that the discretisation dimension $n$ plays no role in the complexity except through its indirect effect on the low-rank dimension $\ell\leq\min\{m,n\}$, and that no additional factor for trace estimation is needed, suggests that the formulas in Theorems \ref{thm:low_rank_OED} and \ref{thm:multiple_OED} provide a high-performing computational strategy for use in e.g.~first- and second-order methods when numerically identifying A-optimal designs.

Finally, all steps in the above can be fully parallelised, allowing the calculations to scale with available hardware. Put together, this suggests that the formulations of Theorems \ref{thm:low_rank_OED} and \ref{thm:multiple_OED} can be efficiently implemented even for very large values of $n$ and $m$.

\subsection{Multiple observations}\label{subsec:multiple}

Thus far, we have operated under the rather simple assumption that each sensor returns exactly one scalar observable. While this allows for straightforward analysis, it may not adequately characterise modern measurement tools. Indeed, it may be more realistic to assume that each sensor returns a finite vector of potentially complex-valued measurements, with each entry corresponding to e.g.~sampling at different time points, or, if understood in frequency domain, each entry corresponding to measurements on one frequency.

Fortunately, a rather straightforward extension of the above analysis is possible also in this setting, by employing block matrix notation, and modifying the forward map and the dependence of the observed data on the experimental design $\w$.

\begin{defn}\label{def:multiple}

Assume that each sensor returns $\mobs\in\N$ scalar observations, that is, that the design-dependent parameter-to-multiple-observables map can be expressed as $\Fc_\w :X\to\R^{m\mobs}$ and can be written in block form
\[
	\Fc_\w  f := \begin{bmatrix}
	\Mw \Fc_1 f \\
	\vdots \\
	\Mw \Fc_\mobs f
	\end{bmatrix} \in \R^{\msens\mobs}
\]
for all $f\in X$, $\w \in\R^m$, where $\Fc_\kobs\in L(X,\R^\msens)$ for all $\kobs\in\N$, $\kobs\leq\mobs$. We then denote by $\Mobs$ the $\R^{\msens\mobs\times \msens\mobs}$ diagonal matrix
\[
	\Mobs :=\begin{bmatrix}
		\Mw  \\
		& \Mw  \\
		& & \ddots \\
		& & & \Mw 
	\end{bmatrix} \qquad
\text{under which} \qquad
	\Fc_\w  f= \Mobs \begin{bmatrix}
	\Fc_1 f \\
	\vdots \\
	\Fc_\mobs f
	\end{bmatrix} \in \R^{\msens\mobs}
\]
for all $f\in X$, $\w \in\R^\msens$.
\end{defn}

In this setting, we can quickly adapt our previous results. %Indeed, if $\Fc_\kobs$, $\kobs\in\N$, $\kobs\leq\mobs$ satisfies the QR decomposition 

\begin{prop}\label{prop:multiple}

With the setting of \dref{def:multiple}, the following statements all hold:

\begin{enumerate}

\item $\Fc^*\Fc f = \sum_{\kobs=1}^\mobs\Fc^*_\kobs\Fc_\kobs f$ for all $f\in X$.

\item If for each $\kobs\in\N$, $\kobs\leq\mobs$, $\Fc_\kobs\in L(X,\R^\msens)$ is discretised as $\Fb_\kobs$, in the sense of \ssref{ssec:FEM} and \eqref{eq:F}, and satisfies the thin QR decomposition $\Fb_\kobs^T=Q_\kobs R_\kobs$, then the discretisation $\Fb\in\R^{\msens\mobs\times n}$ of $\Fc\in L(X,\R^{\msens\mobs})$ satisfies the thin QR decomposition
\begin{equation}\label{eq:multiple_observations_QR}
	\Fb^T = QR := 
	\begin{bmatrix}
	Q_1 &
	Q_2 &
	\ldots &
	Q_\mobs
	\end{bmatrix}
	\begin{bmatrix}
		R_1 \\
		& R_2 \\
		& & \ddots \\
		& & & R_\mobs
	\end{bmatrix}
	.
\end{equation}

\end{enumerate}

\end{prop}

While \eqref{eq:multiple_observations_QR} gives one possible $QR$ decomposition of the discretised parameter-to-multiple-observables map, it is typically not the optimal choice, in the sense that a lower-rank decomposition is frequently available. To see this, it is enough to realise that if $\Fc_\kobs = \Fc_{\kobs'}$ for any $\kobs\neq\kobs'$, then the above $QR$ decomposition will contain linearly dependent components and can be reduced. As such, while the $QR$ decomposition in \eqref{eq:multiple_observations_QR} can be used as an intermediate computational step, allowing each $\Fc_\kobs$ to be decomposed individually, it will generally only be used to compute a lower-rank $QR$ decomposition for $\Fb^T$.

The above Proposition leads to the following generalisation of \tref{thm:low_rank_OED} for the A-optimal objective $\Jc$ and its derivatives in the multiple observations setting:

\begin{thm}[$\Jc$ for multiple observations]\label{thm:multiple_OED}

With the setting of \dref{def:multiple} and \pref{prop:multiple}, and otherwise employing the notational conventions of \tref{thm:low_rank_OED}, consider the A-optimal objective
\[
\Jc(\w):=\tr(\Cpst(\w)) = 
\tr\left(
		\left(
		\Fb^T
		 \,\Mobs  
		\Fb + I		
	\right)^{-1}
	\Mbh\Cp\Mbhi
	\right),
\]

Let $\ell\in\N$ be such that $\Fb^T=QR\in\R^{n\times m\mobs}$, where $Q\in\R^{n\times \ell}$ is semi-orthogonal, that is, $Q^TQ=I_{\ell}\in\R^{\ell\times\ell}$, the $\ell$-dimensional identity matrix, and $R\in\R^{\ell\times m\mobs}$. Additionally, write \[
	\Lc_\w:=R\,\Mobs R^T + I_\ell\in\R^{\ell\times\ell}, \qquad \Cbt:=Q^T\Mbh \Cp\Mbhi Q\in\R^{\ell\times\ell},
\]
the latter being independent of $\w$ and being decomposed as $\Cbt=(\Cbt^{1/2})^T\Cbt^{1/2}$.

Then, given data $\g\in\R^{m\mobs}$, a prior mean $m_0\in\R^n$ and a prior mean $\Cp\in\R^{n\times n}$, the posterior distribution in \eqref{eq:bayesian_inversion} satisfies
\begin{equation}\label{eq:multiple_OED:bayesian_inversion_QR}
\begin{aligned}
\mpst(\w) & = \Cph\Mbhi Q\Lc_\w^{-1}R\Gmnih \gw + \Cpst\Cpi m_0
, \\
\Cpst(\w) & = \Cph\Mbhi \left(
I - QQ^T + 
Q\Lc_\w^{-1}Q^T
\right)\Mbh \Cph
\end{aligned}
\end{equation}
and in a neighbourhood of $[0,1]^m$, the objective functional $\Jc$, its gradient and its Hessian matrix satisfy
\begin{align}\label{eq:multiple_OED:objective}
	\Jc(\w) & = \tr(\Cp) - \tr(\Cbt) + \tr\left(
		\Lc_\w^{-1}\Cbt
	\right) \in \R, \\
	\label{eq:multiple_OED:jacobian}
	\nabla \Jc(\w) & = 
	\mobssum \left(- \left|
\Cbt^{1/2}
	\Lc_\w^{-1}R \e_k
\right|^2
	\right)_{k=1}^{m\mobs} \in \R^m, \\
	\label{eq:multiple_OED:hessian}
	H_\Jc(\w) & = \sum_{\kobs=1}^{\mobs}\sum_{\lobs=1}^{\mobs}
2\left[
(R^{[\kobs]})^T\Lc_\w^{-1}\Cbt\Lc_\w^{-1}R^{[\lobs]}\right]\odot
\left[
(R^{[\kobs]})^T\Lc_\w^{-1}R^{[\lobs]}
\right] \in \R^{m\times m}.
\end{align}

Here, $R^{[\kobs]} := (R_{i,j+(\kobs-1)m})_{i,j=1}^{i=\ell,j=m}\in\R^{\ell\times m}$ is the $k$-th slice of $R$'s columns for each $k\in\N$, $k\leq\mobs$, and $\mobssum:\R^{m\mobs}\to\R^m$ is the operation that sums each $m$-th element, in the sense that
\[
	\left(\mobssum \g\right)_k :=
	\sum_{\lobs=1}^{\mobs}\g_{k+m\lobs}
\]
for each $k\in\N$, $k\leq m$.
\end{thm}

As the proof of the above Theorem is simply a repeat of that of \tref{thm:low_rank_OED} with additional indexing and use of the chain rule, it is omitted.

It is clear that the same properties as in the single observation case apply also in the multiple observations case; in particular, $\nabla\Jc(\w)<0$ elementwise continues to hold, and computation of $\nabla \Jc(\w)$ can still be done efficiently, as one again simply computes the column norms of the $\ell\times m\mobs$-matrix $\Cbt\left(R\Mw R^T + I_\ell\right)^{-1}R$, then sums over each $m$-thm element. 

Moreover, while the Hessian in the multiple observations case becomes somewhat complicated due to the additional indexing, its application to vectors can still be computed in an efficient manner completely analogous to the single observation case. Indeed, for $\v \in\R^m$, \revia{and recalling the row-sum operator $\rowsum$ in \eqref{eq:low_rank_OED:hessian_matvec} mapping matrices to the vector whose entries are the sums over each matrix row,} one can compute

\vspace{-1ex}

\begin{equation}\label{eq:multiple_OED:hessian_matvec}
\begin{aligned}
	H_\Jc(\w)\v & = 2\mobssum\rowsum
		\left[
			R^T\Lc_\w^{-1}\Cbt\Lc_\w^{-1}R\,\Mobs[\v ]R^T
		\right]\hada
		\left[
			R^T\Lc_\w^{-1}
		\right]
\end{aligned}
\end{equation}

\vspace{-1ex}

by an argument mirroring that of the single observation case, which again requires only setting up two thin $(m\mobs)\times\ell$ matrices and computing the row- and $m$-th element sum of their elementwise products.

All computational complexities for the multiple observations scenario can be deduced in a manner completely analoguous to that of the single observation case carried out in \ssref{ssec:complexity}, noting that the dimensionality $R\in\R^{\ell\times m\mobs}$ clearly leads to replacing all instances of $m$ with $m\mobs$ in \tabref{tab:complexity}.

\section{Optimal design for an inverse source problem}\label{sec:numerics}

We now turn our attention to the numerical treatment of the A-optimality criterion for a real-world application. The main strategy is as follows: We first establish the forward operator $\Fc$ and its adjoints, fitting this to a FEM discretised framework. We then obtain a numerical solution $\w^*$ of the $1$-relaxed problem $\prob[1]$ that satisfies the optimality criteria laid out in \tref{thm:optimality}, \revix{noting which indices are dominant or redundant in the sense of \dref{def:reddom}}. As argued in the introduction, this design serves as a lower bound on all $p$-relaxed and binary optimal designs, in the sense that
\[
	\Jc(\w^*) \leq \Jc(\w) \quad \text{for all $\w \in\consK$, $p\in[0,1]$}.
\]

This allows a relative assessment of any proposed solution to $\prob$. Due to a lack of a non-convex optimality condition, it is not generally realistic to determine whether a proposed solution is a global minimum for $p<1$; however, objective values close to that of $\Jc(\w^*)$ suggest that little improvement can be made. Moreover, this property also reinforces the notion that $\w^*$ serves as a good initial guess for \revix{the continuation-type \algref{alg:p_cont}}, where we iteratively obtain more and more binary designs, seeking to approximately solve the non-convex binary optimal experimental design problem $\prob[0]$. \revib{By keeping dominant and redundant indices fixed throughout iterations, we obtain significant computational speedup, while maintaining the information yielded from the global optimum $\w^*$ via \tref{thm:optimality}.}

\begin{wrapfigure}{r}{0.35\textwidth}
\vspace{-1.8cm}
  \begin{center}
\begin{tikzpicture}
\pgfplotscolorbardrawstandalone[
	colormap/jet,    
    colorbar horizontal,
    point meta min=-1,
    point meta max=1,
    colorbar style={
        width=0.3\textwidth,
        tick style={draw=none}}]
\end{tikzpicture}
    \includegraphics[width=\figsize,keepaspectratio]{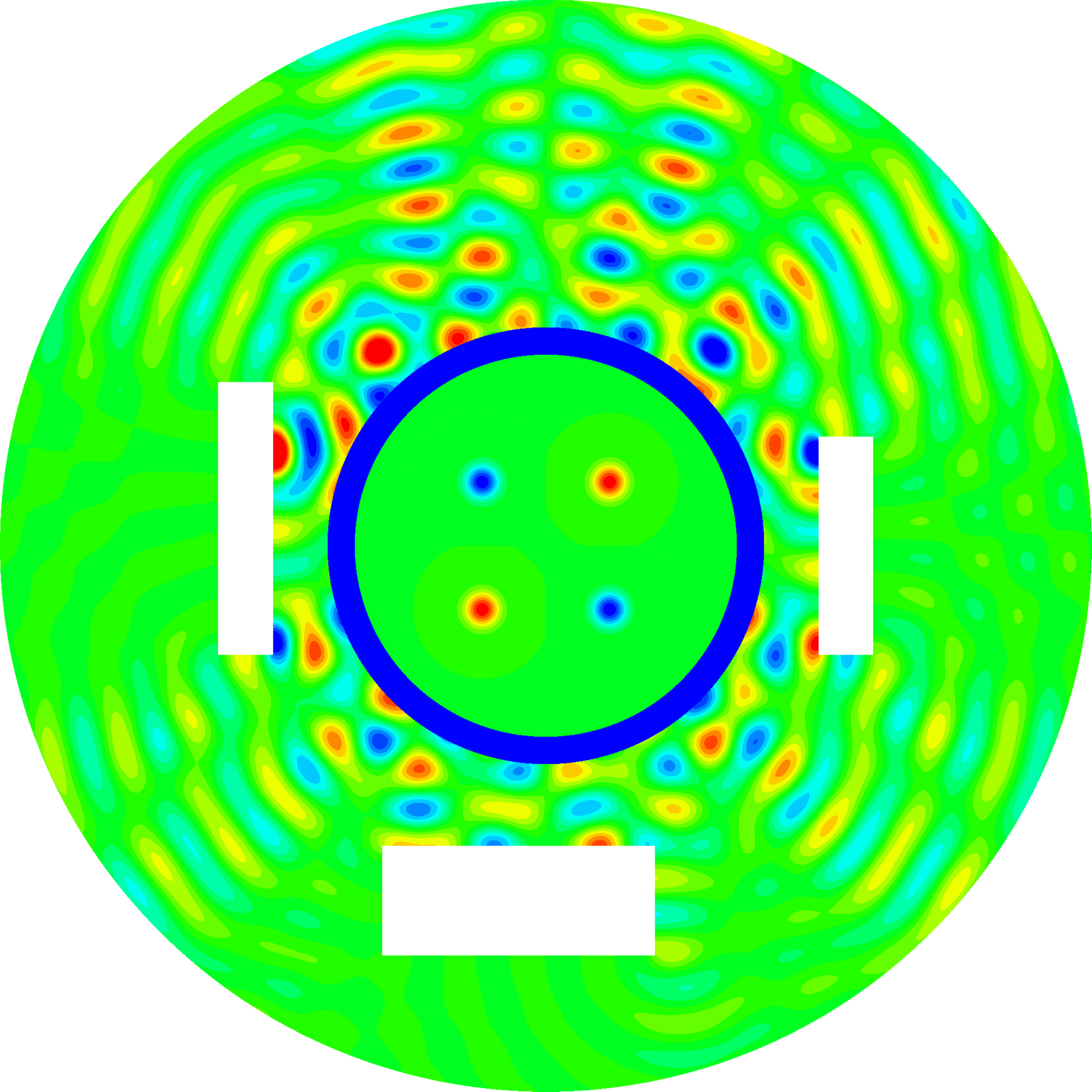}
  \end{center}
  \caption{Domain. Inner circle: Source domain $\Om$ with example source $f$ \eqref{eq:f_ground}. Outer circle: Measurement domain $\M$ with wave field $u:=\Sc_{50}f$ (rescaled to fit same color scale).}
\label{fig:domain}
\vspace{-3.2cm}
\end{wrapfigure}

\subsection{Problem setup}

We consider the forward operator $\Fc:\Lt\to\R^m$ to be of the form discussed in \sref{sec:tracefree}, that is, we write $\Fc = \Oc\circ \Sc$, where $\Sc:\Lt\to Y$ \revix{is a PDE solution operator} for a Banach space $Y$, and $\Oc:Y\to\R^{m\mobs}$ for all $u\in Y$, $k\in\N$, $k\leq m$ given as $(\Oc u)_k = \dual[Y]{o_k}{u}$ with some family of \emph{sensors} $(o_k)_{k=1}^m\subset Y^*$. When sufficient smoothness is provided by the space $Y$, as will be the case in this example, $\Oc$ may consist of pointwise evaluations, although the general case will be left open for future investigation.

%While this formulation may appear specific, it is in fact ubiquitous: The component $A$ represents the \enquote{true} forward mapping, here a PDE solution operator mapping the true source $f$ to infinite-dimensional data $u$, while $\Oc$ can describe \emph{any} bounded linear finite measurement operator by suitable choice of the $o_k$.

\paragraph{The Helmholtz equation and adjoint computation} As a well-understood linear PDE with readily available discretisation and the option to perform multiple complex-valued measurements for different frequencies, the Helmholtz equation \cite{Hel1860} is well suited as an example; we here employ it in the sense of the (linear) inverse source problem. Its connection to the wave equation makes this a reasonable simplification for e.g.~an aeroacoustic experimental setting, where the experimenter may wish to identify a sound source from pressure wave readings in sensors distributed throughout the laboratory.

To construct the associated source-to-observable map $\Fc:X\to\R^m$, we let $\Om:=\overline{B}_{0.2}(0)\subset\R^2$ be the source domain for the true solution $f\in \Lt$, and define the measurement region $\M:=\overline{B}_{1}(0)\setminus\left(B_{0.25}(0)\cup S_1\cup S_2\cup S_3\right)\subset\R^2$ with three disjoint rectangular sound-hard scatterers $S_i$, $i\in\{1,2,3\}$. Given $f\in\Lt$ and a wave number $\wav>0$, define the Helmholtz solution operator $\Sc_k:\Lt\to\HtM$ via
\begin{equation}\label{eq:helmholtz}
	  \revix{\Sc_k f := u_{|\M} \quad \text{where $u$ solves}} \qquad \left\{
	  \begin{minipage}{0.3\textwidth}
	  \vspace{-5mm}
	  \begin{alignat*}{2}
	  -\Delta u - \wav^2 u & = f &&\quad \text{in $B_1(0)\setminus\bigcup_{i=1}^3S_i$}, \\
	  \frac{\partial u}{\partial n} & = i\wav u &&\quad \text{on $\partial B_1(0)$,} \\
	  \frac{\partial u}{\partial n} & = 0 &&\quad \text{on $\bigcup_{i=1}^3\partial S_i$} 
\end{alignat*}
\end{minipage}
\right.
\end{equation}
with impedance boundary, implicitly extending $f$ by $0$ to all of $B_1(0)$. While the above equation is uniquely solvable by \cite{ColKre13}, $\Sc_\wav:\Lt\to\HtM$ is generally not invertible due to non-injectivity, unless one has data corresponding to a continuous band of wave numbers $\wav$; as we are considering a finite number of measurements here, we instead restrict ourselves to the situation where one observes several discrete frequencies to partially counteract the non-injectivity. 

As is clear from the discussions of \sref{sec:tracefree} and from the form of \algref{alg:rSVD}, we require for each fixed wave number $\wav$ both the forward operator $\Sc_\wav:\Lt\to\HtM$ and its adjoint $\Sc_\wav^*:\HtmM\to\Lt$. Indeed, it is clear that for all $v\in\HtmM$,
\begin{equation}\label{eq:helmholtz_adjoint}
	  \revix{\Sc_k^* v = \Re h_{|\Om} \quad \text{where $h$ solves}} \qquad \left\{
	  \begin{minipage}{0.3\textwidth}
	  \vspace{-5mm}
	  \begin{alignat*}{2}
	  -\Delta h - \wav^2 h & = v &&\quad \text{in $B_1(0)\setminus\bigcup_{i=1}^3S_i$}, \\
	  \frac{\partial h}{\partial n} & = -i\wav h &&\quad \text{on $\partial B_1(0)$,} \\
	  \frac{\partial h}{\partial n} & = 0 &&\quad \text{on $\bigcup_{i=1}^3\partial S_i$}
\end{alignat*}
\end{minipage}
\right.
\end{equation}
where $v$ is again extended implicitly by $0$, and only the impedance boundary has changed by conjugation. To verify this, it is enough to note that if $f\in\Lt$, $v\in\HtmM$ are arbitrary but fixed, $u:=\Sc_\wav f$ and $h$ is the solution of \eqref{eq:helmholtz_adjoint} given $v$, then

\begin{align*}
& \duall{\HtM}{\HtmM}{\Sc_\wav f}{v} = \Re \int_\M u \overline{v} \d \x = 
\Re \int_{B_1(0)\setminus\bigcup_{i=1}^3S_i} -u\Delta h - \wav^2 u h\d \x \\
= & 
\Re \left[ 
\int_{\partial B_1(0)}u \frac{\partial h}{\partial n}\d \s + 
\int_{\bigcup_{i=1}^3\partial S_i}u  \frac{\partial h}{\partial n}\d \s +
\int_{B_1(0)\setminus\bigcup_{i=1}^3S_i} \nabla u\cdot \nabla h - \wav^2 u h\d \x 
\right]
\\
= & 
\Re \left[\int_{\partial B_1(0)}-i\wav u h\d \s +
\int_{\partial B_1(0)} \frac{\partial u}{\partial n}   h \d \s +
\int_{\bigcup_{i=1}^3\partial S_i} \frac{\partial u}{\partial n}   h \d \s -
\int_{B_1(0)\setminus\bigcup_{i=1}^3S_i} \Delta u \, h + \wav^2 u h\d \x 
\right]
\\
= & 
\Re \left[
\int_{\partial B_1(0)}-i\wav u h\d \s +
\int_{\partial B_1(0)} i\wav u h \d \s +
\int_{B_1(0)\setminus\bigcup_{i=1}^3S_i} f h\d \x
\right] = 
\int_\Om f h \d \x = \inner[\Lt]{f}{h}
\end{align*}
where $u$ and $f$ are implicitly extended by $0$ to $B_1(0)\setminus\bigcup_{i=1}^3S_i$ (from resp.~$\M$ and $\Om$), proving  $h=\Sc_\wav^*v$.

%To showcase both the single observation and the multiple observations scenarios studied in \sref{sec:tracefree}, we will thus consider two settings: One where we in each candidate sensor location make only a single real-valued measurement for a single frequency, and one where we partially counteract non-injectivity by considering observations of a larger number of discrete frequencies, as a discretised attempt at capturing the behaviour when a continuous band of frequencies is available.

%\paragraph{Single observation}

%In this setting, we will make only a single observation corresponding to the frequency $\wav=50$; moreover, to strictly adhere to the presented setting, $\Fc:

%\paragraph{Multiple observations}

\paragraph{Source-to-multiple-observables map}

Following \ssref{subsec:multiple} and \tref{thm:multiple_OED}, each sensor measures the solution of the Helmholtz equation for seven different wave numbers $\wav\in\{20,25,30,35,40,45,50\}$. Separating real and imaginary components leads to $\mobs:=14$; the source-to-multiple-observables map $\Fc$ takes the operator block form
\[
	\Fc := \begin{bmatrix}
		\Fc_1 \\
		\Fc_2 \\
		\vdots \\
		\Fc_{13} \\
		\Fc_{14}	
	\end{bmatrix} := \begin{bmatrix}
		\Oc \circ \Re \circ \Sc_{20} \\
		\Oc \circ \Im \circ \Sc_{20} \\
		\vdots \\
		\Oc \circ \Re \circ \Sc_{50} \\
		\Oc \circ \Im \circ \Sc_{50}	
	\end{bmatrix} : \Lt\to\R^{m\mobs},
\]
where $\Oc:\HtM\to\R^m$ is the pointwise measurement operator given via $\Oc u:=\left((\dl_{\x_k},u)\right)_{k=1}^m=\left(u(\x_k)\right)_{k=1}^m\in\R^m$ for each $u\in\HtM$, with $m$-dependent grid $(\x_k)_{k=1}^m\subset\M$ specified below; as is easily verified, $\Oc^*:\R^m\to\HtmM$ is given as $\Oc^*\g = \sum_{k=1}^m\g_k\dl_{\x_k}$ for all $\g\in\R^m$.

\paragraph{Pointwise variance}

A valuable tool in assessing the quality of a design $\w$ is the induced \emph{pointwise variance} $c_\w$, which can be thought of as the Green's function corresponding to $\Cpst(\w)^{-1}$, see \cite[Subsection 3.7]{BThGhaMarSta13} for further details and computational techniques; that is, $\Cpst(\w)c_\w (x,y)=\delta_x(y)$ for a.e.~$x$, $y\in\Omega$. Mercer's theorem implies that under certain assumptions, $\tr(\Cpst(\w))=\int_\Om c_\w (x,x)\d x$; thus, the A-optimality of a design $\w$ directly relates to the average size of $c_\w$ over the source domain $\Om$.

\paragraph{Prior and noise} To set up the stochastic inverse problem \eqref{eq:forward} and its solution \eqref{eq:bayesian_inversion}, we make a choice of prior covariance $\Cp:\R^n_\Mb\to\R^n_\Mb$ by discretising the bilaplacian $\left(\alpha\Delta + I\right)^{-2}$, $\alpha = 0.01125$ on $\Lt$, with Robin boundary condition $\tfrac{\partial u}{\partial n} = \beta u$, $\beta = \tfrac{\sqrt{\alpha}}{1.42}$; for a discussion on the effects of this boundary condition and the associated parameter choice, we refer to \cite{YaiSta18,VilOLe24}. Here, we only comment that it is known to  improve spatial uniformity of the prior pointwise variance. It is precisely this prior pointwise variance that contributes to the A-optimality when no sensors are placed, as $\Jc(\mathbf{0})=\tr(\Cph\left(\mathbf{0} + I_\ell\right)^{-1}\Cph) = \tr(\Cp)$. %Spatial uniformity of the prior pointwise variance thus means we have no prior information that weighs any particular region 

The noise covariance was set as $\Gmn=\sigma^2 I$, where the scalar $\sigma>0$ is chosen proportional to $1\%$ of the average pointwise variance in $10^3$ i.i.d.~samples of data drawn from the prior covariance; that is, if $(s^{(i)})_{i=1}^{1000}\subset\R^n$ are i.i.d.~Gaussians with mean $0$ and pointwise variance $1$, then $\sigma^2:=\frac{0.01^2}{1000}\sum_{i=1}^{1000}\sum_{k=1}^m|(\Fb s^{(i)})_k|^2$.

\begin{wrapfigure}{l}{0.35\textwidth}
\vspace{-0.84cm}
\begin{center}

\pgfplotscolorbardrawstandalone[
	colormap/jet,    
    colorbar horizontal,
    point meta min=0,
    point meta max=\posterioronemax,
    colorbar style={
        width=0.3\textwidth,
        /pgf/number format/fixed,
        /pgf/number format/precision=2,
        xticklabel style={anchor=north},
        every axis label/.append style={/pgf/number format/precision=3},  % Ensure proper precision for the tick labels
        yticklabel style={/pgf/number format/none},
        xtick distance=\posterioronemax/4, 
        tick style={draw=none}}]
        
    \includegraphics[width=\figsize,keepaspectratio]{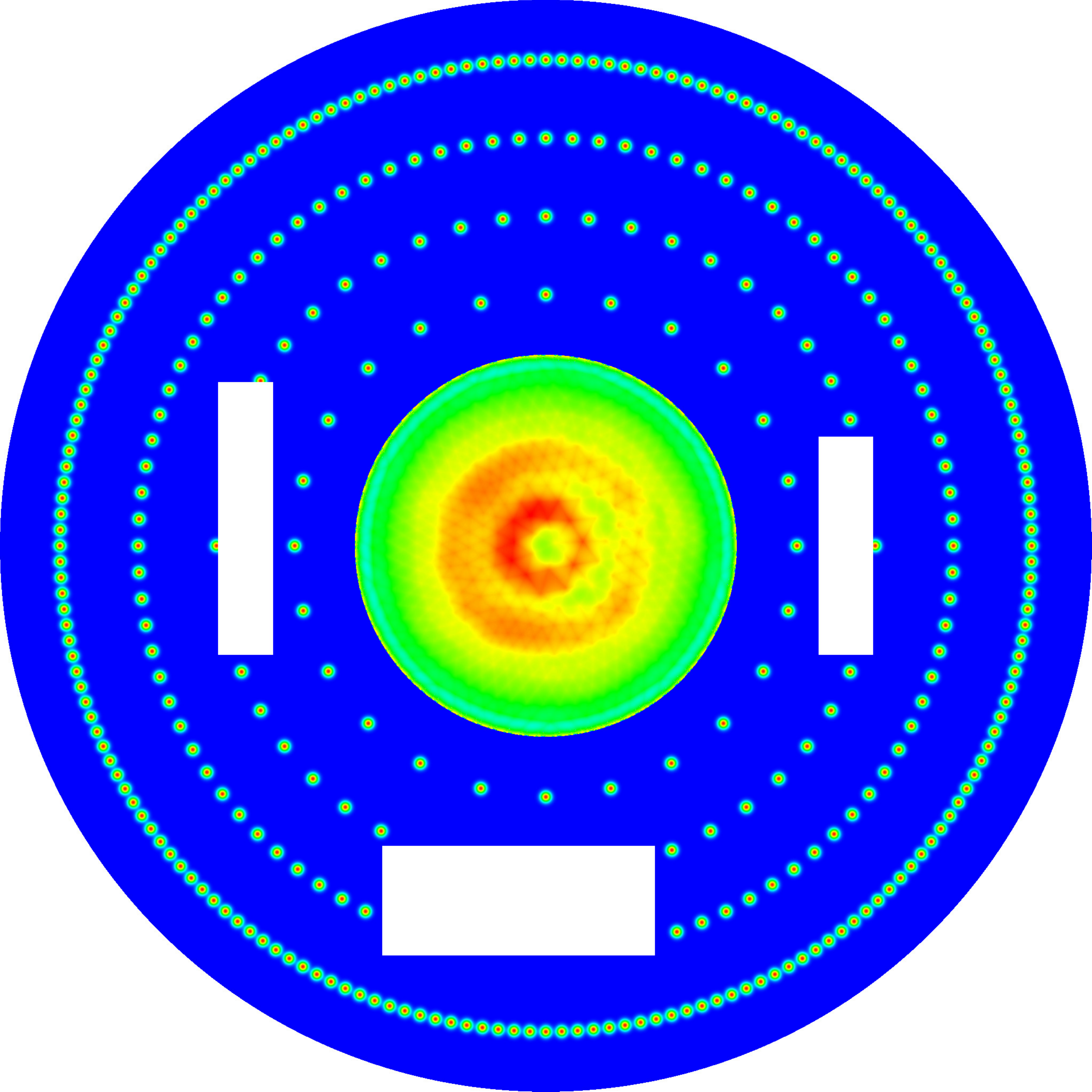}
     \end{center}
  \captionof{figure}{Outer circle: Circular grid of all candidate sensor locations. Inner circle: Pointwise variance field corresponding to placing every sensor.}
  \label{fig:grid}
  \vspace{-1cm}

\end{wrapfigure}

\paragraph{Discretisation and implementation}

Discretisation was carried out as described in \sref{sec:tracefree}, specifically by use of the \texttt{NGSolve} Python package \cite{Sch14, NGS} to construct FEM discretisations $\R^n_{\Mb}$ of $\Lt$ with varying $n\in\N$ degrees of freedom, as described in \ssref{ssec:FEM}, with FEM-discretised covariance operator $\Cp:\R^n_\Mb\to\R^n_\Mb$.

This discretisation allows us to solve the Helmholtz equations \eqref{eq:helmholtz}--\eqref{eq:helmholtz_adjoint} by employing a complex $H^1$-conforming second-order finite element space with $\nco=21409$ degrees of freedom; measurements of the wave field $u$ were also carried out on this discretisation, employing \texttt{NGSolve}'s in-built pointwise evaluation function to assemble the sparse observation matrix $\Oc$ and its numerical adjoint, i.e.~$\Mbco^{-1}\Oc^T$, with $\Oc^T$ denoting the transpose matrix and $\Mbco$ the mass matrix of the $H^1$-conforming finite element space, following the reasoning in \sref{sec:tracefree}. %Meanwhile, the adjoint equation \eqref{eq:helmholtz_adjoint} was implemented as-is, i.e.~not via setting up the numerical adjoint.

\figref{fig:domain} illustrates the domain, and graphically shows an example ground truth $f$ (c.f.~\eqref{eq:f_ground}) and the resulting scattered wave $u=\Sc_{50} f$. Visualisation for this Figure, as well as for various subsequent Figures, was performed via \texttt{NGSolve}'s inbuilt \texttt{Draw} method.

%\paragraph{Sensor grid} A uniform grid over $[0,1]^2$ is chosen such that its intersection with $(\M\cap B_{0.95}(0))\setminus B_{0.35}(0)$ has $m:=335$ elements. %With this, we can define the radius $r$ and scaling constant $c$ via 
%\[
%	r:=\tfrac{60}{19}\min\limits_{k\neq k'}|(x_k-x_{k'},y_k-y_{k'})|^2, %\qquad c:=\left(
%\int_{B_{r}(0)}e^{-r(x^2+y^2)}\d x\d y\right)^{-1/2}.
%\]

%For each experiment, $m$ is chosen in the following manner: Given an initial target $m'\in\N$ for $m$, we fix the list of candidate sensor locations $(x_k)_{k=1}^{m}\subset\R^2$ by taking the coarsest uniform grid over $[0,1]^2$ such that its intersection with $(\M\cap B_{0.95}(0))\setminus B_{0.35}(0)$ has at least $m'$ elements, setting $m$ to be equal to the number of elements in this intersection. %With this, we can define the radius $r$ and scaling constant $c$ via 
%\[
%	r:=\tfrac{60}{19}\min\limits_{k\neq k'}|(x_k-x_{k'},y_k-y_{k'})|^2, %\qquad c:=\left(
%\int_{B_{r}(0)}e^{-r(x^2+y^2)}\d x\d y\right)^{-1/2}.
%\]

\paragraph{Candidate sensor locations and low-rank approximation} 

A total of $m:=334$ candidate sensor locations $\{\x_k\}_{k=1}^m\subset\M$ were chosen along four concentric circles, see \figref{fig:grid}; each sensor, while technically corresponding to observation by pairing with $o_k:=\delta_{\x_k}$, is graphically represented by a Gaussian-like dot. Note that \emph{all} graphical representations of the sensor grid, as well as the subdomain they inhabit, always ignore any colorbars or scaling, and always take values between $1$ (visible red dot) and $0$ (no sensor visible), with faint dots representing values in $(0,1)$ for non-binary designs.

The low-rank approximations in \sref{sec:tracefree} were computed with $\ell=217$ via the randomised SVD \algref{alg:rSVD} upfront, as it does not change over the course of the experiment. This value of $\ell$ was chosen such that the smallest eigenvalue of the output SVD was six orders of magnitude smaller than the largest eigenvalue; a similar reasoning was employed in \cite{AlePetStaGha14} on the basis of the previously discussed error analysis. % over approximately 35 minutes\footnote{All computation times given for a 12th Gen Intel(R) Core(TM) i5-12500H (4.50 GHz) processor with 16 cores.} 

\subsection{Calculated optimal experimental designs}

For each $m_0\in\N$, $m_0\leq 36$, $\prob[1]$ was solved using SciPy's implementation of the SLSQP algorithm \cite{SLSQP}, as it allows for efficient solutions of nonlinear convex problems with bounds and inequality constraints, which were given explicitly as
\[
	0 \leq \w_k \leq 1 \qquad \text{for all $k\in\N$, $k\leq m$}, \qquad
	\sum_{k=1}^m \w_k \leq m_0.
\]
The resulting design $\w^*$ and its gradient $\nabla\Jc(\w^*)$ are shown for a selection of $m_0$ in \figref{fig:w_star_1}; the index ordering of \tref{thm:optimality} has already been applied, and the adherence of the optimal design to the ordering of the gradient suggestested by this Theorem is remarkable. Indeed, the Theorem predicts exactly that for all sensors taking values between but not equal to $0$ and $1$, the gradient must be constant, while for all sensors whose corresponding gradient is strictly lower resp.~higher than this constant value, the sensor must be exactly equal to $1$, resp.~$0$.

%\vfill

\begin{figure}[h]
\begin{center}
\begin{tikzpicture}
\begin{groupplot}[scale=0.89,
            group style={
                group size=2 by 1,
                horizontal sep=45pt
            },
]

\nextgroupplot[
	ymin = 0,
	ymax = 1,
	enlargelimits = 0.05,
	xlabel = {index $k$}, 
	ylabel = {sensor $\w^*_k$},
	ylabel shift = -3pt]
	
\addplot[thick, color=col2,mark=square,mark size=0.6pt,mark repeat=6] table[x=ind,y=w, col sep=comma]{graphics/w0_36_dom.csv};
\addlegendentry{Dominant indices}
\addplot[thin, color=col3] table[x=ind,y=w, col sep=comma]{graphics/w0_36_free.csv};
\addlegendentry{Free indices}
\addplot[color=col5,mark=triangle,mark size=0.5pt,mark repeat=6] table[x=ind,y=w, col sep=comma]{graphics/w0_36_red.csv};
\addlegendentry{Redundant indices}

\nextgroupplot[
	ymax = 0,
	enlargelimits = 0.05,
	legend pos = south east,
	xlabel = {index $k$}, 
	ylabel = {gradient $(\nabla\Jc(\w^*))_k$},
	ylabel shift = -3pt]

\addplot[thick, color=col2,mark=square,mark size=0.8pt,mark repeat=3] table[x=ind,y=jac, col sep=comma]{graphics/w0_36_dom.csv};
\addlegendentry{Dominant indices}
\addplot[thin, color=col3] table[x=ind,y=jac, col sep=comma]{graphics/w0_36_free.csv};
\addlegendentry{Free indices}
\addplot[color=col5,mark=triangle,mark size=0.5pt,mark repeat=6] table[x=ind,y=jac, col sep=comma]{graphics/w0_36_red.csv};
\addlegendentry{Redundant indices}

\end{groupplot}
\end{tikzpicture}
\captionof{figure}{Left: $1$-relaxed optimal design $\w^*$ using $m_0=36$ out of $m=334$ sensors. Right: corresponding gradient. Dominant indices ($\w^*_k=1$, resp.~large negative gradient) as green squares, redundant indices ($\w^*_k=0$, resp.~small negative gradient) as cyan triangles, free indices red.}\label{fig:w_star_1}
\end{center}
\end{figure}
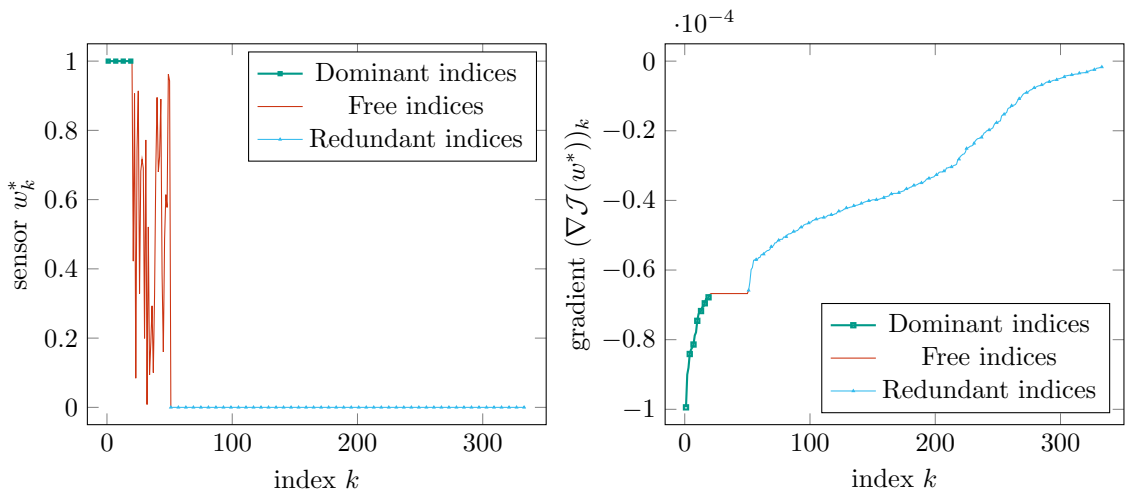

\revib{Next, we applied the $p$-continuation algorithm \algref{alg:p_cont}, using the global non-binary optimum $\w^*$ as a starting point towards obtaining an approximate binary optimum of $\prob[0]$. It was found that taking a number of smaller refinement steps ($\delta\in[10^{-2},10^{-1}]$ in \algref{alg:p_cont}) yielded good results, with each step resolving quickly.

Due to the non-convex nature of the solution steps of \algref{alg:p_cont} for $p<1$, it cannot be ascertained whether the global optima of $\prob$ become unreachable; as such, a careful comparison of \algref{alg:p_cont} to the original problem $\prob$ would be of great interest. For the scope of this article, we instead offer a comparison of the outputs of \algref{alg:p_cont} to randomly drawn designs, where we draw $10^3$ randomly chosen designs for each $m_0$ and note their A-optimality, as well as with the A-optimality of the non-binary global optima $\w^*$.}

\begin{figure}[H]
\centering
\begin{tikzpicture}[scale=1.5]
\begin{axis}[
	no markers, 
	xmin = 7, 
	xmax = 36, 
	ymax = 0.16,
	enlargelimits = 0.01, 
	xlabel = {target number of sensors $m_0$}, 
	xlabel style = {font = \large}, 
	ylabel = {A-optimality $\Jc$}, 
	ylabel style = {font = \large}, 
	ylabel shift=-20pt, 
	xlabel shift=-5pt, 
	xshift=-12pt,
	legend style={at={(0.275,0.47)},anchor=south west}
]

\addplot[color=red!40, name path=maxes, forget plot] table[x=targets,y=randommax, col sep=comma]{graphics/Aoptimalities.csv};

\addplot[color=red!40, name path=mins, forget plot] table[x=targets,y=randommin, col sep=comma]{graphics/Aoptimalities.csv};

\addplot[color=OrangeRed!40]fill between[of=maxes and mins];

\addlegendentry{Random designs}

\addplot[very thick, color=ForestGreen, name path=ws] table[x=targets,y=w, col sep=comma]{graphics/Aoptimalities.csv};

\addlegendentry{\algref{alg:p_cont} output $\wpcont$}

\addplot[very thick, color=blue, name path=w0s] table[x=targets,y=w0, col sep=comma]{graphics/Aoptimalities.csv};

\addlegendentry{Global non-binary optimum $\w^*$}

\end{axis}
\end{tikzpicture}
\caption{Comparison of the A-optimal objective of outputs of \algref{alg:p_cont} (green) vs.~$10^3$ random designs (red) vs.~the globally optimal non-binary designs $\w^*$ (blue) for $7\leq m_0\leq 36$.}\label{fig:comps}
\end{figure}

\figref{fig:comps} displays these comparisons, and for $m_0\geq 8$ shows that \algref{alg:p_cont} outperforms random designs by a significant margin, while for $m_0\geq 24$ being almost exactly as good as the globally optimal non-binary design $\w^*$. \revia{As A-optimalities for $m_0<7$ are significantly higher, for both algorithm outputs $\wpcont$ and for random designs, with values of $\Jc$ ranging from around $2.13$ to $0.25$, and since reconstruction in the inverse problem is poor for these lower targets, they are not displayed in the above graph. We do note that in the range $m_0\leq 7$, excluding the very low-performing case $m_0=1$, the algorithm outputs $\wpcont$ on average perform approximately $9\%$ worse than the best random design, while outperforming approximately $96\%$ of all random designs.} \label{sentence:low_performance} Regarding worse performance for $m_0\leq 7$, we take the perspective that this is due to the smaller number of possible configurations for lower values of $m_0$, i.e.~higher likelihood that a near-optimal binary design is found by chance, as it is closer to an exhaustive binary search. Conversely, this proves that while in a large regime, \algref{alg:p_cont} finds extremely good designs, it cannot be expected to equal the optimal binary design, on account of the non-convex optimisation involved in $\prob$. Meanwhile, \figref{fig:designs_various} displays a selection of designs $\wpcont$ as visualised in the sensor grid on the measurement domain.

\pagebreak

\vfill

\begin{figure}[H]
   
\centering

\pgfplotscolorbardrawstandalone[
	colormap name = log_jet,    
    colorbar horizontal,
    point meta min=0,
    point meta max=\posteriorzeromax,
    colorbar style={
        width=0.3\textwidth,
        /pgf/number format/fixed,
        /pgf/number format/precision=2,
        xticklabel style={anchor=north},
        every axis label/.append style={/pgf/number format/precision=3},  % Ensure proper precision for the tick labels
        yticklabel style={/pgf/number format/none},
        xtick distance=\posteriorzeromax/4, 
        tick style={draw=none}}]
\\[1.5ex]

\includegraphics[width=\figsize,keepaspectratio]{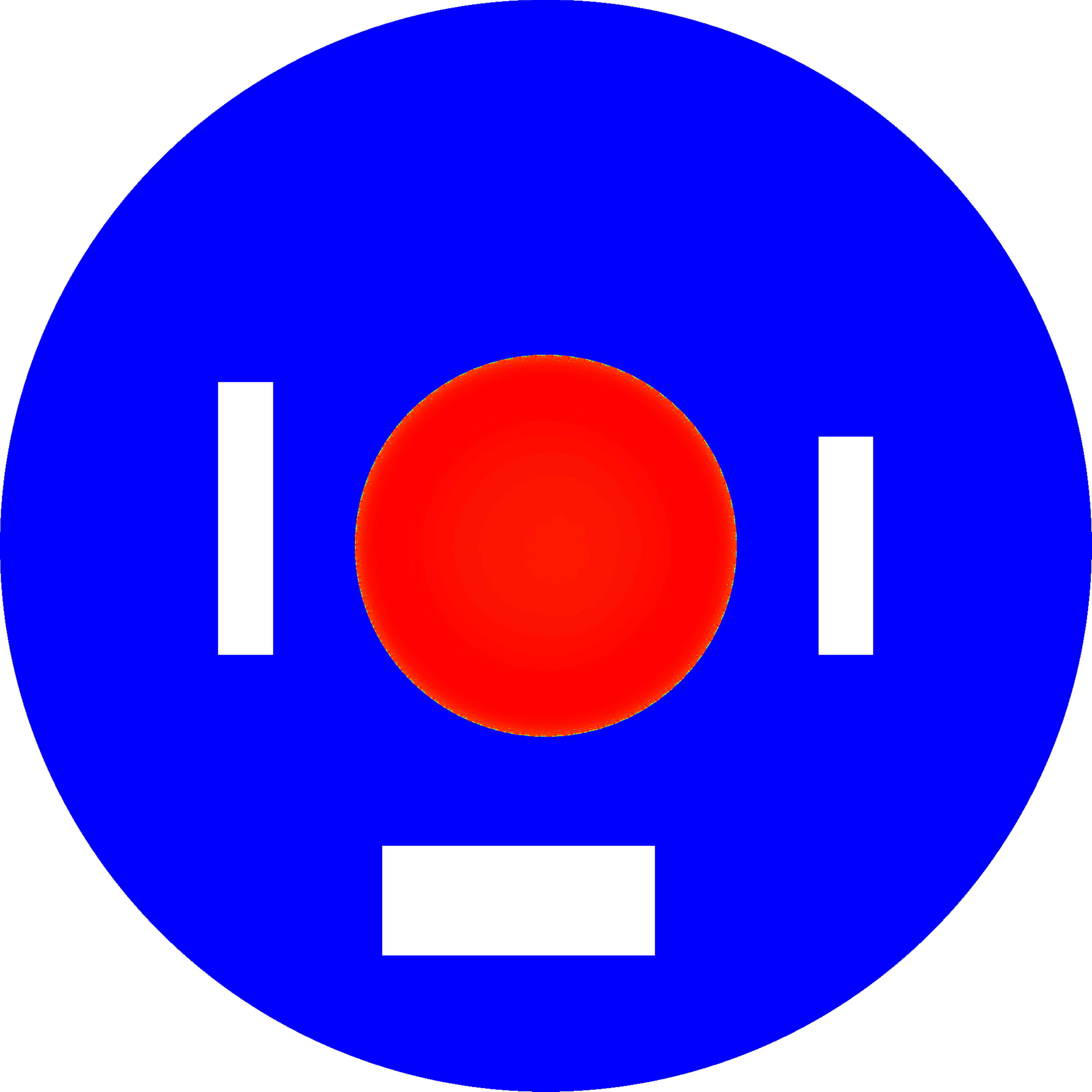}\hfill
\includegraphics[width=\figsize,keepaspectratio]{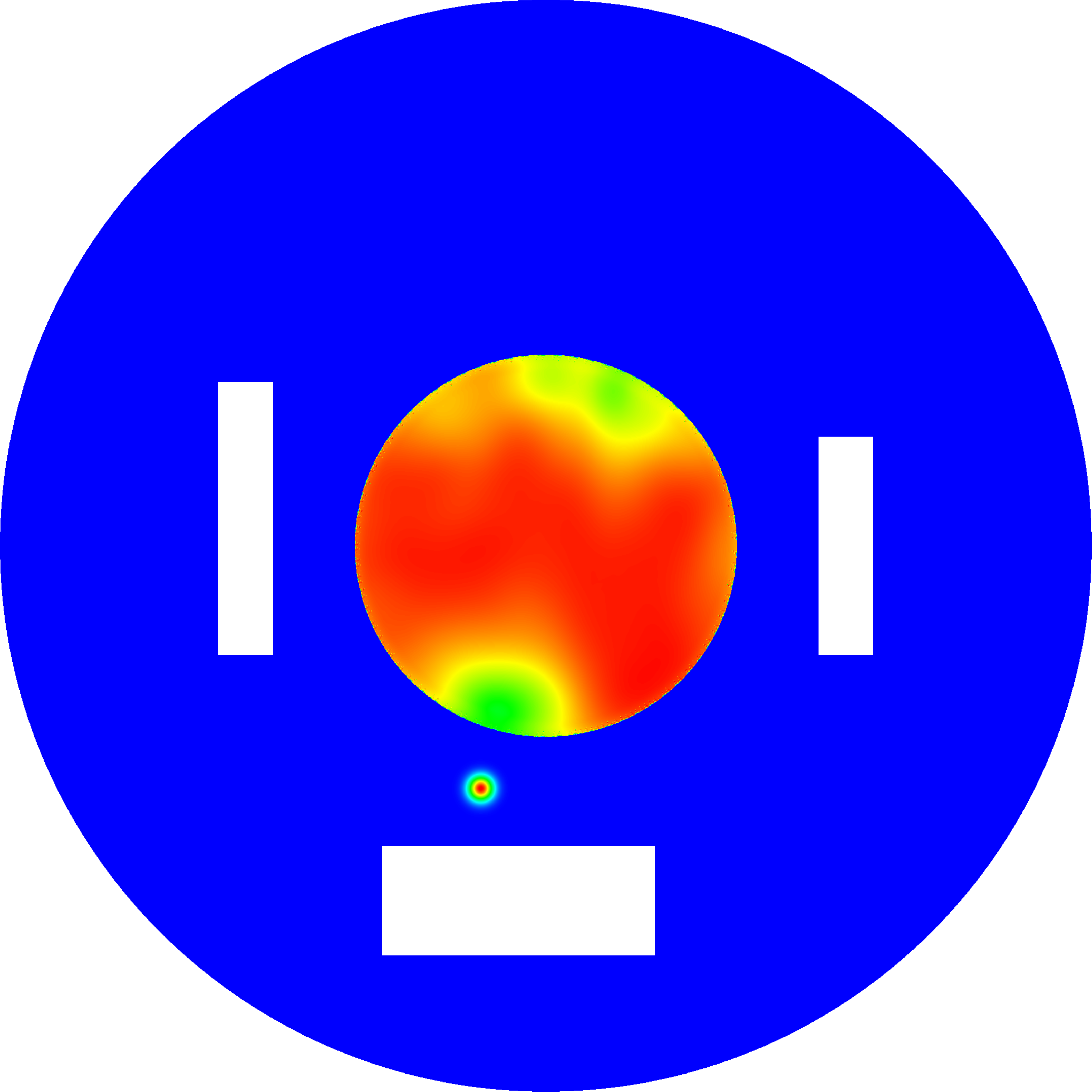}\hfill 
\includegraphics[width=\figsize,keepaspectratio]{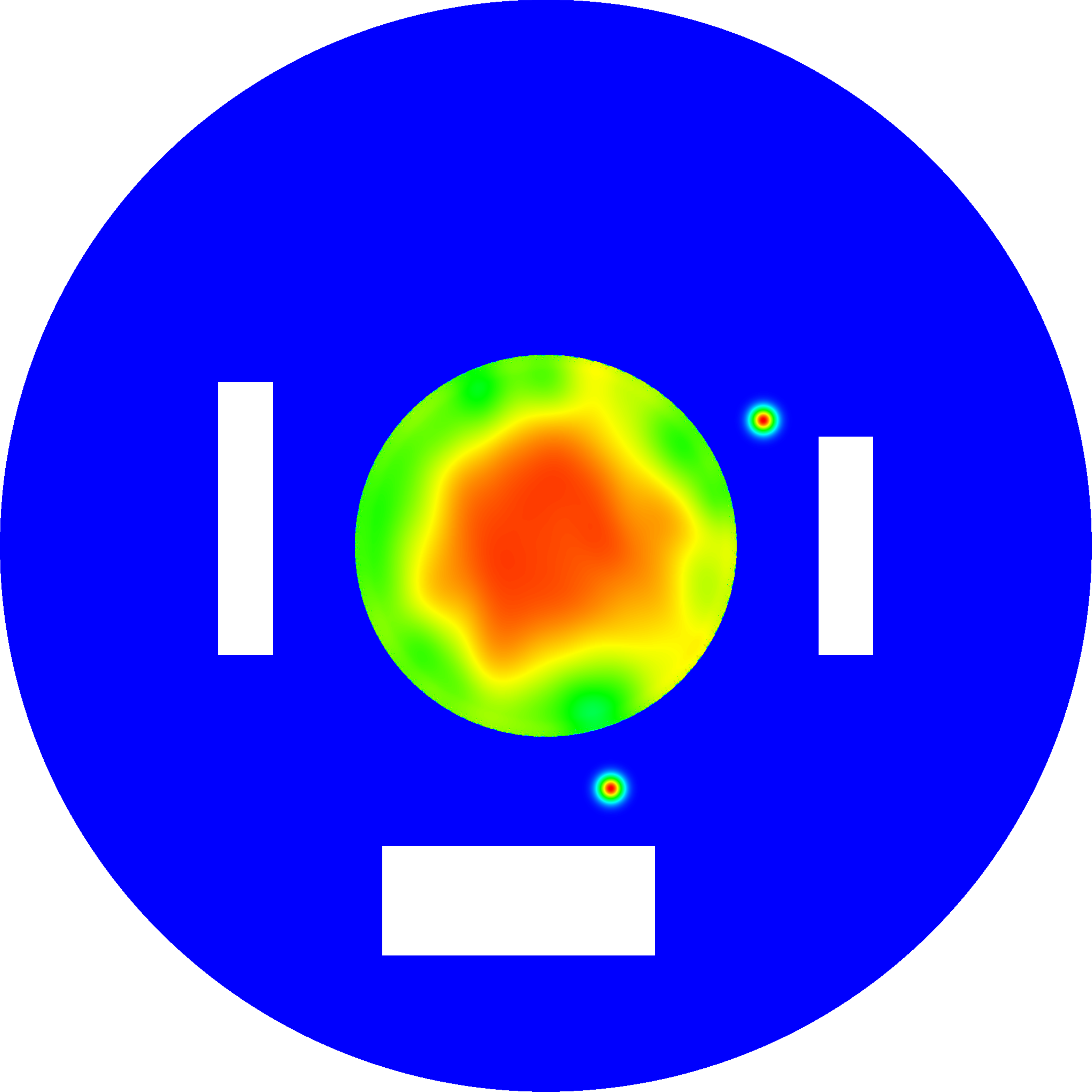} \\[1ex]
\includegraphics[width=\figsize,keepaspectratio]{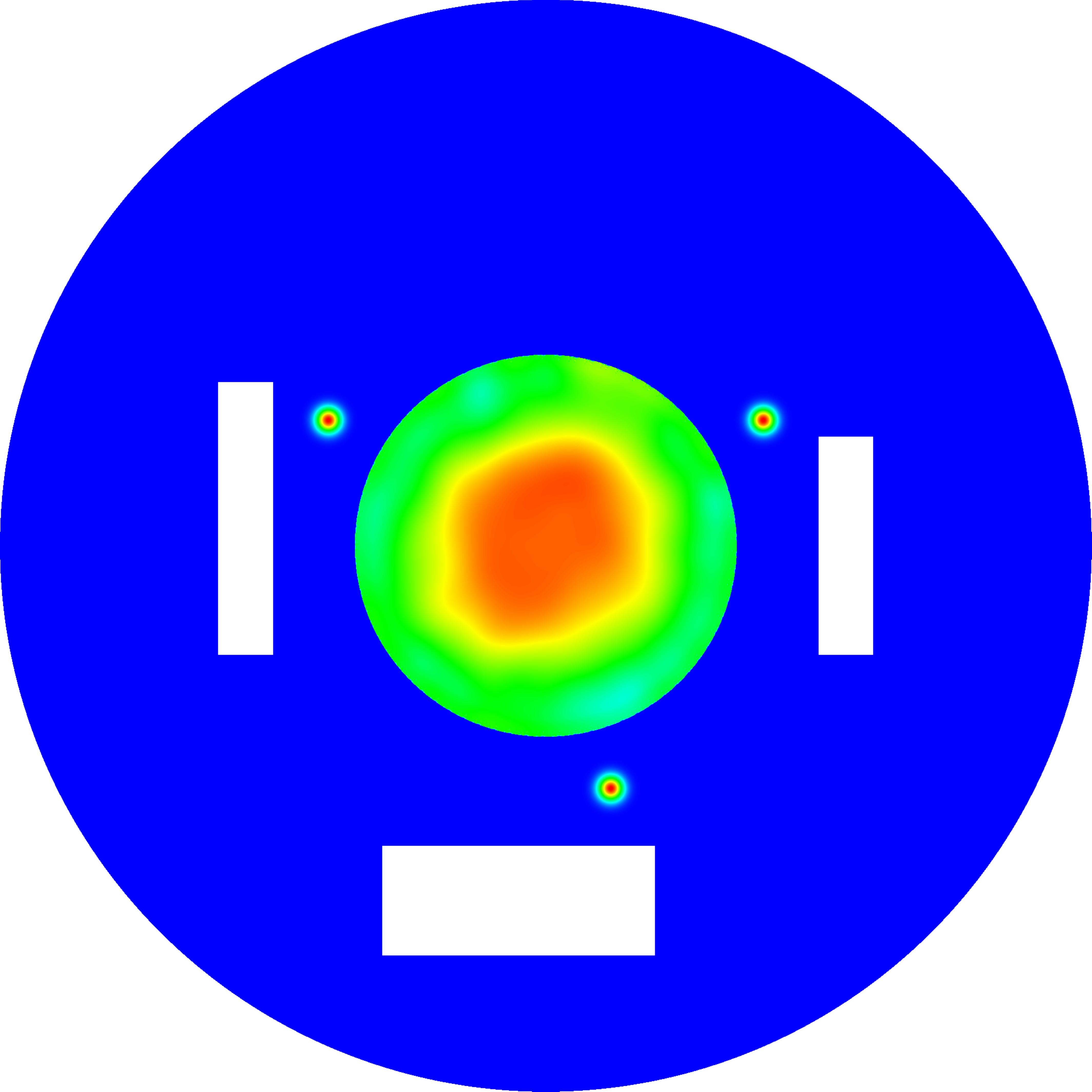}\hfill
\includegraphics[width=\figsize,keepaspectratio]{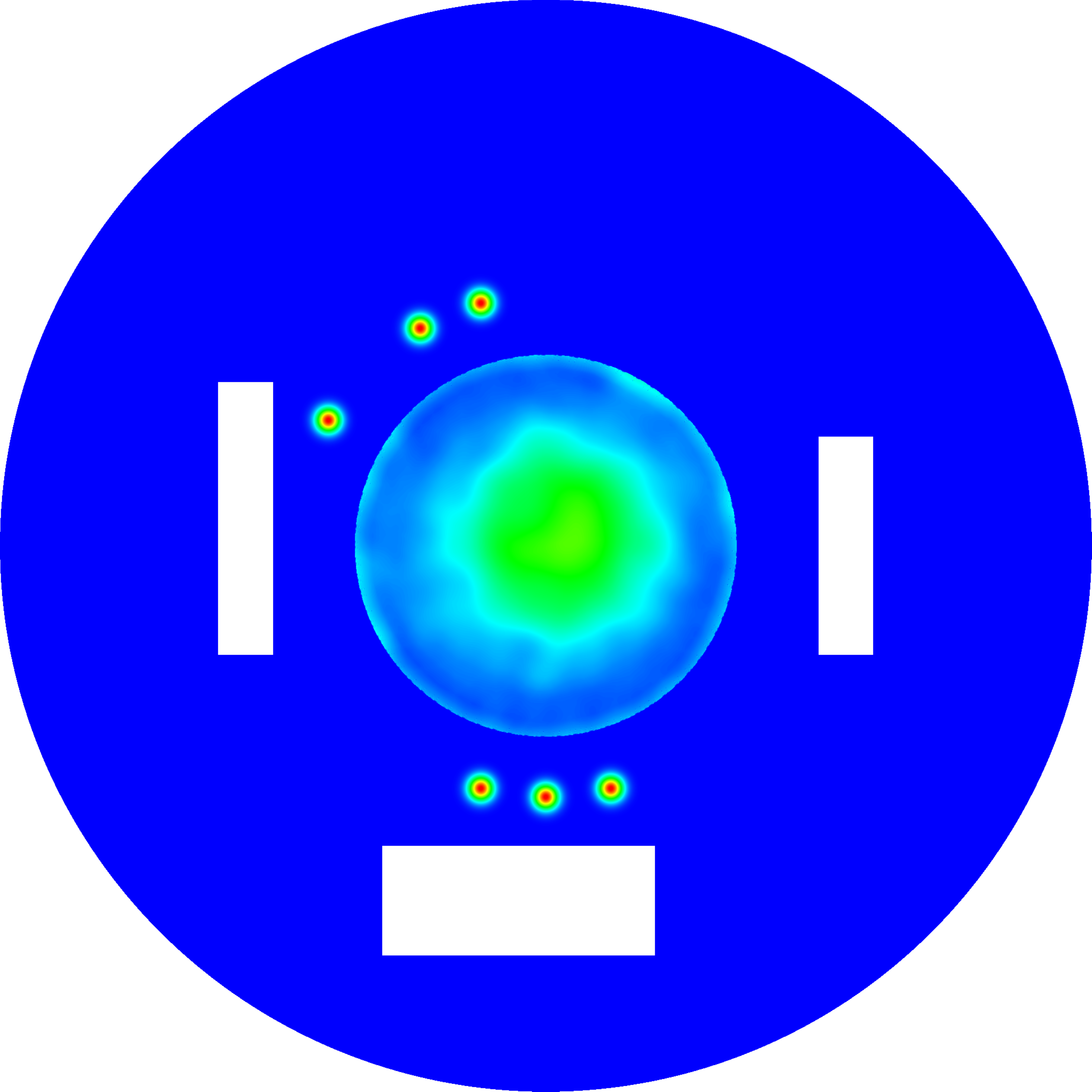}\hfill 
\includegraphics[width=\figsize,keepaspectratio]{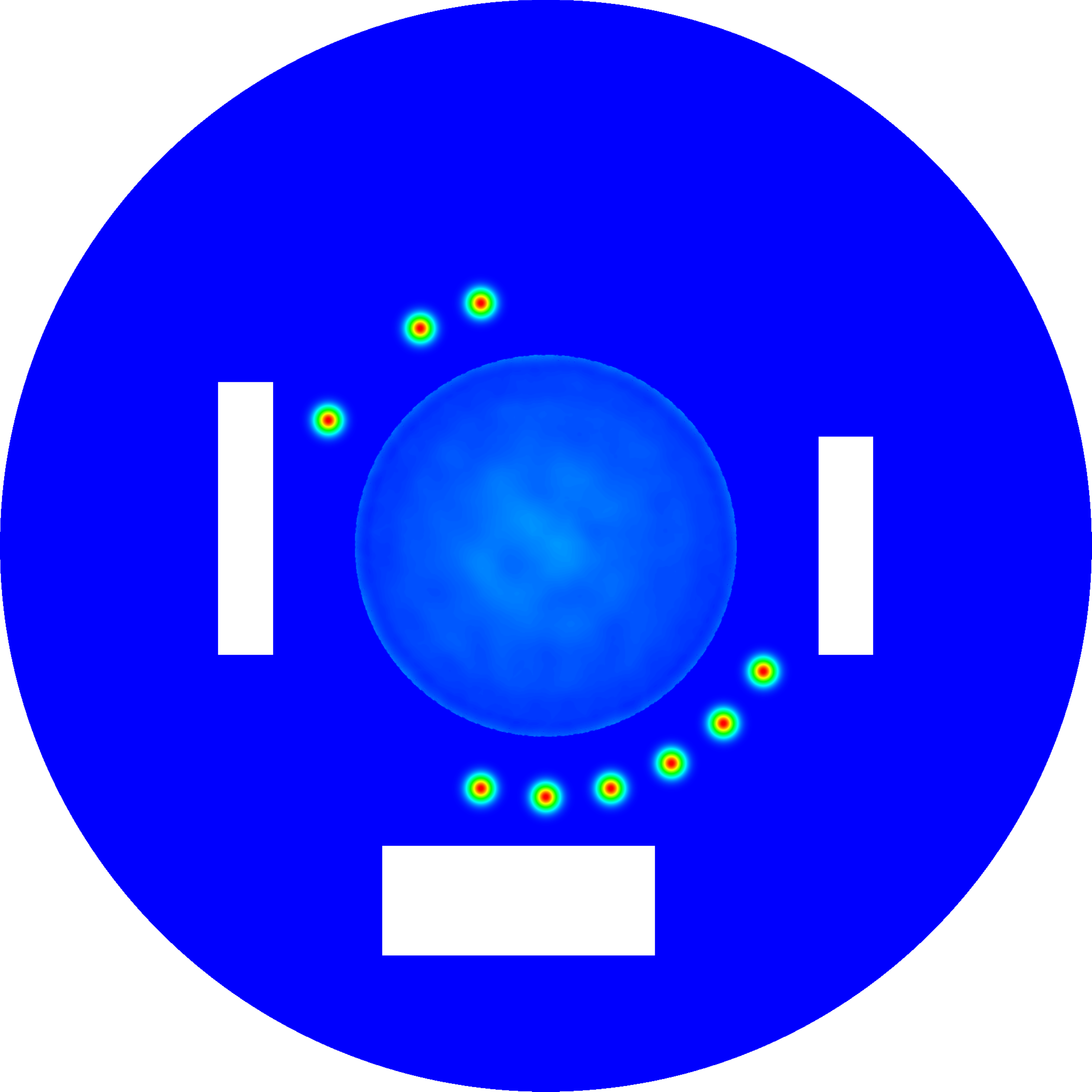} \\[1ex]
\includegraphics[width=\figsize,keepaspectratio]{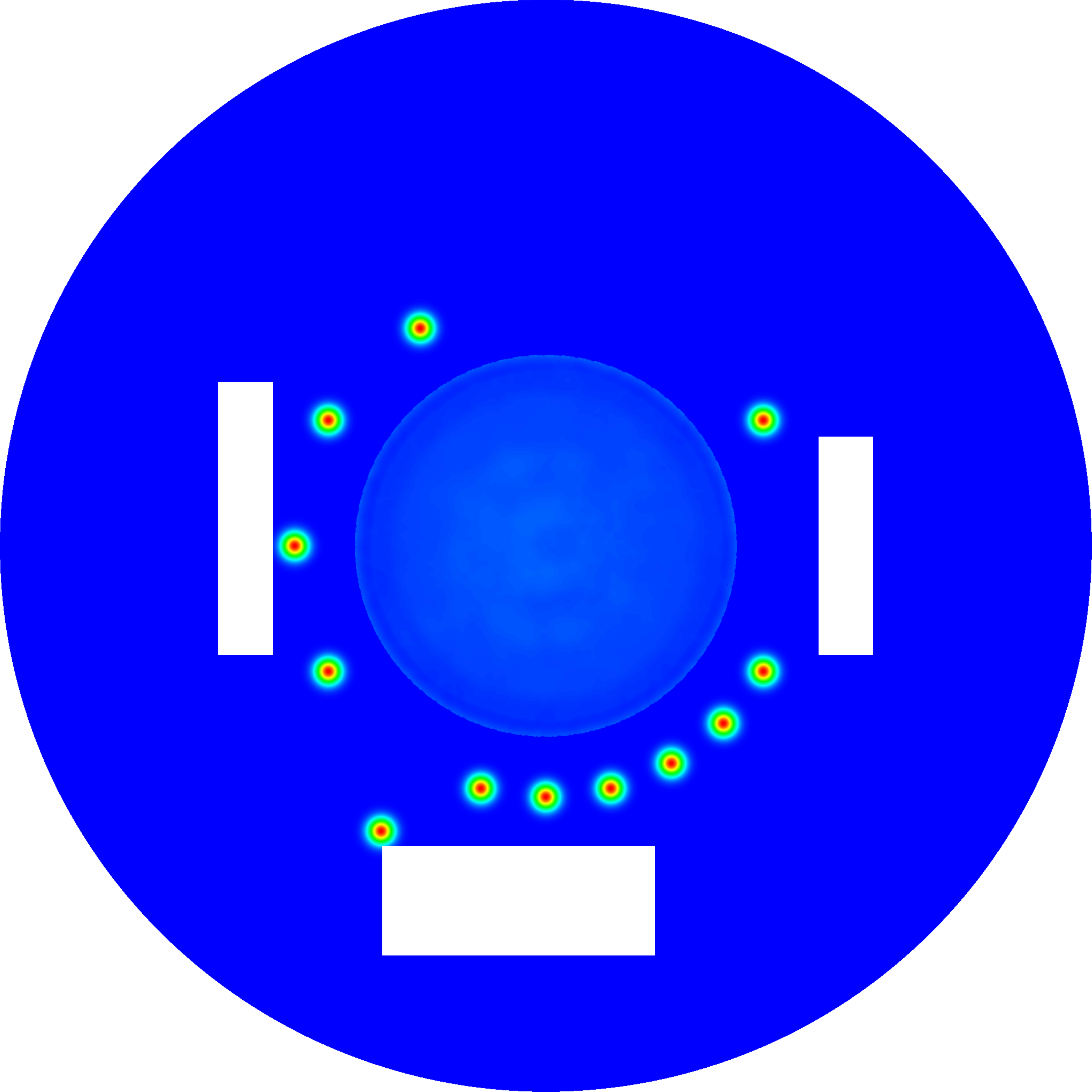}\hfill
\includegraphics[width=\figsize,keepaspectratio]{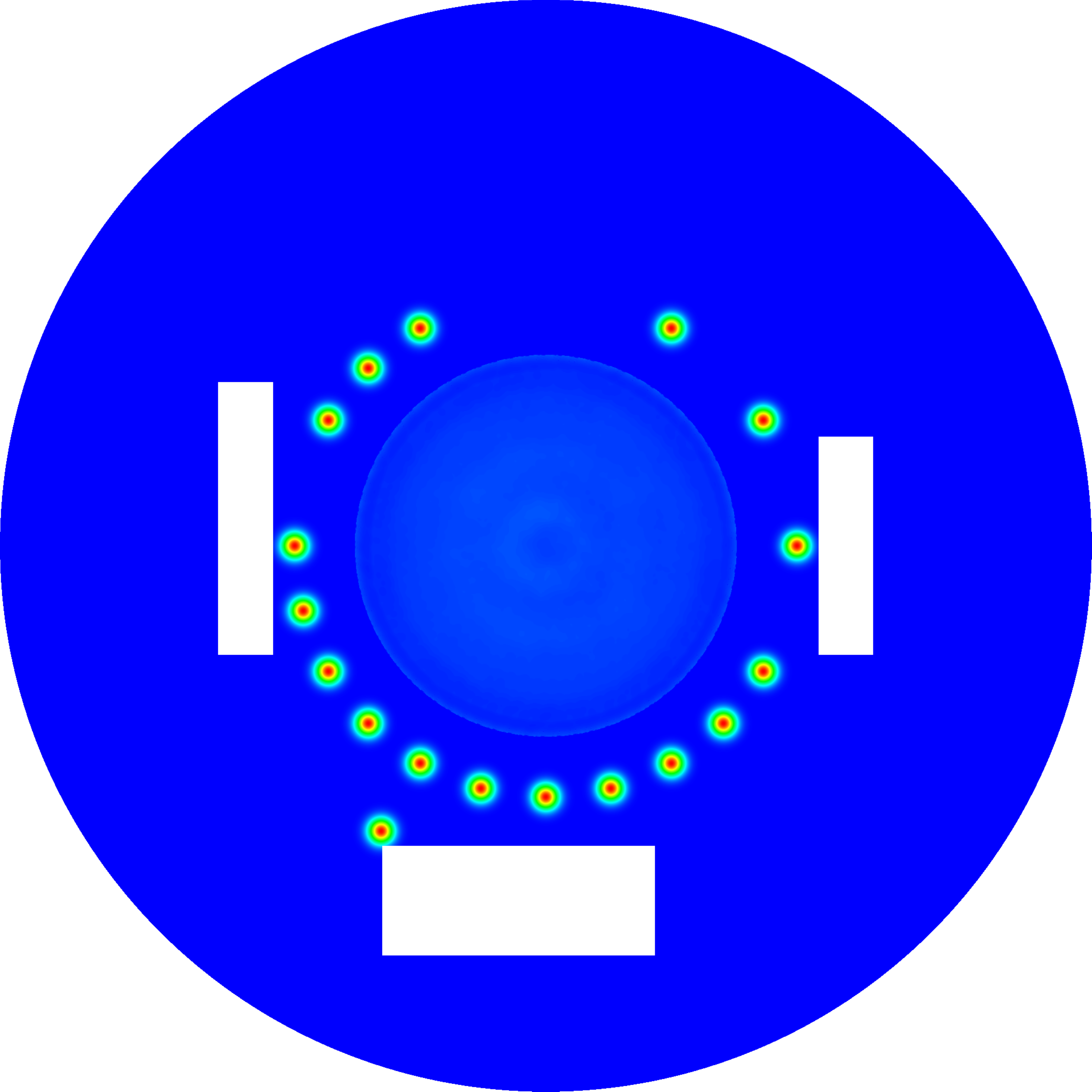}\hfill 
\includegraphics[width=\figsize,keepaspectratio]{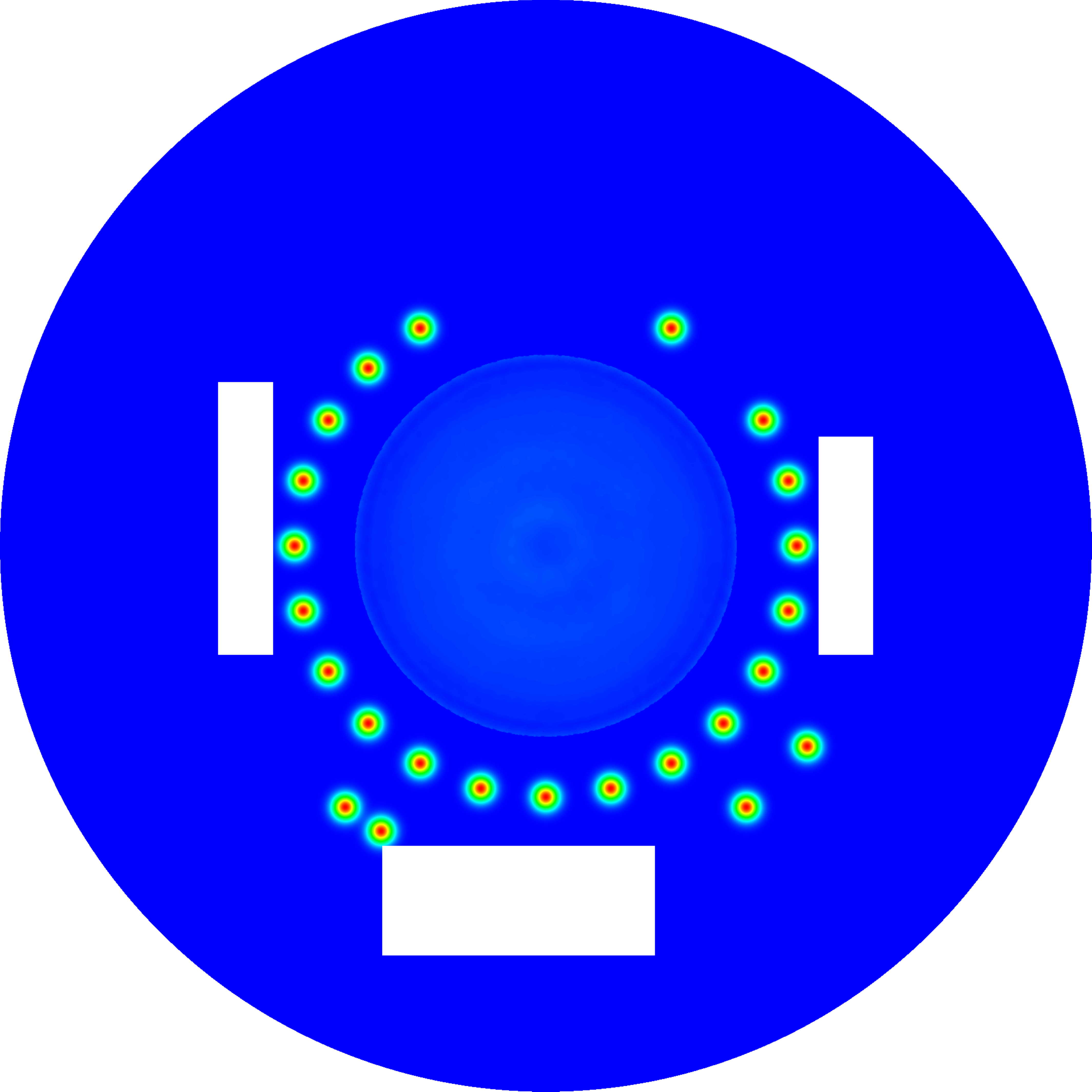}

\vspace{0.25ex} % Avoids caption eating the bottom of the circle

\caption{Outputs of \algref{alg:p_cont} (outer rings) and corresponding pointwise variance fields (inner rings; top left corresponds to prior $\Cp$). Top left to bottom right: $m_0\in\{0,1,2,3,6,9,12,18,24\}$. Note the log-like scaling of the colorbar.}\label{fig:designs_various}
\end{figure}

\vfill

\pagebreak

%\pagebreak

\vfill

\begin{figure}[H]
\begin{center} % The negative hspace helps the plots align horizontally
\hspace{-42.5pt}\pgfplotscolorbardrawstandalone[
	colormap/jet,    
	%colorbar horizontal,    
    point meta min=0,
    point meta max=\posteriortwentyfourmax,
    colorbar style={
        height=0.3\textwidth,
        /pgf/number format/fixed,
        /pgf/number format/precision=2,
        every axis label/.append style={/pgf/number format/precision=3},  % Ensure proper precision for the tick labels
        xticklabel style={/pgf/number format/none},
        ytick distance=\posteriortwentyfourmax/4, 
        tick style={draw=none},
        }]
\includegraphics[width=\figsize,keepaspectratio]{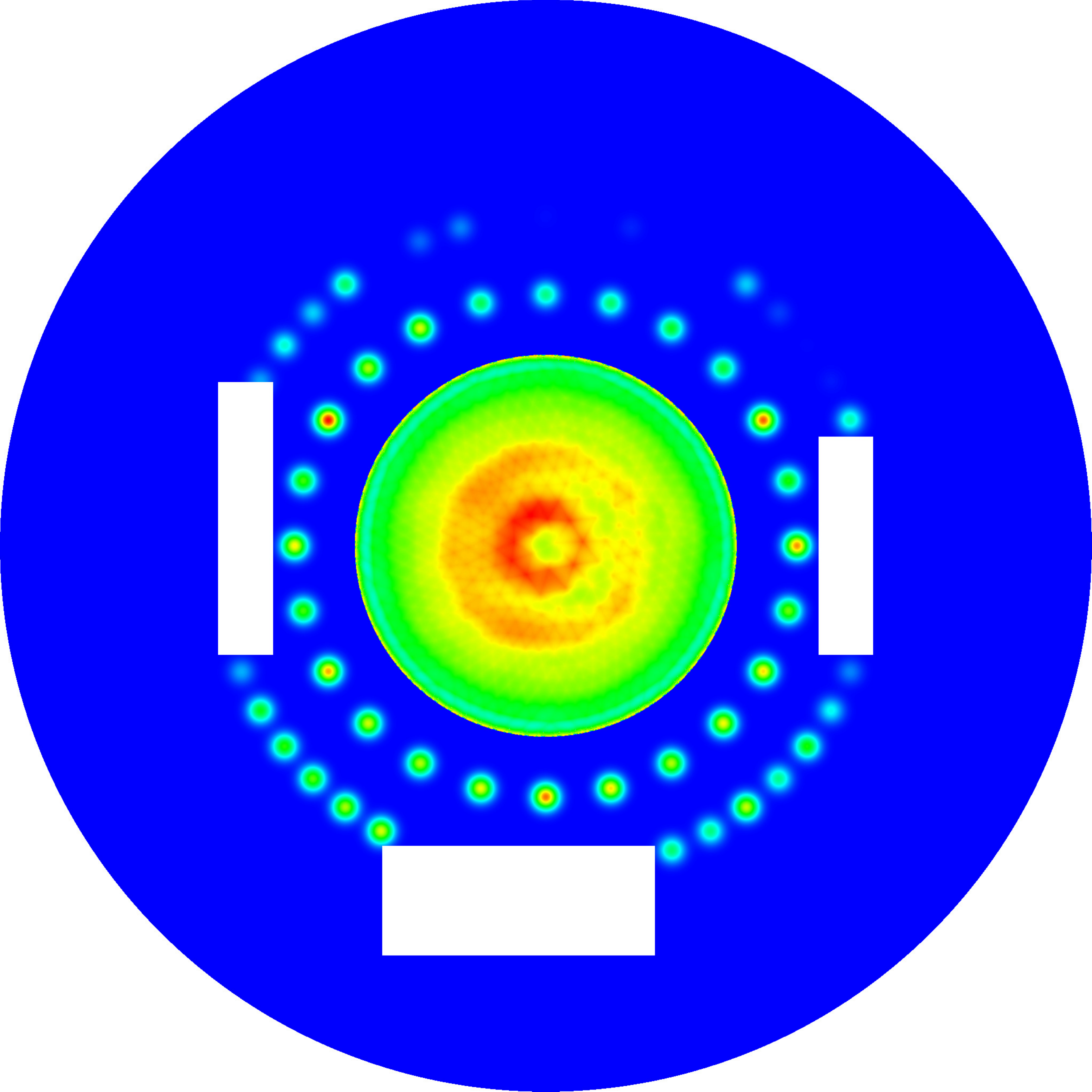}

\vspace{0.25ex} % Avoids caption eating the bottom of the circle

\caption{Globally optimal non-binary design $\w^*$, $\Jc(\w^*)\approx 0.04639$; smaller peaks denote non-binary sensors. Inner circle: Resulting pointwise variance field.}
\label{fig:global_optimum_24}

\vspace{0ex}

\begin{tikzpicture}
\node[coordinate] (0) at (0,-0.5) {};
\node[right=4.55cm of 0] (1) {$\cdot 10^{-2}$};
\node[left=1cm of 0] (2) {}; % Helps realign the colorbar

\pgfplotscolorbardrawstandalone[
	colormap/jet,    
    colorbar horizontal,
    point meta min=0,
    point meta max=\posteriortwentyfourdiffmax,
    colorbar style={
        width=0.3\textwidth,
        /pgf/number format/fixed,
        /pgf/number format/precision=2,
        xticklabel style={anchor=north},
        every axis label/.append style={/pgf/number format/precision=2},  % Ensure proper precision for the tick labels
        yticklabel style={/pgf/number format/none},
        xtick distance=\posteriortwentyfourdiffmax/4, 
        tick style={draw=none}}]
\end{tikzpicture}

\vspace{-2.6ex}

\includegraphics[width=\figsize,keepaspectratio]{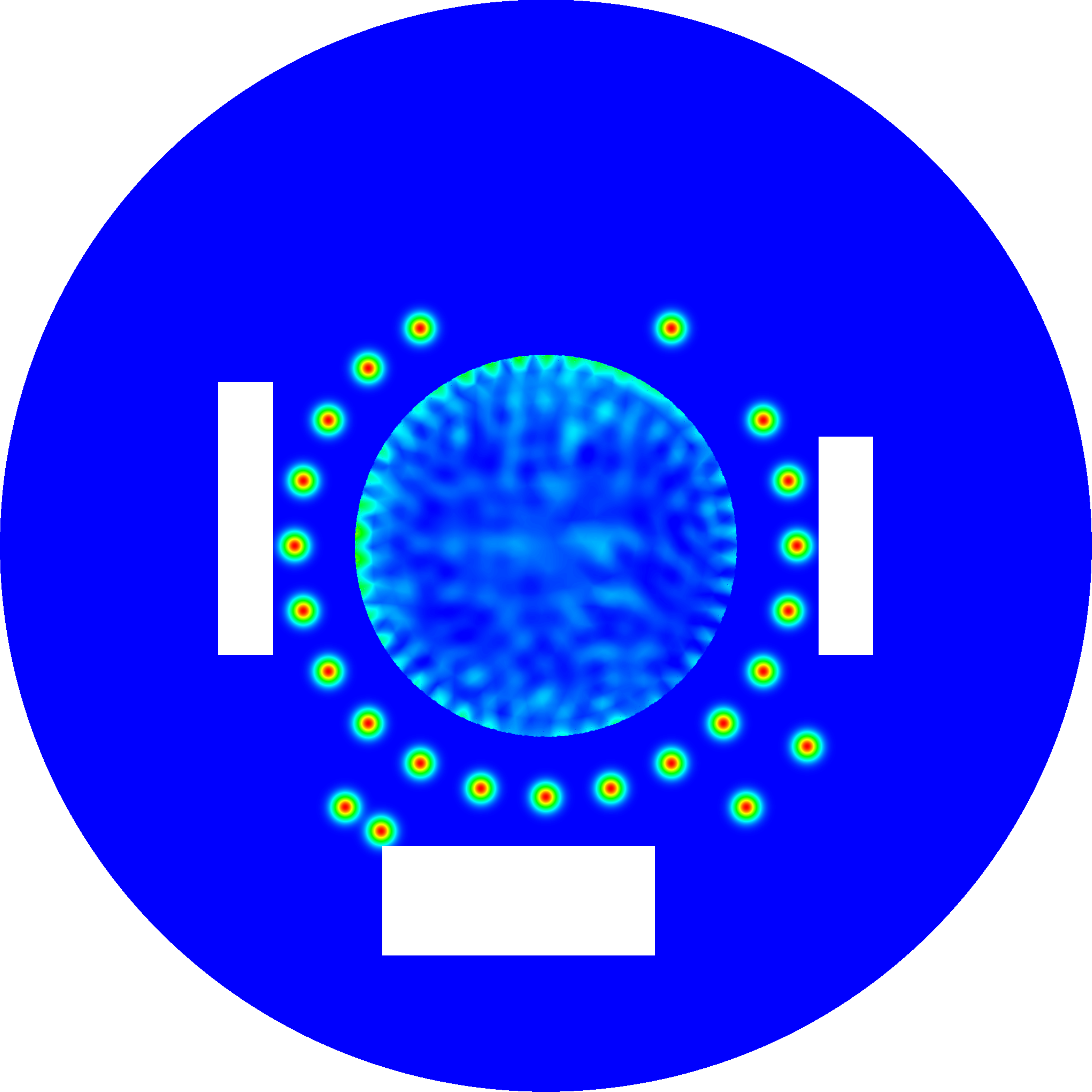} \hfill
\includegraphics[width=\figsize,keepaspectratio]{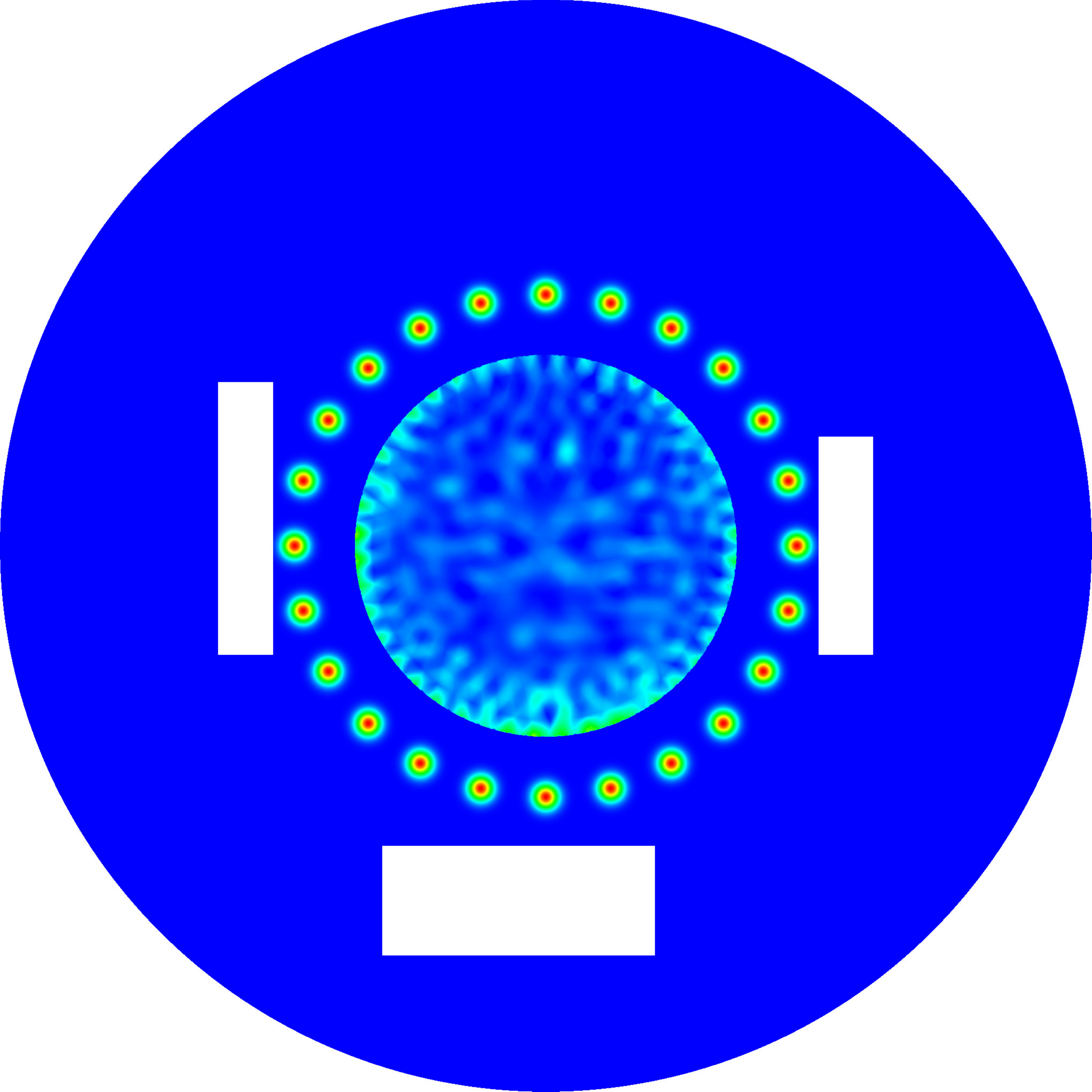} \hfill
\includegraphics[width=\figsize,keepaspectratio]{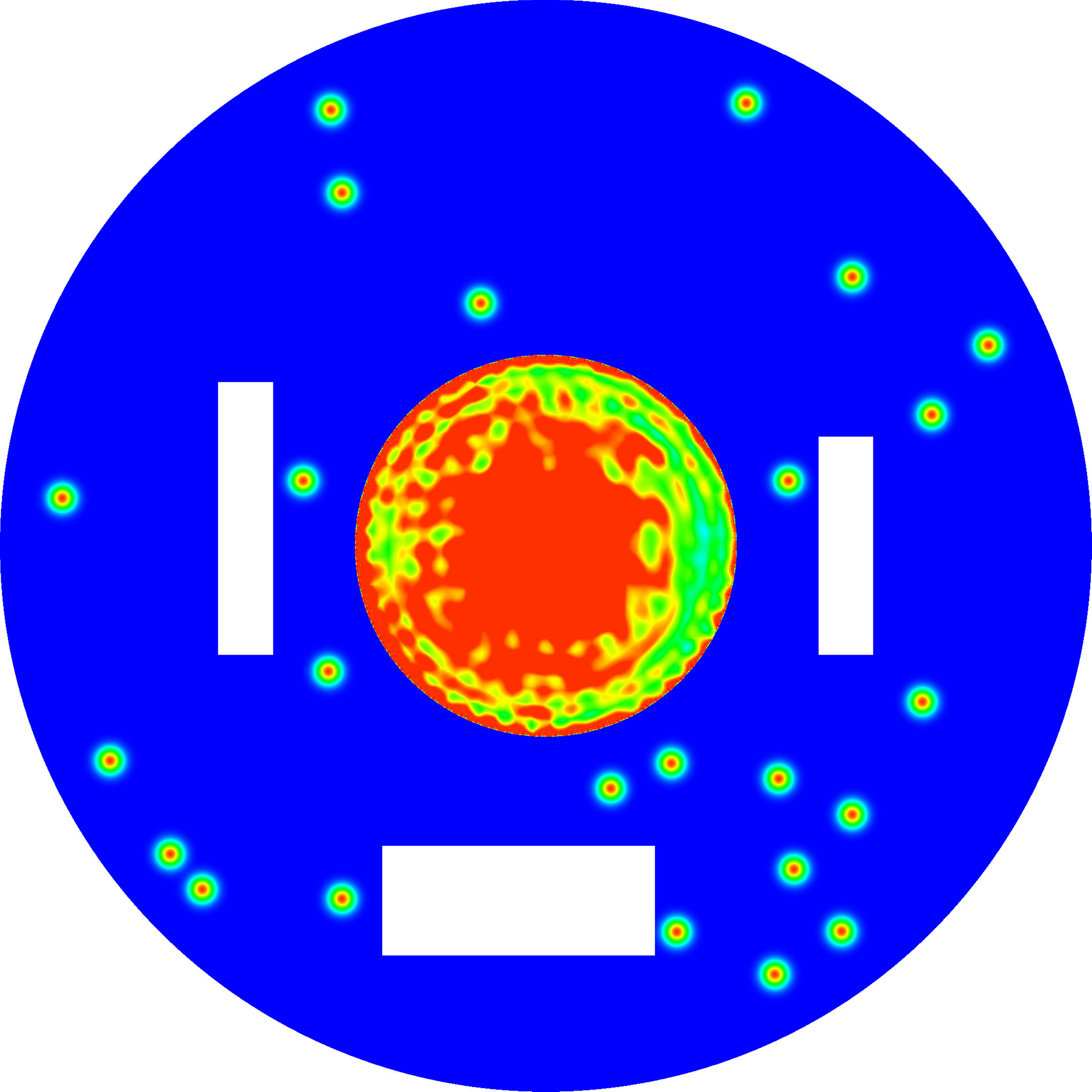}

\vspace{0.25ex}

\caption{$\wpcont[24]$ from \algref{alg:p_cont} (left), outperforming circular design $\w_{\mathsthis{circ}}$ (middle) and random design $\w_\mathsthis{rnd}$ (right). Inner: Difference between induced pointwise variance field and that of global optimum $\w^*$.}
\label{fig:variance_comparisons_24}

\vspace{0.5ex}

\pgfplotscolorbardrawstandalone[
	colormap/jet,    
    colorbar horizontal,
    point meta min=-1,
    point meta max=1,
    colorbar style={
        width=0.3\textwidth,
        tick style={draw=none}}]

\vspace{-0.2ex}

\includegraphics[width=\figsize,keepaspectratio]{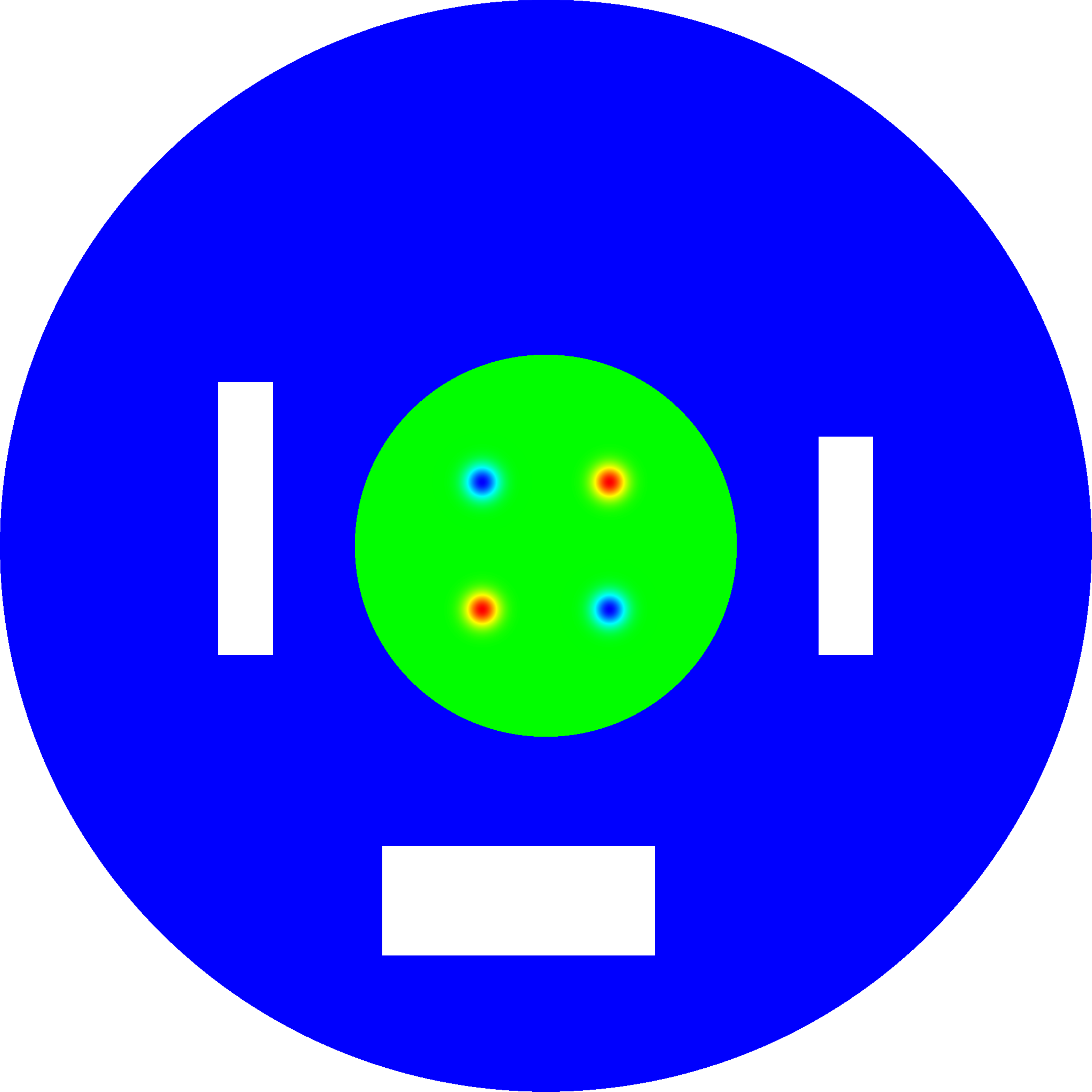}\hfill
\includegraphics[width=\figsize,keepaspectratio]{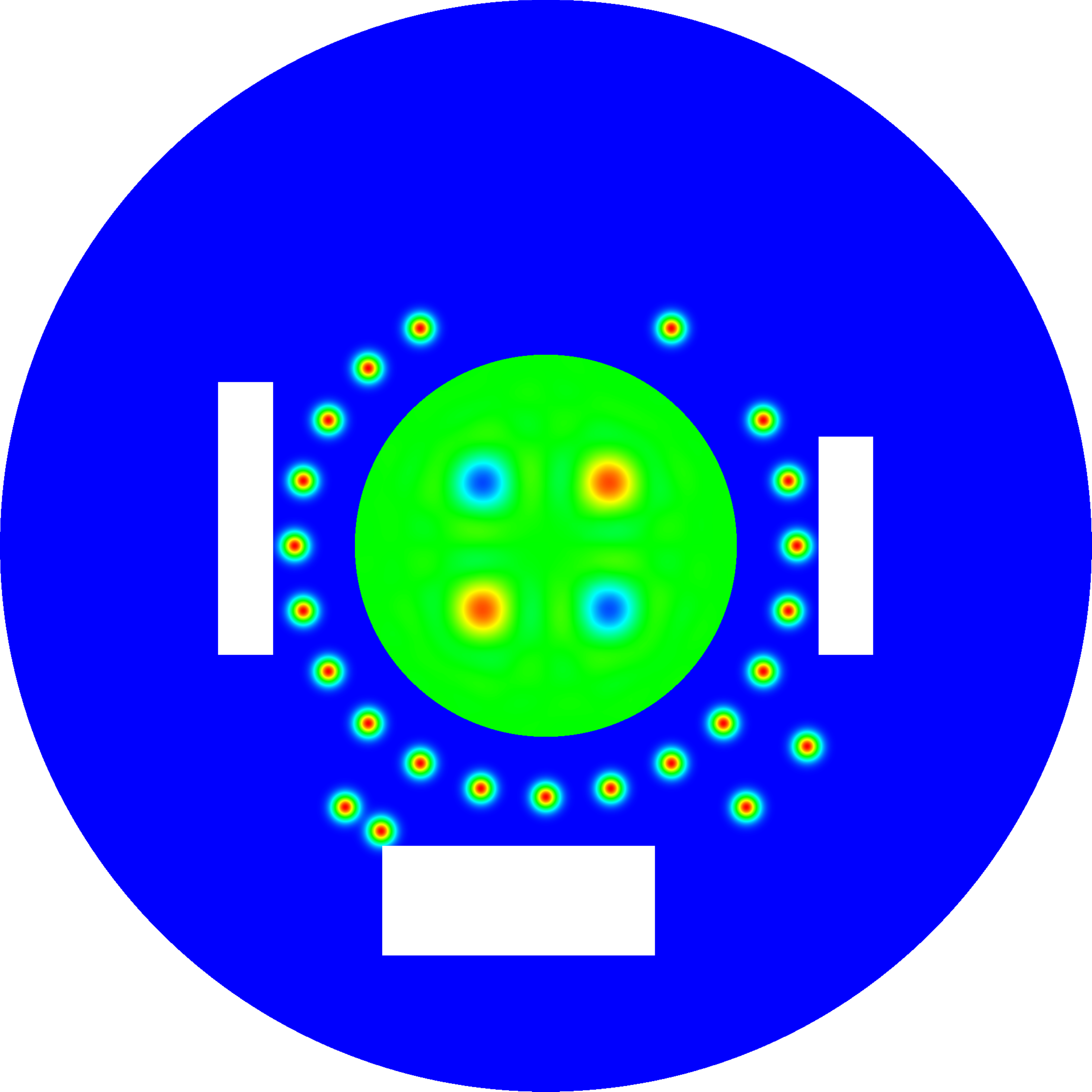}\hfill 
\includegraphics[width=\figsize,keepaspectratio]{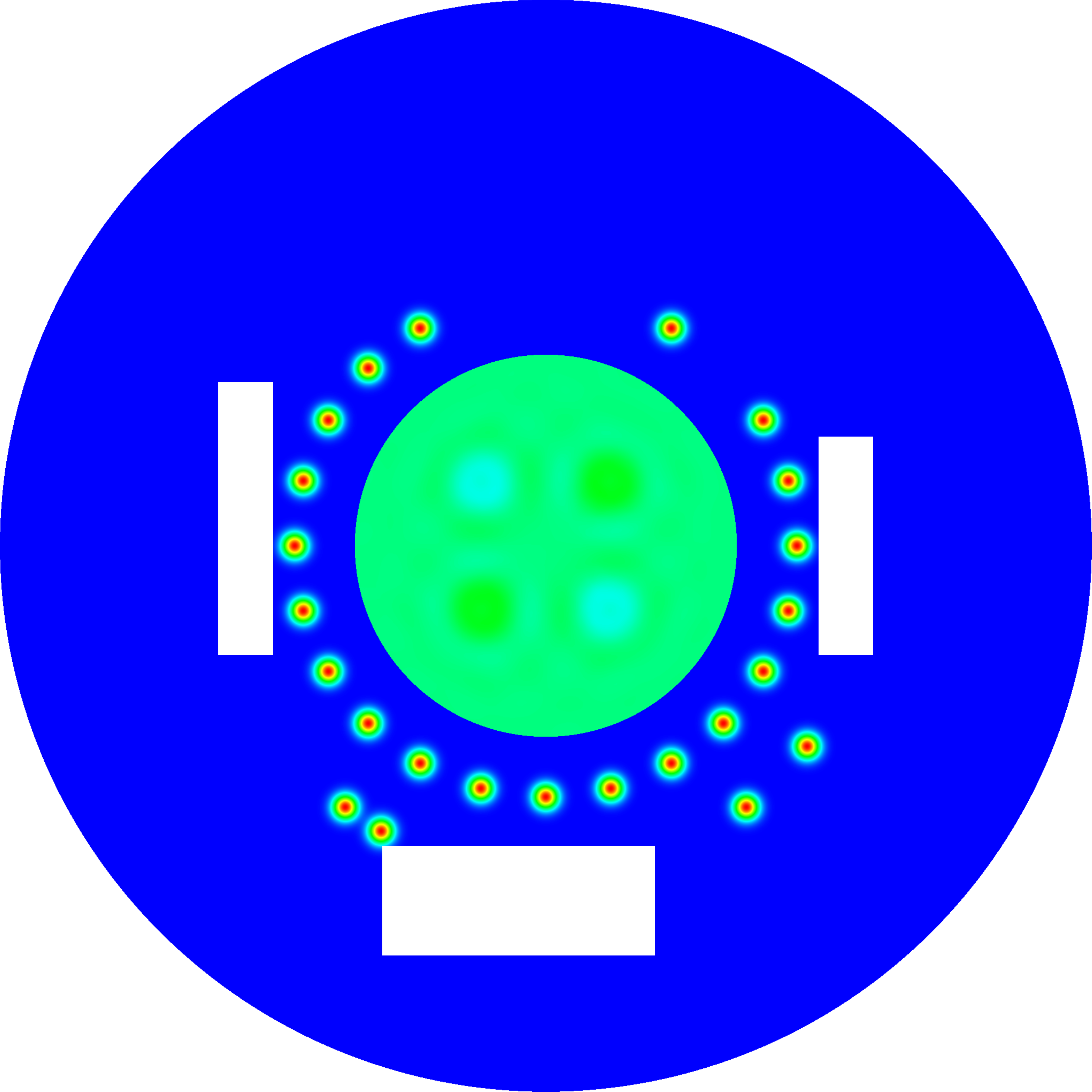}

\vspace{0.25ex}

\caption{Left: Source $f$. Middle: Posterior mean $\mpst(\wpcont[24])$. Right: Reconstruction error $\mpst(\wpcont[24])-f$.}
\label{fig:recos_24}
\end{center}
\end{figure}

\vfill

\pagebreak

\paragraph{Case study}

We present a detailed study of the case $m_0=24$. \figref{fig:global_optimum_24} visualises the non-binary global optimum $\w^*$, smaller dots illustrating sensor placements taking values in $(0,1)$, and draws the accompanying pointwise variance field inside the experimental domain. \figref{fig:variance_comparisons_24} compares this pointwise variance field with that induced by the output $\wpcont[24]$ of \algref{alg:p_cont} (left), a hand-crafted design $\w_\mathsthis{circ}$ with all sensors placed in a uniform circle around the center (middle) and with $\w_\mathsthis{rnd}$, being the best out of the $10^3$ random designs drawn for \figref{fig:comps}. The A-optimality of each design is, respectively, $\Jc(\w^*)\approx 0.04639$, $\Jc(\wpcont[24])\approx 0.04695$, $\Jc(\w_{\mathsthis{circ}})\approx 0.04702$ and $\Jc(\w_\mathsthis{rnd})\approx 0.05310$, that is, a successive increase (i.e.~worsening) by approximately $1.2\%$, $0.2\%$ and $12.9\%$ for each pair.

This comparison also illustrates how each design reduces the pointwise variance differently in different regions of the source domain $\Om$, yielding some insight into how the asymmetry of the domain $\M$ may affect the quality of the designs. 

From this, we also remark that the best random design contains a large number of redundant sensors, which may also be a contributing factor of the difference seen in \figref{fig:comps}; this could be seen as further evidence that encoding sensors as redundant or dominant and utilising this information can provide an immediate improvement to generic designs, a perspective the author investigates further in \cite{Aar24proceedings}.

\figref{fig:recos_24} displays the found experimental design $\wpcont[24]$, the posterior mean $\mpst(\wpcont[24])$ (middle) and accompanying reconstruction error, employing as an arbitrary example data $\g :=\Fc_{\wpcont[24]} f$ with source (left)
\begin{equation}\label{eq:f_ground}
	f(\x_1,\x_2):=\sum_{i=0}^3(-1)^i\exp(-800\|(\x_1-(-1)^{i+\delta_{i\geq 2}}r,\x_2-(-1)^{\delta_{i\leq 1}}r)\|^2), \qquad r:=0.35/3.
\end{equation}

\figref{fig:pseq} shows outputs of \algref{alg:p_cont}, ordered as given by the gradient of $\w^*$. Note how the algorithm exactly maintains the sum target of $24$, while gradually pushing towards a more binary design.

\begin{figure}[h]
\begin{center}
\pgfplotsset{scaled y ticks=false}
\begin{tikzpicture}
\begin{groupplot}[group style={
                group size=5 by 1,
                horizontal sep = 22pt,
            },
            yticklabel style={
        	/pgf/number format/fixed,
        	/pgf/number format/precision=2},
			scaled y ticks=false,
			title style = {
			yshift = -5pt
			},
			height = 6cm,
			width = 3.75cm
]

\pgfplotsforeachungrouped \x in {1,3,5,9,14}{
	\DTLfetchsave\pval{pseq}{step}{\x}{p}
	\ifnum\x=1
        	\def\xmax{334}
    		% If this is any subsequent plot, set x-axis range to 1:3
    	\else
        	\def\xmax{56}
    	\fi	
    \edef\tmp{
		\noexpand\nextgroupplot[title={$p = \pval$}]
        \noexpand\addplot[
        	restrict x to domain = 0:\xmax,
        	restrict y to domain = 0:1,
        	thick,
        	blue
        ] table[x=index,y={\x}, col sep=comma]{graphics/wseq_24.csv};
        }
    \tmp
    }
\end{groupplot}
\end{tikzpicture}
\captionof{figure}{$p$-relaxed designs for $m_0=24$ and decreasing values of $p$ (left to right). Ordering of each design $\w$ according to $\nabla\Jc(\w^*)$ and \tref{thm:optimality}. Redundant indices not shown in subsequent plots.}\label{fig:pseq}
\end{center}
\end{figure}
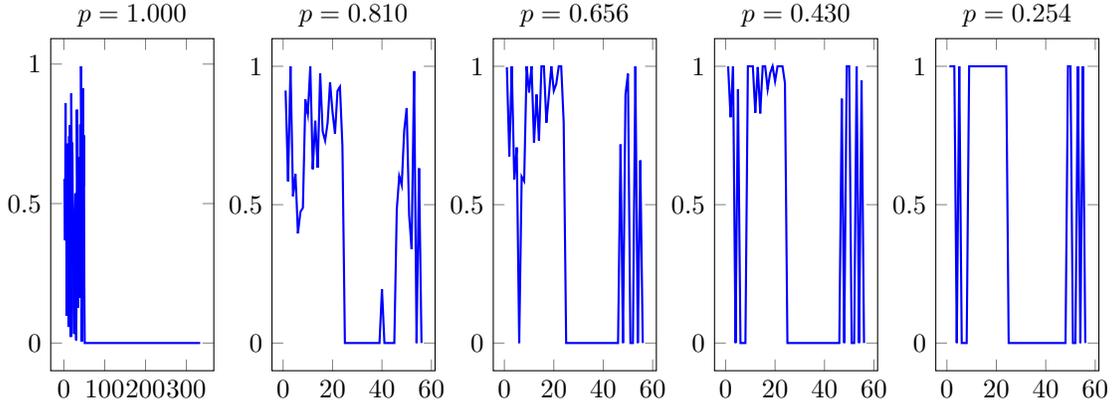

\section{Conclusion and outlook}

In this article, we have developed explicit optimality criteria for the best sensors placement problem in optimal experimental design, presented an algorithm that is suitable for a wide class of design criteria, and demonstrated its applicability to A-optimal experimental designs for infinite-dimensional Bayesian linear inverse problems. In doing so, we have contributed both a set of theoretical sufficient and necessary conditions for global optimality of experimental designs, and an efficient computational framework for evaluating the A-optimal objective and its derivatives, along with a powerful algorithm for approximating binary A-optimal designs.

We have not provided convergence guarantees for our algorithm, nor can we prescribe the level of $p$-relaxation needed to obtain a binary design. Verification of whether the found design truly is the globally optimal binary design similarly remains unfeasible. Thus, the deeper study of this algorithm, and in particular its connection to integer optimisation, remains a fascinating research question.

As is generally the case for A-optimal experimental designs for linear inverse problems, our analysis does not place significant constraints of the form of the linear parameter-to-observable map $\Fc:X\to\R^m$. With this in mind, and in order to draw full benefit from our efforts at reducing computational complexity of the OED itself, it is natural to next extend our numerical studies to more realistic settings, where $\Fc$ acts as a credible model for real-world experiments; efforts in this direction will require further studies of techniques to deal with extreme dimensionality and large data. 

\revia{Extensions to real-world experimental settings may also require treatment of the case of correlated measurement noise, i.e.~removing the assumption that the measurement noise covariance matrix $\Gmn\in\R^{m\times m}$ be diagonal. As discussed in \cite[Sec.~2.4, Sec.~3.1]{AttCon2022}, this necessitates significant alterations to the formulation of the posterior distribution \eqref{eq:bayesian_inversion}. Expanding our framework to this case would therefore be a significant improvement upon the state of the art.}\label{sentence:correlated_noise}

%beyond the fact that each successful iteration will have fixed at least one index as dominant or redundant. The deeper study of this algorithm, and in particular its connection to integer optimisation, remains a fascinating future research question. Additionally, the usage of the optimality criteria in \tref{thm:optimality} for existing optimisation schemes is a fascinating prospect. Combining these aspects to derive an improved version of \algref{alg:OED} is a particularly promising direction for further research.

Note that the assumption of convexity in $\Jc$ cannot easily be eschewed, as it is needed to apply the Fermat principle in \tref{thm:optimality}, allowing us to reliably obtain the non-binary global optimum $\w^*$ as a starting point for our continuation algorithm, and allowing us to identify redundant and dominant indices. Nevertheless, there is significant value in extending our results also to non-convex objectives; we postpone this investigation to a future work. Conversely, while large parts of the present article have relied on the linearity of the parameter-to-observable map $\Fc:X\to\R^m$, this is not a requirement for the application of \tref{thm:optimality} due to its general formulation; embedding our present results in a non-linear inverse problems setting is a fascinating prospect.

\paragraph{Acknowledgements}

The author wishes to express their gratitude to Prof.~Thorsten Hohage at the University of Göttingen, Germany and Prof.~Georg Stadler at the Courant Institute of Mathematical Sciences, USA for their advice and proofreading, and for many rewarding discussions. The author moreover acknowledges support from the DFG through Grant 432680300 -- SFB 1456 (C04).

\printbibliography

\appendix
\section{The role of the QR decomposition}\label{app:QR}

\paragraph{QR by randomized subspace iteration} The low-rank formulations in Theorems \ref{thm:low_rank_OED} and \ref{thm:multiple_OED} require access to efficient QR decompositions of the discretised composition $F:=\Gmnih\Fc\Cph\Mbhi\in\R^{m\times n}$. However, it is worth noting that most implementations of the QR algorithm, particularly those that provide low-rank approximations via e.g.~pivoting \cite{Cha87}, are not matrix-free. In cases where $m$ and $n$ are sufficiently modest, it may be feasible, as well as rather fast, to assemble $\Fb^T$ in memory and perform the QR decomposition directly; however, this is of limited use for many realistic and interesting scenarios, where both the discretisation dimension $n$ and the number of candidate sensor locations $m$ are very large.

With this in mind, we obtain our QR decomposition by means of the matrix-free SVD-based \algref{alg:rSVD}, based on the randomised subspace iteration algorithm \cite[Algs.~4.4 \& 5.1]{HalMarTro11}, remarking that this version allows non-symmetric input matrices, in contrast to the one highlighted in e.g.~\cite[Algorithm 1]{Ale21}. The below \algref{alg:rSVD} allows for efficient matrix-free construction of the QR decomposition, as it requires only evaluations of the forward and adjoint prior-preconditioned unknown-to-observable map $\Fb$, $\Fb^T$ as in \eqref{eq:A_optimal_FEM}. In the final singular value decomposition, one may deliberately truncate additionally by removing sufficiently small eigenvalues in step \ref{alg:rSVD:SVD}, enabling further flexibility than simply adhering to the previously fixed choice of $\ell$. 
 
\begin{algorithm}[h!]
\caption{QR via randomised subspace iteration}\label{alg:rSVD}
\begin{algorithmic}[1]
\Require Routine calculating forward and adjoint applications of a matrix $A\in\R^{n\times m}$ , target rank $\ell\in\N$, number of subspace iterations $q\in\N$, random standard Gaussian input matrix $O\in\R^{m\times\ell}$
\State $B_0:=AO\in\R^{n\times\ell}$, compute thin QR decomposition $Q_0R_0=B_0$
\For{$1\leq j\leq q$}
\State Assemble $\widetilde{B}_j:=A^TQ_{j-1}\in\R^{m\times\ell}$, compute thin QR decomposition $\widetilde{Q}_j\widetilde{R}_j=\widetilde{B}_j$
\State Assemble $B_j:=A\widetilde{Q}_{j}\in\R^{n\times\ell}$, compute thin QR decomposition $Q_jR_j=B_j$
\EndFor
\State Assemble $B:=Q_q^TA\in\R^{\ell\times m}$
\State Compute rank $\ell$ SVD decomposition $B=USV^T$, $U\in\R^{\ell\times\ell}$, $V\in\R^{m\times\ell}$ \label{alg:rSVD:SVD}
\State \revib{Assemble $\widetilde{R}:=SV^T\in\R^{\ell\times m}$}
\State \revib{Compute thin QR decomposition $\widetilde{Q}R=\widetilde{R}$}
\State \revib{Assemble $Q:=Q_qU\widetilde{Q}\in\R^{n\times\ell}$}
\Ensure Approximate thin QR decomposition $A\approx QR$.
\end{algorithmic}
\end{algorithm}

\paragraph{A discussion on sensor backpropagation}\label{par:sensor_backprop} 

While the QR decomposition appearing in \tref{thm:low_rank_OED} appears as merely a computational convenience, it has an interesting interpretation in terms of what one might call \emph{sensor backpropagation}. Indeed, consider the undiscretised inverse problem $\Fc f = \gw + \eps$ as in \eqref{eq:forward}. With (undiscretised) prior covariance $\Cp:\Lt\to\Lt$, the posterior covariance of $f$ given $\gw$ is exactly
\[
	\Cpst(\w) = \Cph
	\left(
		\Cph\Fc^*\Gmnih 
			\Mw 
		\Gmnih\Fc\Cph + I
	\right)^{-1}
	\Cph,
\]
retaining the assumption that $\Gmn\in\R^{m\times m}$ is diagonal. As $\Fc$ is bounded and linear, and as discussed in \sref{sec:numerics}, it is ubiquitous that it is of the form $\Fc = \Oc\circ \Sc$, where $\Sc\in L(X,Y)$ might be viewed as the \enquote{true forward operator} in the infinite-dimensional sense, with some Banach space $Y$, and $\Oc\in L(Y,\R^m)$ is the finite observation operator, e.g.~$\Sc$ being a PDE solution operator mapping the source $f$ to the PDE solution $u:=\Sc f\in Y$, while $\Oc u\in\R^m$ represents discrete observations of the infinite-dimensional function $u$.

By the definition of the dual, it is clear that $\Oc = (\dual[Y]{\oc_k}{\cdot})_{k=1}^m$ for some uniquely determined family $\left(\oc_k\right)_{k=1}^m\subset Y^*$; one may think of each $\oc_k\in Y^*$ as a \emph{sensor}, mapping the true data $u=\Sc f\in Y$ to a scalar observable $\dual[Y]{\oc_k}{u}\in\R$. The space $Y^*$ of valid sensors is therefore much more diverse than simply encompassing the pointwise observation functionals $o_k:=\delta_{\x_k}$ considered in \sref{sec:numerics}; in fact, even this form of observation by pointwise measurement demands a certain level of smoothness, i.e.~$Y\supseteq H^s(\Om)$, typically with $s>d/2$, $\Om\subset\R^d$, $d\in\N$, which immediately suggests constraints on the smoothing effect of the PDE solution operator $\Sc :X\to Y$. When less smoothness is available in $Y$, one must necessarily compensate by higher smoothness in $Y^*$, i.e.~in the space of possible sensors.

As in \sref{sec:numerics}, it remains immediate that $\Oc^*:\R^m\to Y^*$ is the map $\Oc^*\g = \sum_{k=1}^m\g_k\oc_k$ for all $\g\in \R^m$. Neglecting for simplicity the data noise covariance $\Gmn$, it follows that
\[
	\Cph\Fc^* \g = \Cph \Sc^*\Oc^* \g =
	\Cph \Sc^*\left[
		\sum_{k=1}^m\g_k\oc_k
	\right]
	= \sum_{k=1}^m\g_k[\Cph \Sc^*\oc_k]\in\Lt
\]
for all $g\in\R^m$. Thus, the range of $\Cph\Fc^*:\R^m\to\Lt$ is precisely the at most $m$-dimensional span of the backpropagated sensors $(\Cph \Sc^*\oc_k)_{k=1}^m\subset\Lt$. By Gram-Schmidt orthogonalisation in the infinite-dimensional Hilbert space $\Lt$, whose first few steps prior to normalisation can be expressed as
\begin{align*}
q_1 & := \oh_1, & & \oh_1 := \Cph \Sc^*\oc_1, \\
q_2 & := \oh_2 - \inner[\Lt]{\oh_2}{\oh_1}\oh_1, & & \oh_2 := \Cph \Sc^*\oc_2, \\
q_3 & := \oh_3 - \inner[\Lt]{\oh_3}{\oh_2}\oh_2 - \inner[\Lt]{\oh_3}{\oh_1}\oh_1, & & \oh_3 := \Cph \Sc^*\oc_3, \\
& \hspace{-5pt}\vdots & & \hspace{3pt}\vdots
\end{align*} 
there exists $\ell\in\N$, $\ell\leq m$ and a normalised subset $(q_k)_{k=1}^\ell\subset\Lt$ such that
\[
	\Rc(\Cph\Fc^*) = \spann\left\{\Cph \Sc^*\oc_k\right\}_{k=1}^m =
	\spann\left\{q_l\right\}_{l=1}^\ell,
\]
together with scalar coefficients $(R_{lk})_{1\leq l\leq \ell, 1\leq m\leq k}$ so that 
\[
	\sum_{k=1}^mR_{lk}[\Cph \Sc^*\oc_k] = q_l
\]
for all $l\in\N$, $l\leq\ell$.

Defining the operator $Q:\R^\ell\to\Lt$ by $Q\a:=\sum_{l=1}^\ell \a_lq_l\in\Lt$ for $\a\in\R^\ell$, and considering $R\in\R^{\ell\times m}$ as a matrix, it follows that one has precisely $\Cph\Fc^* = QR:\R^m\to\Lt$, with $Q^*Q=I_\ell$, that is, this representation suggests an undiscretised, exact equivalent to the formulation of Theorems \ref{thm:low_rank_OED} and \ref{thm:multiple_OED}.

The \enquote{columns} $q_l$ of $Q$ now encode the maximum amount of linearly independent information we can obtain from backpropagating observed data; truncation in the QR decomposition corresponds to discarding less informative basis elements. In particular, it becomes evident from the above and from \tref{thm:low_rank_OED} and \tref{thm:multiple_OED} that while the underlying PDE may be infinite-dimensional, the expression
\[
	\Jc(\w) = \tr(\Cp) - \tr(C) + \tr\left(
		\left(
			R\Mw R^T + I_\ell
		\right)^{-1}C
	\right), \qquad C := Q^* \Cp Q \in \R^{\ell\times \ell}
\]
for the A-optimal objective in \eqref{eq:low_rank_OED:objective} remains a valid representation of the undiscretised, exact A-optimal objective purely in terms of the $\ell$-dimensional matrices $R\Mw R^T$, $C\in\R^{\ell\times\ell}$, and similarly for the accompanying derivatives \eqref{eq:low_rank_OED:jacobian}--\eqref{eq:low_rank_OED:hessian}. This fascinating perspective suggests exciting directions for further research on undiscretised optimal experimental design, and, due to the relationship $Q$ and $R$ have with the observation operator $\Oc$, on the dependence of the optimal experimental design on the exact form of the sensors $(o_k)_{k=1}^m\subset Y^*$.

\iffalse

, in a certain sense, effectively \emph{all} information required to solve the optimal experimental design problem can be extracted from the symmetric, finite-dimensional $m\times m$ matrix
\[
	\begin{bmatrix}
		\inner[\Lt]{\oh_1}{\oh_1} & \inner[\Lt]{\oh_1}{\oh_2} & \hdots & \inner[\Lt]{\oh_1}{\oh_m} \\
		\vdots & \inner[\Lt]{\oh_2}{\oh_2} & \hdots & \inner[\Lt]{\oh_2}{\oh_m}\\
		\vdots & \vdots & \ddots & \vdots \\
		\hdots & \hdots & \hdots & \inner[\Lt]{\oh_m}{\oh_m} \\
		
	\end{bmatrix}
\]
and the $\ell\times m$ matrix $R$, with no loss of information. 

\fi

\end{document}